\documentclass[12pt]{amsart}
\usepackage{tikz}
\usetikzlibrary{calc,intersections}
\usepackage{mathrsfs, amssymb}
\usepackage{hyperref}
\usepackage{mathrsfs, amssymb, array}
\usepackage{verbatim}
\usepackage{tabularx}
\usepackage{mathcomp}
\usepackage{upgreek}
\usepackage[matrix,arrow,curve]{xy}
\usepackage{longtable}
\usepackage{amsthm}
\makeatletter
\@addtoreset{equation}{subsection}
\makeatother

\newcommand{\KK}{\mathbb{K}}

\newcommand{\CC}{\mathbb{C}}
\newcommand{\ZZ}{\mathbb{Z}}
\newcommand{\PP}{\mathbb{P}}
\newcommand{\QQ}{\mathbb{Q}}
\newcommand{\FF}{\mathbb{F}}
\newcommand{\RR}{\mathbb{R}}

\newcommand{\EEE}{\mathscr{E}}

\newcommand{\MMM}{{\mathscr{M}}}
\newcommand{\FFF}{{\mathscr{F}}}
\newcommand{\OOO}{{\mathscr{O}}}

\newcommand{\HHH}{{\mathscr{H}}} 
\newcommand{\LLL}{{\mathscr{L}}} 
\newcommand{\NNN}{{\mathscr{N}}}

\newcommand{\rk}{\operatorname{rk}}
\newcommand{\corank}{\operatorname{corank}}
\newcommand{\Sing}{\operatorname{Sing}}
\newcommand{\Cl}{\operatorname{Cl}}
\newcommand{\id}{\operatorname{id}}
\newcommand{\etr}{\operatorname{et}}
\newcommand{\Nm}{\operatorname{Nm}}
\newcommand{\hh}{\mathrm{h}}

\newcommand{\tors}{\mathrm{tors}}
\newcommand{\Pic }{\operatorname{Pic}}

\newcommand{\Bs}{\operatorname{Bs}}

\newcommand{\pt}{\operatorname{pt}}

\newcommand{\wt}{\operatorname{wt}}
\newcommand{\qq}{\mathbin{\sim_{\scriptscriptstyle{\QQ}}}}
\newcommand{\Br}{\operatorname{Br}}
\newcommand{\Hom}{\operatorname{Hom}}
\newcommand{\HHom}{\mathscr{H}{om}}
\newcommand{\red}{\operatorname{red}}
\newcommand{\Prym}{\operatorname{Pr}}
\newcommand{\Spec}{\operatorname{Spec}}
\newcommand{\J}{{\mathrm{J}}}
\newcommand{\JG}{{\mathrm{J}_{\mathrm{G}}}}

\newcommand{\g}{{\operatorname{g}}}
\newcommand{\mult}{{\operatorname{mult}}}

\newcommand{\bb}{{\operatorname{b}}}
\newcommand{\Supp}{{\operatorname{Supp}}}
\newcommand{\image}{{\operatorname{image}}}
\newcommand{\dd}{{\operatorname{d}}}
\newcommand{\codim}{{\operatorname{codim}}}
\newcommand{\p}{{\operatorname{p}_{\operatorname{a}}}}
\newcommand{\pg}{{\operatorname{p}_{\operatorname{g}}}}

\newcommand{\mumu}{{\boldsymbol{\mu}}}

\newcommand{\type}[1]{$\mathrm{#1}$}
\newcommand{\typem}[1]{$\mathbf{#1}$}
\newcommand{\muu}{{\boldsymbol{\mu}}}
\newcommand{\ii}{\operatorname{i}}
\newcommand{\ct}{\operatorname{c}}
\newcommand{\e}{\operatorname{e}}
\newcommand{\et}{\mathrm{\acute et}}
\newcommand{\NE}{\overline{\operatorname{NE}}}
\renewcommand{\emptyset}{\varnothing}

\usepackage[OT2,T1]{fontenc}
\DeclareSymbolFont{cyrletters}{OT2}{wncyr}{m}{n}
\DeclareMathSymbol{\Sha}{\mathalpha}{cyrletters}{"58}
\newcommand{\dif}{\operatorname{d}_{\Sha}}

\newcommand{\comp}\circ
\newcommand{\xref}[1]{\textup{\ref{#1}}}

\renewcommand\labelenumi{\rm (\roman{enumi})}
\renewcommand\theenumi{\rm (\roman{enumi})}

\swapnumbers
\theoremstyle{plain}
\newtheorem{theorem}[subsection]{Theorem}
\newtheorem{lemma}[subsection]{Lemma}

\newtheorem{proposition}[subsection]{Proposition}
\newtheorem{stheorem}[equation]{Theorem}
\newtheorem{sobservation}[equation]{Observation}
\newtheorem{corollary}[subsection]{Corollary}
\newtheorem{scorollary}[equation]{Corollary}
\newtheorem*{claim*}{Claim}

\newtheorem{slemma}[equation]{Lemma}
\newtheorem{sproposition}[equation]{Proposition}
\newtheorem{conjecture}[subsection]{Conjecture}
\newtheorem{sconjecture}[equation]{Conjecture}
\newtheorem{emptytheorem}[equation]{}

\theoremstyle{definition}
\newtheorem{definition}[subsection]{Definition}

\newtheorem{squestion}[equation]{Question}
\newtheorem*{definition*}{Definition}
\newtheorem{sdefinition}[equation]{Definition}
\newtheorem{example-remark}[subsection]{Remark-Example}
\newtheorem{subexample-remark}[equation]{Remark-Example}

\newtheorem{scase}[equation]{}

\newtheorem*{notation*}{Notation}
\newtheorem{example}[subsection]{Example}

\newtheorem{subexample}[equation]{Example}

\newtheorem{remark}[subsection]{Remark}

\newtheorem{sremark}[equation]{Remark}
\newtheorem{construction}[subsection]{Construction}

\newcounter{NN}

\newcounter{NO}

\begin{document}
\title{The rationality problem \\ for conic bundles
}
\author{Yuri Prokhorov}
\thanks{
This work  was partially supported 
by the Russian Academic Excellence Project ``5-100''.
The paper was written while the author was visiting the Max Planck Institute for Mathematics (Bonn). 
He would like to thank the institute for the invitation and excellent working condition. 
}
\thanks{Submitted to Uspeshi Mat. Nauk = Russian Math. Surveys}
 \address{
Steklov Mathematical Institute of Russian Academy of Sciences, Moscow, Russia
 \newline\indent
 National Research University Higher School of Economics, Moscow, Russia
}
\email{prokhoro@mi.ras.ru}

\begin{abstract}
This expository paper is concerned with the rationality problems for
three-dimensional algebraic varieties 
with a conic bundle structure. We discuss the main 
methods of this theory. 
We sketch the proofs of certain 
principal results, and present some recent achievements. Many open problems 
are also stated. 
\end{abstract} 
\maketitle

\tableofcontents

\section{Introduction}
\label{section:intro}
We basically work over the field $\CC$ of complex numbers.
In this paper we deal with algebraic varieties having a 
structure of conic bundle over surfaces.
A motivation for the study above is that conic bundles occur in the birational classification
of threefolds of negative Kodaira dimension. 

According to the minimal model program \cite{Mori-1988}, \cite{BCHM}
every uniruled algebraic projective variety $Y$ 
is birationally equivalent to a 
projective variety $X$ with at most $\QQ$-factorial terminal singularities 
that admits a contraction $\pi: X\to S$ 
to a lower-dimensional normal projective variety $S$ such that the anticanonical divisor 
$-K_X$ is $\pi$-ample and 
\[
\Pic(X)=\pi^*\Pic(S)\oplus \ZZ.
\]
\subsection{}
\label{3-dim-MMP-cases}
In dimension $3$ there are the following three possibilities: 
\begin{itemize}
\item 
$S$ is a point and then $X$ is called a \textit{$\QQ$-Fano threefold}; 
\item 
$S$ 
is a smooth projective curve and then $\pi: X\to S$ is called a \textit{$\QQ$-del Pezzo fibration};
\item 
$S$ 
is a normal surface and then $\pi: X\to S$ is called a \textit{$\QQ$-conic bundle}.
\end{itemize}

$\QQ$-Fano threefolds are bounded \cite{Kawamata-1992bF} (this is true even in arbitrary dimension \cite{Birkar2016}).
However, an explicit classification is known only in three-dimensional case 
for smooth ones \cite{IP99}.
There are a lot of partial results related to singular $\QQ$-Fano threefolds,
see e.g. \cite{Alexeev-1994ge}, \cite{GRD}, \cite{Prokhorov-Reid}, \cite{Prokhorov-planes}, \cite{Prokhorov-e-QFano7}, 
\cite{Prokhorov-factorial-Fano-e}, and references therein.

For $\QQ$-del Pezzo fibrations there are partial results on construction of 
\textit{standard models} \cite{Corti1996}, \cite{Kollar1997a} and 
\textit{birational rigidity}, see e.g. \cite{Pukhlikov1998a}, \cite[Ch.~4-5]{Pukhlikovbook2013}, \cite{Cheltsov2005c},
\cite{Sobolev2002}, \cite{Grinenko2000}, \cite{Shokurov-Choi-2011}. 

In this paper we concentrate on the last case of~\xref{3-dim-MMP-cases}. 
We are mainly interested in rationality questions. 
It is known that any $\QQ$-conic bundle has a 
\textit{standard model}, that is, a $\QQ$-conic bundle $\pi^\bullet: X^\bullet\to S^\bullet$ such that 
the total space $X^\bullet$ and the surface $S^\bullet$ are smooth and there are 
birational maps $X^\bullet\dashrightarrow X$ and $S^\bullet\dashrightarrow S$ making the corresponding diagram commutative
(see Theorem~\xref{theorem-standard-models}). The following conjecture is motivated by 
rationality criterion for two-dimensional conic bundles (see Sect.~\xref{section:surfaces}).

\begin{conjecture}[\cite{Shokurov1983}]
\label{conjecture-Shokurov}
Let $\pi: X\to S$ be a standard conic bundle over a rational 
surface $S$ with discriminant curve $\Delta\subset S$. In this case, $X$ is rational if and
only if the following holds:
\begin{enumerate}
\renewcommand\labelenumi{$(\star)$}
\renewcommand\theenumi{$(\star)$}
\item \label{conjecture-Shokurov:2KS+D}
$|2K_S+\Delta|=\emptyset$ 
and, in the case $\p(\Delta)=6$, the Griffiths component $\JG(X)$
of the intermediate Jacobian $\J(X)$ is trivial. 
\end{enumerate}
\end{conjecture}

Note that in the case $|2K_S+\Delta|=\emptyset$ and $\p(\Delta)=6$ the variety $X$ can be non-rational
(i.e. the condition $\JG(X)=0$ here cannot be removed).
In fact, in this case $\pi: X\to S$ is fiberwise birational 
to a standard conic bundle $\pi^\sharp: X^\sharp\to \PP^2$
with discriminant curve $\Delta^\sharp$ of degree $5$ so that the corresponding double cover $\tilde \Delta^\sharp \to \Delta^\sharp$
is defined by an odd theta characteristic (see Corollary
\xref{corollary-transformations:2K+S=0}).
Moreover, $X$ is birational to a three-dimensional cubic hypersurface (see Proposition~\xref{proposition-tregub}).

\begin{conjecture}[\cite{Iskovskikh-1987}]
\label{conjecture-Iskovskikh}
Let $\pi: X\to S$ be a standard conic bundle over a rational surface
$S$ with discriminant curve $\Delta\subset S$. In this case, $X$ is rational if and
only if one of the following conditions is satisfied.
\begin{enumerate} 
\item 
\label{conjecture-Iskovskikh-1}
There exists a commutative diagram
\begin{equation}\label{diagram-Conjecture-Iskovskikh-1}
\vcenter{
\xymatrix{
X'\ar[d]^{\pi'}\ar@{-->}[r]^{\Phi}&X\ar[d]^{\pi}
\\
S'\ar[r]^{\alpha}&S
}}
\end{equation}
where $\pi': X'\to S'$ is a standard conic bundle with a discriminant curve $\Delta'\subset S'$,
$\alpha$ is a birational morphism, and $\Phi$ is a birational map,
and a base point free pencil of
rational curves $\LLL'$ on $S'$ such that
$\LLL'\cdot\Delta'\le 3$.

\item 
\label{conjecture-Iskovskikh-2}
There exists a commutative diagram
\begin{equation}\label{diagram-Conjecture-Iskovskikh-2}
\vcenter{
\xymatrix{
X\ar[d]^{\pi}\ar@{-->}[r]^{\Phi}&X'\ar[d]^{\pi'}
\\
S\ar[r]^{\alpha}&\PP^2
}}
\end{equation}
where $\pi': X'\to S'=\PP^2$ is a standard conic bundle with  discriminant curve $\Delta'\subset\PP^2$
of degree $5$,
$\alpha$ is a birational morphism, $\Phi$ is a birational map, and
$X'$ is a blow-up of $\PP^3$ along a non-singular 
curve $\Gamma\subset \PP^3$ of genus $5$
and degree $7$ \textup(see Example~\xref{example-Panin}\textup).
\end{enumerate}
\end{conjecture}
It is not so difficult to show that Conjectures~\xref{conjecture-Iskovskikh}
and~\xref{conjecture-Shokurov} are equivalent, see Corollary~\xref{corollary-equivalence-conjectures}.
The sufficiency of~\xref{conjecture-Iskovskikh-1} and~\xref{conjecture-Iskovskikh-2} in 
Conjecture~\xref{conjecture-Iskovskikh} and~\xref{conjecture-Shokurov:2KS+D} in
Conjecture~\xref{conjecture-Shokurov} is also known (see Propositions~\xref{proposition-Panin} and 
\xref{proposition-rationality-conic-bundles} and Corollary~\xref{corollary-transformations:2K+S=0}).
The hardest part of these conjectures is the necessity. 
The necessity of~\xref{conjecture-Iskovskikh}\xref{conjecture-Iskovskikh-1}--\xref{conjecture-Iskovskikh-2} 
and~\xref{conjecture-Shokurov}\xref{conjecture-Shokurov:2KS+D} was proved by 
Shokurov in the case of conic bundles over minimal rational surfaces (see Sect.~\xref{section:min-surface}).

\begin{remark}
Note that the condition~\xref{conjecture-Iskovskikh-1} of 
Conjecture~\xref{conjecture-Iskovskikh} can be replaced with the following
(see Proposition~\xref{proposition-transformations:2K+S=0}):
\begin{enumerate}
\renewcommand\labelenumi{\rm (\roman{enumi}${}'$)}
\renewcommand\theenumi{\rm (\roman{enumi}${}'$)}
 \item 
$\pi$ is fiberwise birationally equivalent to a standard conic bundle 
$\pi^\sharp: X^\sharp \to S^\sharp$, where $S^\sharp=\FF_n$
is a rational geometrically ruled surface and the discriminant curve $\Delta^\sharp$
meets a general fiber in at most three points.
\end{enumerate}
\end{remark}

The paper is organized as follows. Sections~\xref{section:Preliminaries} 
and~\xref{section:cb} are preliminary. There we fix the notation, give basic 
definitions, and collect general facts about varieties with conic bundle 
structures. In Section~\xref{section:Sarkisov} we introduce the Sarkisov program 
of decomposition of birational maps between Mori fiber spaces. The proof of the 
main theorem in the three-dimensional case is outlined. Section 
\xref{section:surfaces} is about surface conic bundles over algebraically 
non-closed fields. There we recall the classification of two-dimensional 
Sarkisov links and formulate rationality criterion. This theory is a motivation 
for the main conjectures~\xref{conjecture-Shokurov} and 
~\xref{conjecture-Iskovskikh}, as well as, one of the main tools in the proofs. 

There are two important birational invariants of algebraic varieties that use 
transcendental and topological methods: torsions in the middle cohomology groups 
and the intermediate Jacobian. These notions are discussed in Sections 
\xref{section:AM} and~\xref{section:J}, respectively. In Section 
\xref{section:bir} we collect examples of some special Sarkisov links in the category of 
conic bundles. They are used very frequently below. As the first application of 
the developed theory, following Shokurov \cite{Shokurov1983}, in Section 
\xref{section:min-surface} we reproduce the proof of 
Conjecture~\xref{conjecture-Shokurov} for conic bundles over \textit{minimal} 
rational surfaces.

Although Conjectures~\xref{conjecture-Shokurov} and 
~\xref{conjecture-Iskovskikh} are formulated in terms of non-singular $X$ and 
$S$, in order to analyze the corresponding birational maps we use the techniques 
on the factorization of birational maps between Mori fiber spaces, which, in 
general, admit terminal singularities. These techniques and results are 
discussed in Sections~\xref{section:Q-conic-bundles}--\xref{section:II}. Thus, 
Section~\xref{section:Q-conic-bundles} covers known results on local 
classification of $\QQ$-conic bundles and Section~\xref{section:ex} contains 
numerous examples of Sarkisov links between them. In Sections~\xref{section:I} 
and~\xref{section:II} we apply developed theory to rationality problems of 
conic bundles. In particular, prove Sarkisov Theorem~\xref{Sarkisov-theorem} and 
Theorem~\xref{Iskovskikh-theorem} which is a very weak version of 
Conjecture~\xref{conjecture-Iskovskikh}. We also prove the equivalence of 
Conjecture~\xref{conjecture-Iskovskikh} and some classical conjecture of 
projective geometry (see Proposition~\xref{proposition-equi-Kantor}). In the last 
section~\xref{section:problems} we collect open problems and some results related 
to birational geometry of conic bundles.

\textbf{Acknowledgements.}
The author is grateful to Artem Avilov, Arnaud Beauville, Ivan Cheltsov, Alexander Kuznetsov, Boris Kunyavskii,
Constantin Shramov, and Vyacheslav Shokurov for useful advises and discussions, and also 
to the referees for careful reading the manuscript and
numerous corrections of inaccuracies.

\section{Preliminaries}
\label{section:Preliminaries}
Everywhere in this paper, if we do not specify anything else, we work 
over an algebraically closed field $\Bbbk$ of characteristic $0$. 
In certain situations we prefer to work over $\CC$. 
We use the standard terminology and notation of the (Log) Minimal Model Program 
\cite{Kollar-Mori-1988}. For an introduction to rationality problems 
we refer to \cite{Kollar2004}.
\subsection{Notation.}
Throughout this paper
\begin{itemize}
\item[] 
$\equiv$ denotes the numerical equivalence of cycles, 
\item[] 
$K_X$ is the canonical (Weil)
divisor of a normal variety $X$, 
\item[] 
$\FF_n=\PP_{\PP^1}\left(\OOO_{\PP^1}\oplus\OOO_{\PP^1}(n)\right)$ is the rational ruled (Hirzebruch) surface, 
\item[] 
$\uprho(X)=\rk \Pic(X)$ is the Picard number of $X$,
\item[] 
$\uprho(X/S)$ is the relative Picard number.
\end{itemize}

A \textit{contraction} is a surjective projective morphism $\pi: X\to S$ 
of normal varieties such that $\pi_*\OOO_X=\OOO_S$.
A contraction $\pi: X\to S$ is \textit{extremal} if $\uprho(X/S)=1$.

A contraction is 
said to be a \textit{Mori extremal contraction} if $X$ is a $\QQ$-factorial variety with at most terminal
singularities, the $\QQ$-Cartier divisor $- K_X $ is relatively ample, and the
relative Picard number is $\uprho(X/S)=1$. 
If additionally $\dim(S)<\dim(X)$, then $\pi$ is called a \textit{Mori
fiber space}. In dimension $3$ for a Mori
fiber space there are only three possibilities listed in~\xref{3-dim-MMP-cases}.

Let $\pi: X\to S$ and $\pi^{\sharp}: X^{\sharp}\to S^{\sharp}$ be Mori
fiber spaces. A birational map $\Phi: X \dashrightarrow X^{\sharp}$ is said to be \textit{fiberwise} if 
there exists a birational map $\alpha: S \dashrightarrow S^{\sharp}$ making the following
diagram 
\begin{equation*}
\xymatrix{
X\ar[d]^{\pi}\ar@{-->}[r]^{\Phi} & X^{\sharp}\ar[d]^{\pi^{\sharp}}
\\
S\ar@{-->}[r]^{\alpha} & S^{\sharp}
} 
\end{equation*}
commutative.

For example, $X=\PP^n\times \PP^n$ 
has two Mori fiber space structures $\pi_i: X\to \PP^n$, the projections to factors.
The identity map is not fiberwise birational equivalence but the 
involution interchanging the factors is. 

A Mori fiber space $\pi: X\to S$ is said to be \textit{birationally rigid} if 
given any birational map $\Phi:
X\dashrightarrow X^{\sharp}$ to another Mori fiber space $\pi^{\sharp}:X^{\sharp}\to S^{\sharp}$, there exists a
birational selfmap $\psi: X\dashrightarrow X$ such that the composition
$\Phi\circ\psi: X\dashrightarrow X^{\sharp}$ is fiberwise:
\begin{equation*}
\xymatrix{
X\ar@/^17pt/@{-->}[rr]^{\Phi}\ar@{-->}[r]_{\psi}\ar[d]^{\pi}&X\ar[d]^{\pi}\ar@{-->}[r] & X^{\sharp}\ar[d]^{\pi^{\sharp}}
\\
S&S\ar@{-->}[r]^{\alpha} & S^{\sharp}
} 
\end{equation*}
 and induces an isomorphism $X_{\eta} \to X^{\sharp}_{\eta}$
of generic fibers.
Examples will be given in Corollary~\xref{corollary-surface-rigid} and Theorem~\xref{Sarkisov-theorem}.
It is important to note that the rigidity implies non-rationality 
but the rigidity is much stronger.

\section{Conic bundles}
\label{section:cb}
\begin{definition}\label{definition-conic-bundle}
A \textit{conic bundle} is a proper flat morphism $\pi: X\to S$ of smooth varieties 
such that it is of relative dimension one and the anti-canonical divisor $-K_X$ is relatively ample.
A conic bundle $\pi: X\to S$ is said to be \textit{standard} if one of the following 
equivalent conditions holds:
\begin{enumerate}
\renewcommand\labelenumi{\rm (\alph{enumi})}
\renewcommand\theenumi{\rm (\alph{enumi})}
\item
$\Pic(X)=\pi^*\Pic(S)\oplus\ZZ$,
\item
$\uprho(X/S)=1$, i.e. $\pi$ is a Mori extremal contraction,
\item
for any prime divisor $D\subset S$ its preimage $\pi^*(D)$ is irreducible.
\end{enumerate}
\end{definition}
If $\pi: X\to S$ is any Mori extremal contraction from a smooth threefold to a surface, then
$\pi$ is a standard conic bundle \cite{Mori-1982}. There is a
similar, but weaker, result in arbitrary dimension:
if $X$ is a smooth variety and $\pi: X\to S$ is a Mori extremal contraction such that the dimension of \emph{any} fiber 
equals $1$, then $\pi$ is a standard conic bundle \cite{Ando1985}.

The following facts are well-known, see \cite[Prop.~1.2]{Beauville1977}, \cite[\S~1]{Sarkisov-1982-e}.
\begin{theorem}
Let $\pi: X\to S$ be a conic bundle. Then we have:
\begin{enumerate}
\item 
The anti-canonical line bundle $\upomega_X^{-1}$ is relatively very ample 
and defines an embedding $X\hookrightarrow\PP(\EEE)$, where 
$\EEE :=\pi_*\upomega_X^{-1}$ is a locally free sheaf of rank $3$.
\item 
In the above embedding $X$ is the zero locus of a section 
\begin{equation*}
\sigma\in H^0 \bigl(\PP(\EEE), \LLL^{\otimes 2}\otimes p^*(\det \EEE^\vee\otimes \upomega^{-1}_S)\bigr),
\end{equation*}
where $\LLL=\OOO_{\PP(\EEE)}(1)$ is the tautological line bundle on $\PP(\EEE)$ 
and $p: \PP(\EEE)\to S$ is the projection.

\item 
For any point $s\in S$ the scheme-theoretic fiber $X_s=\pi^{-1}(s)$ is isomorphic
over the residue field $\Bbbk(s)$ to a conic \textup(possibly reducible or non-reduced\textup) on the plane $\PP(\EEE)_s=\PP^2_{\Bbbk(s)}$.
\end{enumerate}
\end{theorem}

\subsection{}\label{discriminant}
If $\pi: X\to S$ is a conic bundle, then 
a general fiber $X_s=\pi^{-1}(s)$ is a non-degenerate conic, i.e. an irreducible smooth rational curve.
Put 
\begin{eqnarray*}
\Delta&:=&\{s\in S\mid \text{$X_s$ is a degenerate conic}\}, 
\\
\Delta_{\mathrm{s}}&:=&\{s\in S\mid \text{$X_s$ is a double line}\}.
\end{eqnarray*}
Then $\Delta$ is a divisor on $S$. It is called the \emph{discriminant} (or \emph{degeneration}) {divisor}.
The set $\Delta_{\mathrm{s}}$ coincides with the singular locus of $\Delta$ (see \eqref{equation-cb} below). 

In a small affine neighborhood $U\subset S$ of $s\in S$,
one can write the equation of $X\subset \PP^2_{x_0, x_1,x_2}\times U$
in the form 
\begin{equation}
\label{equation-conic-bundle}
\sum_{0\le i,j\le 2} a_{i,j} x_ix_j=0
\end{equation} 
where $a_{i,j}\in \CC[U]$. Then $\Delta$ is given by the determinant equation 
\begin{equation*}
\det \|a_{i,j}(s)\|=0.
\end{equation*} 

\begin{scase}
\label{equation-cb}
Moreover, $\rk \|a_{i,j}(s)\|\neq 0$ and using $\CC[U]$-linear coordinate change 
one can put the equation \eqref{equation-conic-bundle} to one of the following forms
\begin{equation*}
\begin{array}{rclll}
b_0x_0^2+b_1x_1^2+b_2x_2^2=0 &\Leftrightarrow& \rk \|a_{i,j}(s)\|=3&\Leftrightarrow& s\notin \Delta
\\[5pt]
b_0x_0^2+b_1x_1^2+c_2x_2^2=0 &\Leftrightarrow& \rk \|a_{i,j}(s)\|=2&\Leftrightarrow& s\in \Delta\setminus \Delta_{\mathrm{s}}
\\[5pt]
b_0x_0^2+c_1x_1^2+c_2x_2^2+ c_3x_1x_2 =0 &\Leftrightarrow& \rk \|a_{i,j}(s)\|=1&\Leftrightarrow& s\in \Delta_{\mathrm{s}}
\end{array}
\end{equation*} 
where $b_i(s)\neq 0$, $c_i(s)=0$, and $\mult_s (c_1)=\mult_s (c_2)=1$.
\end{scase}

\begin{scorollary}[{\cite[Proof of Proposition 1.8.5]{Sarkisov-1982-e}}]
\label{normal-crossing}
The discriminant divisor of a conic bundle $\pi: X\to S$
has only normal crossings in codimension two in $S$.
\end{scorollary}

\begin{sremark}
Let $\pi: X\to S$ be a contraction such that there exists a closed subset 
$Z\subset S$ of codimension $\ge 2$ such that the restriction
$\pi^o: X^o\to S^o$ is a conic bundle, where $S^o:=S\setminus Z$
and $X^o:=\pi^{-1}(S^o)$.
Then we can define the discriminant divisor $\Delta\subset S$ of $\pi$ as the closure
of the discriminant curve $\Delta^o$ of the conic bundle $\pi^o$.
In general $\Delta$ is a (reduced) effective Weil divisor.
Note however that in this case $\Delta$ is just \textit{divisorial} part of the maximal subset 
$R\subset S$ over which $\pi$ is not smooth. 
We do not assert that $\pi$ is smooth or flat over $S\setminus \Delta$
(see e.g.~\xref{item=main--th-pr-toric}).

The above definition can be applied to arbitrary Mori extremal contractions $\pi: X\to S$ 
such that $\dim X=\dim S+1$ and $X$ has at worst terminal singularities.
Indeed, let $S_{\ge 2}\subset S$ be the set of points $s\in S$ such that 
$\dim \pi^{-1}(s)\ge 2$ and let 
\begin{equation*}
Z:=S_{\ge 2}\cup \pi(\Sing(X))
\end{equation*}
Since $\pi$ is extremal, $\pi^{-1}(S_{\ge 2})$ has no divisorial components.
Since $X$ has at worst terminal singularities, $\codim_X \Sing(X)\ge 3$.
Therefore, $\codim_S Z\ge 2$.
Then $S^o:= S\setminus Z$ is smooth and $\pi^o$ is a conic bundle over $S^o$
(see \cite{Mori-1982}, \cite{Ando1985}). 
\end{sremark}

\begin{slemma}\label{lemma-bir-disc}
Let $\pi: X\to S$ 
and $\pi': X'\to S$ be a standard conic bundles over a \textup(not necessarily proper\textup) variety $S$.
Suppose that there is a fiberwise birational equivalence
\[
\xymatrix{
X\ar[d]^{\pi}\ar@{-->}[r]& X'\ar[d]^{\pi'}
\\
S\ar@{=}[r] & S
}
\]
Then discriminant divisors of $\pi$ and $\pi'$ coincide.
\end{slemma}

\begin{proof}
Let $\Delta$ and $\Delta'$ be discriminant divisors of $\pi$ and $\pi'$, respectively.
We may assume that $S=\Spec A$ is a small affine neighborhood of a fixed point.
Then it is sufficient to consider the case where $\Delta=\emptyset$.
Moreover, by shrinking $S$ we may assume that the vector bundles $\pi_*\OOO_X(-K_X)$
and $\pi_*\OOO_{X'}(-K_{X'})$ are trivial.
Then the anti-canonical divisors define embeddings $X,\, X' \hookrightarrow \PP^2\times S=\PP^2_A$.
Thus $X$ and $X'$ are given in $\PP^2_A$ by equations 
$\sum q_{i,j}x_i x_j=0$ and $\sum q_{i,j}'x_i x_j=0$, respectively, where $q_{i,j},\, q_{i,j}'\in A$.
By our assumption the matrix $\|q_{i,j}\|$ is non-degenerate and $\det \|q'_{i,j}\|=0$ 
is the equation of $\Delta'$. Since $X$ and $X'$ are birationally equivalent over $S$, the generic fibers 
$X_\eta$ and $X'_\eta$ are isomorphic over the field $\KK:=\mathrm{Frac}(A)$.
Hence there exists a projective transformation $\PP^2_{\KK}\to \PP^2_{\KK}$
that maps $X_\eta$ to $X'_\eta$. Let $T=\|t_{i,j}\|$, $t_{i,j}\in \KK$ be the corresponding matrix.
Then the equation of $\Delta'$ is written in the form 
\begin{equation*}
\det \|q'_{i,j}\|=(\det\|t_{i,j}\|)^2(\det \|q_{i,j}\|)=0, 
\end{equation*}
where $\det \|q_{i,j}\|$ is an invertible element 
of $A$. Since the discriminant curve of a standard conic bundle is reduced, 
the element $\det\|t_{i,j}\|$ must be invertible as well. This means that $\Delta'=\emptyset=\Delta$
on $S$.
\end{proof}

\begin{scorollary}
\label{cor:disc}
Suppose we have a fiberwise birational equivalence of $\QQ$-conic bundles
\begin{equation*}
\xymatrix{
X\ar[d]^{\pi}\ar@{-->}[r]^{\Phi} & X'\ar[d]^{\pi'}
\\
S\ar[r]^{\alpha} & S'
} 
\end{equation*}
where $\alpha$ is a birational morphism. Then for the corresponding
discriminant curves $\Delta\subset S$ and $\Delta'\subset S'$ one has $\Delta'=\alpha(\Delta)$.
\end{scorollary}
The corollary shows that the discriminant curves of fiberwise birational class of $X/S$ define a \textit{b-divisor} on $S$,
see \cite{Shokurov2003-en}, \cite{Iskovskikh-2003-b-div}. Note however that this is 
not a b-Cartier b-divisor (if it is not trivial).

\begin{lemma}
\label{lemma-properties-Conic-bundles-2}
Let $\pi: X\to S$ be a standard conic bundle. 
Then there is an isomorphism
\begin{equation*}
\Pic(X)\simeq
\begin{cases} 
\pi^*\Pic(S)\oplus\ZZ\cdot K_X &\text{
if $\pi$ has no rational sections,}
\\
\pi^*\Pic(S)\oplus\ZZ\cdot D &\text{
if $\pi$ has a rational section $D$.}
\end{cases}
\end{equation*}
\end{lemma}

From now on, we mainly concentrate on the study of
conic bundles over rational surfaces. 
Let us consider several well-known examples.

\begin{subexample}
Let $X\subset \PP^2\times \PP^2$ be a smooth divisor of bidegree $(2,d)$, $d>0$ 
given by the equation $f(x_0,x_1,x_2,y_0,y_1,y_2)=0$. This equation is quadratic in
$x_0,x_1,x_2$, so it can be viewed as 
a symmetric $3\times 3$-matrix $Q$ whose entires are homogeneous polynomials in $y_0,y_1,y_2$
of degree $d$. 
The projection $X\to \PP^2$ to the second factor is a standard conic bundle whose 
discriminant curve is given by the 
equation $\det(Q)=0$ of degree $3d$.
\end{subexample}

Starting with some well-known rationally connected variety, for example, Fano threefold, in many cases
one can construct a conic bundle by using some special birational transformations,
so-called Sarkisov links, see Sect.~\xref{section:Sarkisov}.
We give several such constructions below.

\begin{subexample}\label{example-cubic-conic-bundle}
Let $Y=Y_3\subset \PP^4$ be a smooth cubic hypersurface.
It is well known (see e.g. \cite{Altman1977}) that 
$Y$ contains a two-dimensional family of lines $\Sigma(Y)$.
Let $l\subset Y$ be a line and let $\sigma: X\to Y$ be the blowup of $l$.
Let $E$ be the exceptional divisor and let $H^*=\sigma^*H$ be the pull-back of a hyperplane section.
Then the two-dimensional linear system $|H^*-E|$ is base point free and defines 
a conic bundle structure $\pi: X\to \PP^2$.
The discriminant curve $\Delta\subset \PP^2$ is of degree $5$.
Indeed, a general member $F\in |H^*-E|$ is a cubic surface and the restriction
$\pi|_{F}$ is a conic bundle over a line $l\subset \PP^2$.
Degenerate fibers of $\pi|_{F}$ correspond to points $\Delta\cap l$.
By the Noether formula there are exactly $5$ such fibers.

For a general choice of $l$ in the corresponding Hilbert scheme,
the curve $\Delta$ is smooth. 
However, for some special choice of $l$,
$\Delta$ can be singular. To illustrate this, we recall that 
the normal bundle $\NNN_{\Gamma/Y}$ of any line $\Gamma\subset Y$ has the form $\OOO_{\PP^1}(a)\oplus \OOO_{\PP^1}(-a)$, where
$a=0$ or $1$ \cite[Proof of Proposition 2.2.8]{Kuznetsov-Prokhorov-Shramov}. 
The lines with $a=1$ are called \emph{special}.
They are characterized by the property that there exists a plane $\PP^2\subset \PP^4$
such that $Y\cap \PP^2=2\Gamma+\Gamma'$, where $\Gamma'$ is also a line which is called 
\emph{complementary} to $\Gamma$. The set of complementary lines is a closed one-dimensional subset in  $\Sigma(Y)$. 
Now, if we take in the above construction the line  $l=\Gamma'$ to be complementary, then 
the proper transform of $\Gamma$ on  $X$ will be a non-reduced fiber of the conic bundle  $\pi$ and so
the discriminant curve  $\Delta$ will be singular at the corresponding point.
If $l$ is not complementary, then the  discriminant curve $\Delta$ is smooth.
Moreover, the discriminant curve $\Delta$ can be reducible:
suppose that $Y$ contains a cubic cone $Z$. Take a line $l$ so that $l\subset Z$
and let $Z_X\subset X$ be the proper transform of $Z$. Then $\Delta_1:=\pi(Z_X)$
is an irreducible component of $\Delta\subset \PP^2$ of degree $1$.
\end{subexample}

There is another type of conic bundles over $\PP^2$ with discriminant curve of degree $5$:

\begin{subexample}[cf. \cite{Panin1980}, \cite{Blanc-Lamy-2012}]
\label{example-Panin}
Let $\Gamma\subset \PP^3$ be a smooth curve of degree $7$ and genus $5$.
By the Riemann-Roch theorem, 
the embedding $\Gamma\subset \PP^3$ is given by the complete linear system
$|K_{\Gamma}-P|$, where $P\in \Gamma$ is a point.
Let $\sigma: X\to \PP^3$ be the blowup of $\Gamma$
and let $E$ be the exceptional divisor. 
Let $H^*=\sigma^*H$ be the pull-back of a hyperplane in $\PP^3$.
Then the linear system $|3H^*-E|$ is base point free and defines 
a conic bundle structure $\pi: X\to \PP^2$.
The fibers of $\pi$ are proper transforms of conics in $\PP^3$ meeting 
$\Gamma$ at six points and a general member of $|3H^*-E|$ is a cubic surface.
As above, the discriminant curve $\Delta\subset \PP^2$ is of degree $5$. 

Since $\uprho(X)=2$ and both contractions $\pi$ and $\sigma$ are $K_X$-negative, $X$ is a Fano threefold
with $-K_X^3=16$ (see \cite[No. 9]{Mori-Mukai-1981-82}).
The morphism $\sigma\times \pi: X\to \PP^3\times \PP^2$ is an embedding and its image 
is an intersection of two divisors of bidegrees $(1,1)$ and $(2,1)$.
\end{subexample}

Since any smooth three-dimensional cubic hypersurface is not rational \cite{Clemens-Griffiths}, 
these two types of conic bundle (with discriminant curve
of degree $5$) cannot be birationally equivalent. They differ by the type of the corresponding double cover 
$\tilde\pi:\tilde\Delta\to\Delta$ (see~\xref{def-double-cover} below); for a conic bundle originating from a
cubic hypersurface, the cover is defined by an odd theta-characteristic, whereas the
theta-characteristic for the conic bundle constructed in Example~\xref{example-Panin}
is even (see~\xref{theta-characteristic} below).

\begin{subexample}[{\cite[Example~1.4.4]{Beauville1977}}, {\cite[Th.~4.3.3]{IP99}}]
\label{example-V8-conic-bundle}
Let $Y=Y_{2\cdot 2\cdot 2}\subset \PP^6$ be a smooth complete intersection of 
three quadrics.
It is known that $Y$ contains a one-dimensional family of lines
(see e.g. \cite[Lemma 5.2]{Tjurin1975}).
Let $l\subset Y$ be a line and let $\sigma: \tilde Y\to Y$ be the blowup of $l$.
Let $E$ be the exceptional divisor and let $H^*=\sigma^*H$ be the pull-back of a hyperplane section.
The linear system $|-K_{\tilde Y}|$ is base point free 
and defines a generically finite (but not finite) morphism. 
For a general choice of line $l$ this morphism is small,
i.e. it does not contract divisors. According to \cite{Kollar1989a}, there exists a flop 
$\tilde Y\dashrightarrow X$ and by the Cone Theorem on $X$ there exists a Mori extremal contraction. 
It is not hard to show that the only possibility for this contraction
is conic bundle $\pi: X\to \PP^2$ with discriminant curve $\Delta\subset \PP^2$ is of degree $7$.
Then the map $\tilde Y \dashrightarrow \PP^2$ is given by the linear system $|2H^*-3E|$,
see \cite[Theorem~4.3.3(ii)]{IP99} for details.
\end{subexample}

\begin{subexample}
Let $Y=Y_{2\cdot 3}\subset \PP^5$ be an intersection of a quadric and a cubic.
Suppose that $Y$ contains a plane $\Pi=\PP^2$. The projection from $\Pi$ 
induces a rational curve fibration on $Y$. 
If $Y$ is sufficiently general, then by blowing up $\Pi$ we obtain a standard conic bundle 
over $\PP^2$ with discriminant curve of degree $7$, see e.g. \cite[Example~1.4.6]{Beauville1977}.
In this case $Y$ is the midpoint of a Sarkisov link \cite[Proposition~7.11]{Jahnke-Peternell-Radloff-II}.
\end{subexample}

More examples of conic bundles over $\PP^2$ with discriminant curve of degree $7$ can be found in
\cite{BrownCortiZucconi-2004}.

\begin{subexample}[{\cite[Example~1.4.3]{Beauville1977}}, \cite{Conte1977a,Conte1977b}]
Let $Y\subset \PP^4$ be a quartic hypersurface which is singular along a line $l$.
Suppose that $Y$ is general. Then by blowing up $l$ we obtain a 
standard conic bundle over $\PP^2$ with discriminant curve of degree $8$.
\end{subexample}

\begin{subexample}[{\cite[Example~1.4.5]{Beauville1977}}, \cite{Debarre1990},
\cite{Przhiyalkovskij-Cheltsov-Shramov-2015}, \cite{Prokhorov-factorial-Fano-e}]
\label{example-4-double-solid}
Let $\varphi:Y\to \PP^3$ be a double cover branched over a
quartic surface $B\subset \PP^3$. Assume that $B$ is singular and its singular 
locus consists of a unique node $Q$. Then $P:=\varphi^{-1}(Q)$ is a unique 
singular point of $Y$. In this case, $Y$ is $\QQ$-factorial (see e.g. 
\cite{Cheltsov2009a}). Let $\sigma: X\to Y$ and $\lambda: \tilde \PP^3\to \PP^3$ 
be blowups of $P$ and $Q$, respectively. Then $\varphi$ induces a 
double cover $\tilde \varphi: X\to \tilde \PP^3$ and the projection from 
$Q$ induces a $\PP^1$-bundle $\psi:\tilde \PP^3\to \PP^2$. The 
composition $\pi: X\to \tilde \PP^3\to \PP^2$ is a standard conic bundle whose 
fibers are double covers of the fibers of $\psi$. The discriminant curve is of 
degree $6$.

Now if the point $P$ is not a unique singular point of $X$, then 
in the above construction $X$ is singular and $\pi: X\to \tilde \PP^3\to \PP^2$ is non-standard 
conic bundle. The standard forms for $X/\PP^2$ were described in \cite{Przhiyalkovskij-Cheltsov-Shramov-2015}.
\end{subexample}

\subsection{}
Let $\pi: X\to S$ be a conic bundle over a surface $S$ and let $\Delta\subset S$ be its 
discriminant curve. According to discussions in~\xref{discriminant}
the curve $\Delta$ is a reduced normal crossing divisor (possibly $\Delta=\emptyset$).
Moreover, a fiber $X_s$, $s\in S$ is smooth (resp. a pair of meeting lines, double line)
if $s\in S\setminus \Delta$ (resp. $s\in \Delta\setminus\Sing(\Delta)$, $s\in \Sing(\Delta)$).

The following facts are easy and well-known.
\begin{lemma}
\label{lemma-7-conic-bundle}
Let $\pi: X\to S$ be a conic bundle over a \textup(projective\textup) surface
and let $\Delta$ be the discriminant curve. Then
\begin{eqnarray}
\label{lemma-7-conic-bundle-1}
\chi_{\operatorname{top}}(X)&=&2\chi_{\operatorname{top}}(S)-2\p(\Delta)+2,
\\
\label{lemma-7-conic-bundle-2}
\bb_1(X)&=&\bb_1(S), 
\\ 
\label{lemma-7-conic-bundle-3}
\bb_3(X)&=&2\bb_1(S)+2\bb_2(X)-2\bb_2(S)+2\p(\Delta)-4,
\end{eqnarray}
\textup(for the case $\Delta=\emptyset$ we put $\p(\emptyset)=1$\textup).
\end{lemma}

\begin{lemma}\label{lemma-properties-Conic-bundles}
Let $\pi: X\to S$ be a conic bundle over a surface and let $\Delta\subset S$ be its 
discriminant curve.
\begin{enumerate}
\item \label{lemma-properties-Conic-bundles-1}
If $X$ is rationally connected, then 
the surface $S$ is rational. Conversely, if $\pi: X\to S$ be a conic bundle over a rational
surface, then $X$ is rationally connected \cite{Kollar-Miyaoka-Mori-1992a}, \cite{Graber-Harris-Starr-2003}.
\item 
If $\pi$ has a rational section, then $X$ is birationally equivalent to $S\times \PP^1$.

\item 
Assume that the conic bundle is standard and the Brauer group of $S$ is trivial 
\textup(for example this holds if $S$ is rational\textup). 
Then $\pi$ has a rational section if and only if $\Delta=\emptyset$.
\end{enumerate}
\end{lemma}

\subsection{Admissible double covers.}\label{double-cover}
Let $\tilde\Delta$ be a reduced connected curve with at worst nodal singularities
and let $\tau: \tilde\Delta\to \tilde\Delta$ be an involution 
(an automorphism of order $2$). Let $\Delta:= \tilde\Delta/\tau$ be the quotient
and let $\tilde\pi:\tilde\Delta\to\Delta$ be the natural projection.
It is easy to show that the singularities of $\Delta$ are also at worst nodes.
We also assume that the following condition is satisfied 
(the so-called \emph{Beauville condition}):
\begin{equation}\label{Beauville-condition}
\pi(\Sing (\tilde\Delta))=\Sing (\Delta),\quad 
\Sing (\tilde\Delta)=\{ x\in\tilde\Delta\mid
\tau(x)=x\}.
\end{equation}
Then the restriction of $\tilde\pi$ to each irreducible component 
of $\tilde\Delta$ does not split and, at a singular point of $\tilde \Delta$, 
two branches are not switched by $\tau$. 
From \eqref{Beauville-condition} one can easily deduce the following condition:
\begin{equation}
\label{equation-condition-S-star-star}
\begin{tabularx}{0.8\textwidth}{X}
for every decomposition $\Delta=\Delta_1+\Delta_2$ with $\Delta_i\ge 0$
\\
we have
$\#(\Delta_1\cap \Delta_2)\equiv 0 \mod 2$.
\end{tabularx}
\end{equation}
Note however, that in our situation $\tilde\pi$ is not necessarily flat.

\subsection{}
\label{def-double-cover}
Now, let $\pi: X\to S$ be a standard conic bundle over a projective surface $S$ and let $\Delta\subset S$ be its 
discriminant curve.
Let $\tilde\Delta$ be the curve parameterizing components of fibers in the ruled surface
$X_{\Delta}:=\pi^{-1}(\Delta)$. The induced projection $\tilde\pi:\tilde\Delta\to\Delta$ 
is finite of degree $2$ and satisfies the conditions of~\xref{double-cover}. 

\begin{scorollary}\label{discriminant-divisor-pa=0}
Any connected component of the discriminant divisor of a standard conic bundle
over a surface has arithmetic genus at least $1$.
\end{scorollary}

\begin{scase}\label{theta-characteristic}
Recall that a \textit{theta-characteristic} of a smooth curve $\Delta$ is the
linear equivalence class of a divisor $D$ such that $2D\sim K_\Delta$.

A theta-characteristic can be even or odd depending on the parity of the dimension of 
$H^0(\Delta, \OOO_{\Delta}(D))$
(see e.g. \cite{Mumford1971}, \cite[Ch.~5]{Dolgachev-ClassicalAlgGeom}).
This definition can be naturally generalized to the case of reduced Gorenstein curves.
A \textit{theta-characteristic} of such a curve is a rank-$1$ torsion-free sheaf $\FFF$
such that $\HHom_{\OOO}(\FFF,\upomega_{\Delta})\simeq \FFF$ \cite{Beauville1977a}, \cite{Piontkowski2007}.
\end{scase}

\begin{subexample}
Suppose that $S=\PP^2$, $\Delta$ is smooth, and the degree of $\Delta$ is odd and $\ge 5$.
Write $\deg \Delta=2m+3$.
There is one-to-one correspondence between \'etale
double covers $\tilde \pi: \tilde \Delta\to \Delta$ and elements $\sigma$ of order $2$ in the Jacobian $\J(\Delta)$.
The linear system $|mh|$, where $h$ is the class of hyperplane section of $\Delta$, 
is a half-canonical linear system, i.e. it is a theta-characteristic.
Then $mh+\sigma$ is another theta-characteristic.
In other words, the group of $2$-torsion points $\J_2(\Delta)\subset \J(\Delta)$ acts on the set of all theta-characteristics $\operatorname{Th}(\Delta)$
making it a principal homogeneous space.
The choice of a distinguished point $|mh|\in \operatorname{Th}(\Delta)$ establishes an identification $\operatorname{Th}(\Delta)\simeq \J_2(\Delta)$.
We say that $mh+\sigma$ is the theta-characteristic corresponding to the cover $\tilde \Delta\to \Delta$.
Thus there is one-to-one correspondence between \'etale
double covers $\tilde \pi: \tilde \Delta\to \Delta$ and theta-characteristics of $\Delta$.
\end{subexample}

It turns out that the cover $\tilde \Delta\to \Delta$ is the most important invariant. 
In particular, this cover ``almost determine'' the conic bundle:

\begin{proposition}\label{proposition-AM}
Let $S$ be a rational surface, $\Delta\subset S$ a reduced normal crossing
curve, and $\tilde\pi:\tilde\Delta\to\Delta$ a double cover 
satisfying conditions \eqref{Beauville-condition}. Then there exists a standard conic
bundle $\pi: X\to S$ with the given $\tilde\pi:\tilde\Delta\to\Delta$, and all such standard conic bundles are
birationally equivalent over $S$.
\end{proposition}
\begin{proof}[Sketch of the proof.]
Consider so-called \textit{Artin-Mumford exact sequence}
\begin{equation*}
0\to\Br S\to\Br\Bbbk(S)\overset{\alpha}\longrightarrow 
\bigoplus_{\substack{\text{curves}\\
C\subset S}}H^1_{\mathrm{et}}(C,\QQ/\ZZ) 
\overset{\beta}\longrightarrow\bigoplus_{\substack{\text{points}\\ P\in C}}\mumu^{-1} 
\overset{\gamma}\to\mumu^{-1}\to 0
\end{equation*}
(see~\cite{Artin-Mumford-1972}), where
\begin{equation*}
\mumu^{-1}=\cup_n\Hom(\mumu_n,\QQ/\ZZ), 
\end{equation*}
and $\mumu_n$ denotes the group of $n$-th roots
of unity. Then the double cover $\tilde\pi: \tilde \Delta\to \Delta$
determines an element 
\begin{equation}
\label{eq:local-inv}
a\in \bigoplus _{C\subset \Delta} H^1_{\et}(C,\QQ/\ZZ)
\end{equation}
of order $2$ (a collection of \textit{local invariants}). 

From \eqref{Beauville-condition} it follows immediately that $\beta(a)=0$,
and hence there is an element $A\in\Br(\Bbbk(S))$ such that $\alpha(A)=a$. 
As $S$ is a
rational surface, we have $\Br(S)=0$. Hence, $A\in \Br(\Bbbk(S))$ has order $2$. By virtue of the
well-known Merkur'ev theorem~\cite{Merkurev1981} $A$ is a product of classes of
quaternion algebras over the function field $\Bbbk(S)$. 
Furthermore, since $\Bbbk(S)$ is a
$\mathrm{c}_2$-field, by Albert's theorem \cite[Theorem 2.10.9]{Jacobson1996} 
the product of classes of quaternion algebras is
also represented by a quaternion algebra. Thus there exists a quaternion algebra $\mathscr A$
over $\Bbbk(S)$ whose class in $\Br(\Bbbk(S))$ is equal to $A$. According to the classical
theorem on central simple algebras $\mathscr A$ is uniquely
determined up to
isomorphism.
As in ~\cite{Artin-Mumford-1972} (see also~\cite[Theorem~5.3]{Sarkisov-1982-e}), one proves that the maximal orders of
$\mathscr A$ over $S$ are in $1${-}$1$ correspondence with standard conic bundles $\pi: X\to S$
with given local invariants $\tilde\pi:\tilde\Delta\to\Delta$ \eqref{eq:local-inv}. All such
conic bundles are
birationally equivalent over $S$ since their generic fibers are isomorphic conics over
$\Bbbk(S)$ associated to the same quaternion algebra $\mathscr A$.
\end{proof}

\begin{subexample}
Let $X\subset \PP^2_{u_0,u_1,u_1}\times \PP^2_{x_0,x_1,x_1}$ be given by the equation 
\begin{equation*}
f_0(u_0,u_1,u_1)x_0^2 +f_1(u_0,u_1,u_1)x_1^2 +f_2(u_0,u_1,u_1)x_2^2 =0,
\end{equation*}
where $f_i$ are general homogeneous polynomials of degree $d\ge 1$. The projection 
$\pi: X\to \PP^2$
to the first factor 
is a standard conic bundle whose discriminant curve is reducible curve of degree $3d$
given by $f_0f_1f_2=0$. 
The corresponding quaternion algebra $A$ over $\Bbbk(\PP^2)$
is generated by four vectors $\mathbf 1$, $\mathbf i$, $\mathbf j$, $\mathbf k$
with standard relations
\begin{equation*}
\mathbf i \cdot \mathbf j=\mathbf k=- \mathbf j \cdot \mathbf i,
\quad \mathbf i^2=- f_0/f_2,
\quad \mathbf j^2=- f_1/f_2,
\quad \mathbf k^2=- f_0f_1/f_2^2.
\end{equation*}
\end{subexample}

Note however that the recipe given in Proposition~\xref{proposition-AM}
allows to recover a conic bundle $\pi: X\to S$ only up to \textit{birational equivalence} over the base $S$.
There is no canonical way of reconstruction of $X/ S$ by the double cover $\tilde \Delta\to \Delta$.

\begin{lemma}[cf. {\cite{Sarkisov-1980-1981-e}}]
\label{lemma-canonical-bundle-formula}
Let $\pi: X\to S$ be a contraction of projective varieties
such that any fiber is one-dimensional, 
$X$ has at worst terminal singularities, and 
$-K_X$ is $\pi$-ample. 
Let $\Delta\subset S$ be the discriminant divisor.
Then the following numerical equivalence of cycles on $S$ holds, where $K_X^2$ is regarded as a 
rational equivalence class of codimension two cycles 
\begin{equation}
\label{equation-canonical-bundle-formula}
-\pi_* K_X^2\equiv 4K_S+\Delta. 
\end{equation}
\end{lemma}

\begin{proof}
We prove \eqref{equation-canonical-bundle-formula} by induction on the 
dimension.
If $\dim X=2$, then $X$ is smooth and \eqref{equation-canonical-bundle-formula}
is an immediate consequence of the Noether formula.
For $\dim X\ge 3$, let $Z\subset S$ be an effective very ample divisor and let $Y:=\pi^{-1}(Z)$.
Take $Z$ to be general in the corresponding linear equivalence class.
Then $\pi_Y: Y\to Z$ will satisfy the same conditions as $\pi: X\to S$.
Let $\Delta_Z=\Delta\cap Z$ be the corresponding discriminant locus.
It is sufficient to show that 
\begin{equation*}
(\pi_* K_X^2+4K_S+\Delta)\cdot Z\equiv 0.
\end{equation*}
Using the projection and adjunction formulas we can write
\begin{multline*}
(\pi_* K_X^2+4K_S+\Delta)\cdot Z
\equiv
\pi_*(K_X^2\cdot Y)+4(K_Z-Z^2)+\Delta_Z\equiv
\\
\equiv\pi_* K_Y^2-2\pi_* (K_Y \cdot Y^2)+4K_Z-4Z^2+\Delta_Z\equiv
\pi_* K_Y^2+4K_Z+\Delta_Z.
\end{multline*}
By the inductive hypothesis this proves \eqref{equation-canonical-bundle-formula}.
\end{proof}

\begin{theorem}
\label{theorem-standard-models}
Given any $\QQ$-conic bundle $\pi: X\to S$, there is a birational contraction 
$\alpha: S^{\bullet}\to S$ and a standard conic bundle $\pi^{\bullet}: X\to S^{\bullet}$ that is fiberwise birationally
equivalent to $\pi: X\to S$.
More precisely, there exists a commutative diagram
\begin{equation}
\label{equation-(1)}
\vcenter{
\xymatrix{
X^{\bullet}\ar@{-->}[r]^{\psi}\ar[d]^{\pi^{\bullet}}&X\ar[d]^{\pi}
\\
S^{\bullet}\ar[r]^{\alpha} & S
}}
\end{equation}
where $\psi$ is a birational map and $\alpha$ is a birational \emph{morphism}.
\end{theorem}
This theorem is a particular case of more general fact proved by Sarkisov:
for any dominant rational map 
$g: Y\dashrightarrow T$ of relative dimension $1$ whose generic
fiber is an irreducible rational curve there exists a standard conic bundle 
$\pi^{\bullet}: X\to S^{\bullet}$ that is fiberwise birationally
equivalent to $g$. 
The proof can be found in~\cite[Theorem~1.13]{Sarkisov-1982-e}.
Three-dimensional case was outlined earlier in \cite{Zagorskiui1977}
and \cite{Miyanishi1983}.
See also \cite{Avilov2014} for the three-dimensional equivariant version.

\section{Sarkisov category}
\label{section:Sarkisov}
In this section we describe the general structure of Sarkisov program.
For details we refer to \cite{Corti1995a}, see also \cite[Ch.~13]{Matsuki2002}, \cite{Iskovskikh2001}, 
\cite{IskovskikhShokurov2005}, \cite{Shokurov-Choi-2011}, 
\cite{HaconMcKernan2013}.
\subsection{}
\label{definition-Sarkisov-links}
First, we recall the definition of a
Sarkisov link~\cite[Definition~3.4]{Corti1995a}. Suppose that $\pi :X\to S$ and 
$\pi_1: X_1\to S_1$ are Mori fiber spaces. 
A Sarkisov link $\Phi: X_1 \dashrightarrow X$ between them is a transformation of one of four types
\begin{equation*}
\begin{array}{cccc}
\mathbf{I}&\mathbf{II}&\mathbf{III}&\mathbf{IV}
\\
\xymatrix{
Z\ar[d]_p\ar@{-->}[r]^{\chi}&X_1\ar[d]^{\pi_{1}}
\\
X\ar[d]_{\pi}&S_{1}\ar[ld]^{\alpha}
\\
S&
}
&
\xymatrix@C=7pt{
Z\ar[d]_p\ar@{-->}[rr]^{\chi}&&Z_1\ar[d]^q
\\
X\ar[dr]_{\pi}&&X_{1}\ar[dl]^{\pi_{1}}
\\
&S&
}&
\xymatrix{
X\ar[d]^{\pi}\ar@{-->}[r]^{\chi}&Z_1\ar[d]^q
\\
S\ar[rd]_{\alpha}&X_{1}\ar[d]^{\pi_{1}}
\\
&S_{1}
}&
\xymatrix@C=7pt{
X\ar[d]^{\pi}\ar@{-->}[rr]^{\chi}&&X_1\ar[d]^{\pi_1}
\\
S\ar[dr]_{\alpha}&&S_1\ar[dl]^{\alpha_1}
\\
&T&
}
\end{array}
\end{equation*}
where $p: Z\to X$ and $q: Z_1\to X_1$ are divisorial Mori extremal contractions, $\chi$ is a finite
sequence of log-flips, in particular, $\chi$ is an isomorphism in codimension one, and
$\alpha$ and $\alpha_1$ are extremal contractions in the log terminal category. 

A link of type~\typem{III} is a birational transformation that is inverse
to a transformation corresponding to a type~\typem{I} link.
For standard conic bundles, links of types \typem{I} and \typem{III} are exactly
elementary transformations described in Proposition~\xref{proposition-Sarkisov-elementary-transformations}.
In dimension $\le 3$, for a type~\typem{IV} link the contractions $S\to T$ and $S_1\to T$ must be of fiber type
(so the link switches the Mori fiber space structures).
In higher dimensions the contractions $S\to T$ and $S_1\to T$ can be small.

\begin{theorem}\label{theorem-Sarkisov-program}
Let $\pi: X\to S$ and 
$\pi^\sharp :X^\sharp \to S^\sharp$ be Mori fiber spaces 
and let
\begin{equation*}
\Phi: X \dashrightarrow X^\sharp.
\end{equation*} 
be
a \textup(not necessarily fiberwise\textup) birational map.
Then $\Phi$ is a 
composition of Sarkisov links. 
\end{theorem}
The three-dimensional version of this program was outlined in preprints 
\cite{Sarkisov1987}, \cite{Sarkisov1989},
\cite{Reid-1991-sarkisov} and a complete proof (with termination) was given in 
\cite{Corti1995a} (see also \cite{IskovskikhShokurov2005}, \cite{Shokurov-Choi-2011}).
In higher dimensions the program is established in a weaker form
\cite{HaconMcKernan2013}. This approach uses a little different decomposition algorithm. 
We follow the ``standard'' variant \cite{Corti1995a} which is more suitable for 
our applications.
The process of decomposition which existence is claimed in Theorem 
\xref{theorem-Sarkisov-program} is called the \textit{Sarkisov program}.
Below we outline the general structure of this program.
\subsection{}
Let $\HHH^{\sharp}$
be a very ample linear system on $X^{\sharp}$ and let $\HHH$ be its proper
transform on $X$. Clearly, we can write
\begin{eqnarray}
\label{definition-mu-star}
\HHH^{\sharp}&\equiv&-\mu^{\sharp}K_X+\pi^{\sharp *}A^{\sharp},
\\
\label{definition-mu}
\HHH&\equiv&-\mu K_X+\pi^*A,
\end{eqnarray}
where $\mu^{\sharp}$ and $\mu$ are positive rational numbers and 
$A^\sharp$ (resp. $A$) is a $\QQ$-divisor on $S^\sharp$ (resp. on $S$).
We say that the pair $(X,\HHH)$ or the linear system $\HHH$ \textit{has maximal singularity} if the pair 
$(X,\frac 1{\mu}\HHH)$ is not canonical.

The basic tool of the method of maximal singularities is the Noether-Fano inequality.
Below is its version which is adapted for Sarkisov program.

\begin{theorem}[Noether-Fano inequalities {\cite[Theorem~4.2]{Corti1995a}}, \cite{Iskovskikh-Noether-Fano-2004}, {\cite[Ch.~1, Def.~1.4]{Pukhlikovbook2013}}]
\label{Noether-Fano}
In the above notation the following hold.
\begin{enumerate}
\item \label{Noether-Fano-1}
$\mu \ge \mu^{\sharp}$ and equality implies that $\Phi$ is fiberwise
 and induces an isomorphism $\Phi_{\eta}: X_{\eta} \to X^{\sharp}_{\eta}$
of generic fibers.
\item \label{Noether-Fano-2}
If $(X,\frac1{\mu}\HHH)$ is canonical and $K_X+\frac1{\mu}\HHH$ is nef, then $\Phi$ is an isomorphism,
and it also induces an isomorphism $S\simeq S^{\sharp}$. In particular, $\mu=\mu^{\sharp}$.
\end{enumerate}
\end{theorem}

\subsection{Case where $\HHH$ has maximal singularity}
\label{Sarkisov-program-untwisting}
This step is often called ``\textit{untwisting maximal singularities}''.
Since $X$ is terminal, the pair $(X,c\HHH)$ is canonical but not terminal for some 
$0<c<\frac 1{\mu}$. 
This number $c$ is called the \textit{canonical threshold} of $(X,\HHH)$
and denoted by $\ct(X,\HHH)$.
There exists a log-crepant extremal blowup 
$p: (Z,c\HHH_Z)\to (X,c\HHH)$. Thus we have $\uprho (Z/S)=2$ and
\begin{equation*}
K_Z+c\HHH_Z\equiv p^*(K_X+c\HHH).
\end{equation*}
Moreover, $p$ is a divisorial contraction, $Z$ is $\QQ$-factorial and has terminal
singularities. 
Then we run the $(K_Z+c\HHH_Z)$-MMP over $S$. Since $\uprho (Z/S)=2$ we 
obtain one of the links \typem{I} or \typem{II}. Then we replace $X$ with $X_1$ and 
$\HHH$ with its proper transform, and continue the process.

\begin{scase}\label{Sarkisov-program-termination}
Each step as above either decreases the number $\e( X,c\HHH)$ of crepant divisors 
or (if $\e( X,c\HHH)=1$) increases the canonical threshold $\ct(X,\HHH)$.
The number $\e( X,c\HHH)$ is a positive integer, so it cannot decrease infinitely many times.
Termination of an increasing sequence of thresholds $\ct(X,\HHH)$ is more delicate problem.
At the moment it is not known that the set of all canonical thresholds
satisfy ascending chain condition (ACC) even in dimension $3$. Particular cases of this problem were
studied in \cite{Prokhorov-2008ct}, \cite{Stepanov-threshplds}, and \cite{Shramov-acc}.
Termination of the Sarkisov program in \cite[Theorem~6.1]{Corti1995a} was proved by reduction to the \emph{log canonical} case
\cite{Alexeev-1994}, \cite{Hacon-McKernan-Xu-2014}.
The approach of \cite{HaconMcKernan2013} is different (see also \cite{IskovskikhShokurov2005}, \cite{Shokurov-Choi-2011}).
\end{scase}

\subsection{Case where $\HHH$ has no maximal singularities}
\label{Sarkisov-program--no-max-sing}
By \eqref{definition-mu}
\begin{equation*}
\HHH+\mu K_X\equiv \pi^*A,
\end{equation*}
where the divisor $A$ is not nef by Theorem~\xref{Noether-Fano}\xref{Noether-Fano-2}.
If $\uprho(S)\le 1$, then $-(\HHH+\mu K_X)$ is ample
and we just run the $(K_X+\frac1{\mu}\HHH)$-MMP and get a link of type~\typem{IV}.

Assume that $\uprho(S)> 1$. In particular, this implies that $S$ is a surface.
There is a natural identification of Mori cones $\NE(S)=\pi_* \NE(X)$.
By the cone theorem on $X$, the cone $\NE(S)\cap \{z \mid A\cdot z<0\}$
is locally polyhedral. 
Hence there exist extremal rays $R_S\subset \NE(S)$ 
and $R_X\subset \NE(X)$ 
such that $(\HHH+\mu K_X)\cdot R_X<0$, \ $A\cdot R_S<0$, and $\pi_*R_X=R_S$.
Since $\uprho(S)>1$, we have $R_S^2\le 0$ and 
one can show that $R_S$ is contractible. Thus there exists its contraction $S\to T$
with $\uprho(S/T)=1$ and $\dim T=2$ or $1$. Then we run the $(K_X+\frac1{\mu}\HHH)$-MMP over $T$.
This produces a link of type~\typem{III} with $S_1=T$ or 
a link of type~\typem{IV}. 

In all these cases, for links of type~\typem{III} or \typem{IV},
we get a new Mori fiber space $\pi_1: X_1\to S_1$ such that 
\begin{equation*}
\label{equation-mu-reduced}
\mu_1\le \mu,
\end{equation*}
where $\mu_1$ is defined for the 
proper transform $\HHH_{X_1}$ of $\HHH$ similar to \eqref{definition-mu-star}
and \eqref{definition-mu}. Since the numbers $\mu$ have bounded denominators
(see \cite{Kawamata-1992bF}),
this shows that the sequence of these links also terminates.

\section{Surfaces over non-closed fields}
\label{section:surfaces}
In this section all the varieties (surfaces and curves) are supposed to be defined over a 
\textit{perfect} field $\Bbbk$ of arbitrary characteristic. If $\pi: X\to B$ is a two-dimensional conic bundle over a curve $B$,
then, as in~\xref{definition-conic-bundle}, we say that it is standard if $\uprho(X/B)=1$.
The discriminant locus in this case is a (reduced) zero-dimensional subscheme $\Delta\subset B$ (or empty).
Every degenerate geomtric fiber $X_{\bar b}$, $\bar b\in B\otimes \bar \Bbbk$
is a union of two $(-1)$-curves meeting transveresely at one point (see Fig. \ref{riss}).
\begin{figure}[b] 
\begin{tikzpicture}
\node () at (-3.7,2.1) {$X$};
\node () at (-3.7,-1.5) {$B$};
 \foreach \position in {  (-3,1), (-1,0.8),(0,0.7),  (2,0.9),  (4,1.1),  (5,1.15)} 
    \draw[thick,black] \position  ellipse (0.25 and 1) ;
\begin{scope}[xshift=-2cm,yshift=0.9cm]
\draw[thick,black](1.3-1.15,-1)-- (1-1.15,1) ;
\draw[thick,black] (1-1.15,-1)-- (1.3-1.15,1);
\end{scope}
\begin{scope}[xshift=1cm,yshift=0.8cm]
\draw[thick,black](1.3-1.15,-1)-- (1-1.15,1) ;
\draw[thick,black] (1-1.15,-1)-- (1.3-1.15,1);
\end{scope}
\begin{scope}[xshift=3cm,yshift=1cm]
\draw[thick,black](1.3-1.15,-1)-- (1-1.15,1) ;
\draw[thick,black] (1-1.15,-1)-- (1.3-1.15,1);
\end{scope}
\draw [thick,black,->] plot [smooth,tension=0.9] coordinates{ (-4,1) (-4.05,0) (-4,-1)} node[left,yshift=24pt]{$ f$};
\draw[thick,black] plot [smooth,tension=0.9] coordinates{(-3.5,-1.1)(-1.5,-1.4) (2,-0.95) (5.4,-1)};
\end{tikzpicture}
 \caption{}\label{riss}
\end{figure}
If the surface $X$ is geometrically rational, then by the Noether formula one has 
\begin{equation*}
K_X^2+\deg \Delta=8. 
\end{equation*}
In particular, $K_X^2\le 8$ and the equality $K_X^2=8$ implies that $\pi$ is a smooth morphism.
It is easy to show that for a standard conic bundle $K_X^2\neq 7$ (see e.g. \cite{Iskovskikh-1979s-e}).

The Sarkisov program works in the category of surfaces over $\Bbbk$.
Moreover, in the two-dimensional case all the dashed arrows (log flips) must be isomorphisms.
So, the Sarkisov links have the following simple form:
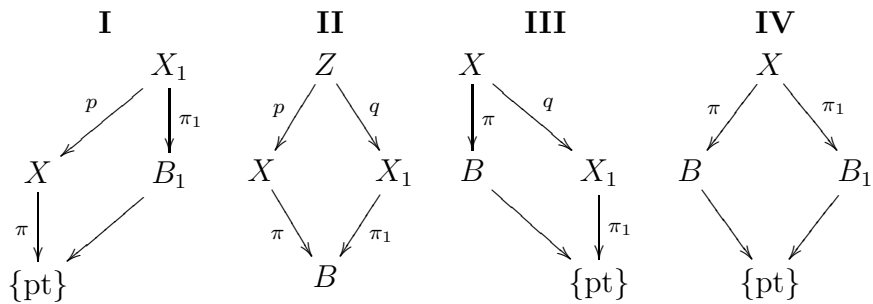
\begin{equation*}
\begin{array}{cccc}
\mathbf{I}&\mathbf{II}&\mathbf{III}&\mathbf{IV}
\\
\xymatrix{
&X_1\ar[d]^{\pi_{1}}\ar[dl]_p
\\
X\ar[d]_{\pi}&B_{1}\ar[ld]
\\
\{\pt\}&
}
&
\xymatrix@C=7pt{
&Z\ar[dl]_p\ar[dr]^q&
\\
X\ar[dr]_{\pi}&&X_{1}\ar[dl]^{\pi_{1}}
\\
&B&
}&
\xymatrix{
X\ar[d]^{\pi}\ar[dr]^q
\\
B\ar[rd]&X_{1}\ar[d]^{\pi_1}
\\
&\{\pt\}
}&
\xymatrix@C=7pt{
&X\ar[dl]_{\pi}\ar[dr]^{\pi_1}&
\\
B\ar[dr]&&B_1\ar[dl]
\\
&\{\pt\}&
}
\end{array}
\end{equation*}
In the case \typem{II} we distinguish two subcases:
\begin{itemize}
\item[\typem{II_0}] $B$ is a point,
\item[\typem{II_1}]$B$ is a curve.
\end{itemize}

Here is a short description of the links. 
In all cases morphisms $p$ and $q$ are blowups of closed points.
\par\medskip\noindent
\begin{center}
\begin{tabularx}{\textwidth}{l|X|X}
type&$X$ and $\pi$ &$X_1$ and $\pi_1$ 
\\\hline
\typem{I} &\mbox{del Pezzo surface with} $\uprho(X)=1$&
$X_1$ is a del Pezzo surface with $\uprho(X_1)=2$
\newline
$\pi_1$ is a conic bundle
\\\hline
\typem{II_0} &\mbox{del Pezzo surface with} $\uprho(X)=1$
&\mbox{del Pezzo surface with} $\uprho(X_1)=1$
\\[-6pt]\\ 
& \multicolumn{2}{c}{$Z$ is a del Pezzo surface with $\uprho(Z)=2$ }
\\\hline
\typem{II_1}& standard conic bundle & standard conic bundle 
\\\hline
\typem{III} & \mbox{$X$ is a del Pezzo surface with} $\uprho(X)=2$\newline $\pi$ is a conic bundle
&\mbox{del Pezzo surface with} $\uprho(X_1)=1$
\\\hline
\typem{IV} & standard conic bundle & standard conic bundle 
\\[-6pt]
\\ & \multicolumn{2}{c}{\mbox{$X$ is a del Pezzo surface with} $\uprho(X)=2$ }
\end{tabularx}
\end{center}
\bigskip
Iskovskikh \cite{Iskovskikh-Factorization-1996e} classified all links and relations between them in the category 
of surfaces over non-closed fields.
Below we reproduce a part of this classification.

\begin{theorem}[{\cite[Theorem 2.6]{Iskovskikh-Factorization-1996e}}, {\cite[Theorem~A.4]{Corti1995a}}, 
\cite{Prokhorov-re-rat-surf}]
\label{classification-Sarkisov-links}
Let $X/B\dashrightarrow X_1/B_1$ be a Sarkisov link in the category 
of rational surfaces. Assume that $K_X^2\le 4$.
Then one of the following holds.
\begin{enumerate}
\item \label{link-4-3}
Type \typem{I}: 
$K_X^2=4$, $K_{X_1}^2=3$, $X_1$ is a cubic surface containing 
a line $l$ defined over $\Bbbk$, $p$ is a contraction of $l$, 
$\pi_1$ is a projection from $l$. 
\item Type \typem{II}: $B$ is a point,
$K_X^2=2$, $K_Z^2=1$, $p$ and $q$ are blowups of points of degree $1$,
the link is represented by the Bertini involution on $Z$.
\item Type \typem{II}: $B$ is a point,
$K_X^2=3$, $K_Z^2=1$, $p$ and $q$ are blowups of points of degree $2$,
the link is represented by the Bertini involution on $Z$.
\item Type \typem{II}: $B$ is a point,
$K_X^2=3$, $K_Z^2=2$, $p$ and $q$ are blowups of points of degree $1$,
the link is represented by the Geiser involution on $Z$.
\item Type \typem{II}: $B$ is a curve,
the link consists of elementary transformations in non-degenerate 
geometric fibers which are conjugate over $\Bbbk$.

\item Type \typem{III}: inverse to~\xref{link-4-3}.

\item 
Type \typem{IV}: $K_X^2=1$, the link is represented by the \textup(biregular\textup) Bertini involution on $X$. 
\item 
Type \typem{IV}: $K_X^2=2$, the link is represented by the \textup(biregular\textup) Geiser involution on $X$.
\item 
Type \typem{IV}: $K_X^2=4$, $X=X_4\subset \PP^4$ is an intersection of two 
quadrics containing two pencils of conics. 
Typically, this link is not represented by a biregular involution. 
\end{enumerate}
\end{theorem}

\begin{scorollary}[\cite{Iskovskih1967e}, \cite{Iskovskikh-1970}, \cite{Iskovskikh-deg4-1973}, 
\cite{Iskovskikh-Factorization-1996e}]
Let $\pi: X\to B$ and $\pi': X'\to B'$ be standard conic bundles over rational curves
and let $\Phi: X\dashrightarrow X'$ be a birational map.
\begin{enumerate}
\item 
If $K_X^2\le 0$, then $\Phi$ is fiberwise birational.
\item 
If $K_X^2=1$, $2$, or $3$, then there exists a birational automorphism $\psi: X \dashrightarrow X$
such that $\Phi\comp \psi$ is fiberwise birational.
\end{enumerate}
\end{scorollary} 

\begin{scorollary}\label{corollary-surface-rigid}
Let $\pi: X\to B$ be a standard conic bundle over a rational curve.
\begin{enumerate}
\item 
If $K_X^2\le 3$, then $\pi: X\to B$
is birationally rigid. 
\item 
If $K_X^2\le 0$, then $\pi: X\to B$
is birationally \emph{superrigid}.
\item 
If $K_X^2=4$, then $X$ is not $\Bbbk$-rational.
\end{enumerate}
\end{scorollary}
Note that the condition $K_X^2\le 0$ is equivalent to $|4K_B+\Delta|\neq \emptyset$
(cf. Theorem~\xref{Sarkisov-theorem}).

It turns out that a surface with a structure of standard conic bundle
and positive self-intersection of the canonical class is ``almost''
a del Pezzo surface with a few exceptions:

\begin{proposition}[{\cite[Lemma~17]{Kollar-Mella-unirationality}}, see also \cite{Iskovskikh-1970},
\cite{Iskovskikh-deg4-1973},
\cite{Iskovskikh-1979s-e},
{\cite[\S 8]{Prokhorov-re-rat-surf}}]
\label{proposition-conic-bundle=del-Pezzo}
Let $\pi: X\to B$ be a surface standard conic bundle with $0<K_X^2<8$ and
let $\Lambda$ be a geometric fiber.
Then the divisor $-K_X$ is nef and big except for the following two cases
\begin{enumerate}
\item\label{case-conic-bundles-del-Pezzo-1}
$K_X^2=1$ and there exists a geometrically irreducible rational curve $C\in |-K_X-\Lambda|$
such that $C^2=-3$,
\item\label{case-conic-bundles-del-Pezzo-2}
$K_X^2=2$ and there exists a curve $C\in |-K_X-2\Lambda|$ which 
is a pair of conjugate disjoint sections $C_1$ and $C_2$ with 
$C_i^2=-3$.
\end{enumerate}
If furthermore $K_X^2>2$ and $K_X^2\neq 4$, then $-K_X$ is ample
with one exception:
\begin{enumerate}\setcounter{enumi}{2}
\item
\label{case-conic-bundles-del-Pezzo-4}
$K_X^2=4$ and
there exists a 
curve $C\in |-K_X-2\Lambda|$ which 
is a pair of conjugate disjoint sections $C_1$ and $C_2$ with 
$C_i^2=-2$.
\end{enumerate}
\end{proposition}

In the case~\xref{case-conic-bundles-del-Pezzo-4} contracting the sections
$C_1$ and $C_2$ we obtain a singular del Pezzo surface of degree $4$.
Its anticanonical image is (quite special) intersection of two quadrics in $\PP^4$.
This surface is called the \textit{Iskovskikh surface} \cite{Kunyavski-Skorobogatov-Tsfasman},
\cite{Coray-Tsfasman-1988}.

\begin{proof}
Put $d:=K_X^2$. 
By the Riemann-Roch formula $\dim |-K_X|\ge d>0$. 

Assume that $K_X\cdot C>0$ for some reduced irreducible curve $C$.
Then $C$ is a fixed component of $|-K_X|$. Hence, 
$C\sim -K_X-\beta\Lambda$ and so $C\cdot \Lambda=2$.
In particular, this implies that 
$C\otimes \bar \Bbbk$ has at most two components.
We have
\begin{equation*}
0>-K_X\cdot C=d-2\beta,\qquad 2\beta> d. 
\end{equation*}
If $C\otimes \bar \Bbbk$ is connected, then 
\begin{equation*}
-2\le 2\p(C)-2=(K_X+C)\cdot C=-2\beta, \qquad \beta=1,\quad \p(C)=0. 
\end{equation*}
In this case, $d=1$ and $C^2=-3$.
We obtain the case~\xref{case-conic-bundles-del-Pezzo-1}.

Assume that $C\otimes \bar \Bbbk$ has two connected components $C_1$ and $C_2$.
In this case, $\p(C_i)=0$, $C_i^2=\frac12d-2\beta$, and $-K_X\cdot C_i=\frac12 d-\beta$. In particular, $d$ is even and
$\beta\ge\frac12 d$. Then 
\begin{equation*}
-2=2\p(C_1)-2=(K_X+C_1)\cdot C_1=-\beta.
\end{equation*}
Thus $\beta=2=d$.
We obtain the case~\xref{case-conic-bundles-del-Pezzo-2}.

Now we consider the case where $-K_X$ is nef (but not ample).
Let $C$ be an irreducible curve such that $K_X\cdot C=0$.
Write $C\sim \alpha(-K_X)-\beta \Lambda$ 
and let 
$C^{(1)},\dots, C^{(k)}$ be connected components of $C\otimes \bar \Bbbk$.
Since every component $C^{(i)}$ is a tree of smooth rational curves, 
we have $(C^{(i)})^2=-2$ and $C^2=-2k$.
Further,
\begin{equation}
\label{case-conic-bundles-del-Pezzo-inequality}
0=-K_X\cdot C=d\alpha-2\beta, \quad -2k=C^2=\alpha(d\alpha-4\beta).
\end{equation}
This gives us 
$d\alpha=2\beta$, $-2k=\alpha(d\alpha-4\beta)$, and
$k=\alpha\beta$. 
Let $m:=C^{(i)}\cdot \Lambda$.
Then $mk=C\cdot \Lambda=2\alpha$. Hence, $m\beta=2$
and $d\alpha=4/m$. Since $d>2$ by our assumptions, 
the only possibility is $d=4$, $m=\alpha=1$, $\beta=2=k$.
Since $m=1$, each connected component $C^{(i)}$ is (geometrically)
irreducible. We obtain the case~\xref{case-conic-bundles-del-Pezzo-4}.
\end{proof}

Thus in the cases $K_X^2=3,\, 5,\, 6$ there exists an 
$K_X$-negative extremal ray on $X$ whose contraction $q: X\to X'$ is different 
from $\pi$ \cite{Mori-1982}. It is not hard to 
to find out all the possibilities for $q$:
\begin{theorem}[\cite{Iskovskikh-1979s-e}, \cite{Iskovskikh-Factorization-1996e}]
\label{classification-Sarkisov-links-conic-bundle}
Let $\pi: X\to B$ be a standard conic bundle, where 
$B$ is a curve. If $K_X^2\in \{3,\, 5,\, 6\}$, then there exists a 
birational contraction 
$q: X\to X'$, where $X'$ is a del Pezzo surface with $\uprho(X')=1$. 
More precisely, one of the following holds.
\begin{enumerate}
\item \label{classification-Sarkisov-links-conic-bundle-deg=3}
If $K_X^2=3$, then $K_{X'}^2=4$ and $q$ is a contraction of a geometrically
irreducible $(-1)$-curve. In this case $B\simeq \PP^1$ and the contraction $q: X\to X'\subset \PP^4$ is given by the linear system
$|-2K_{X}-C|$, where $C$ is the class of a fiber.

\item \label{classification-Sarkisov-links-conic-bundle-deg=5}
If $K_X^2=5$, then $X'\simeq \PP^2$ and $q$ is a contraction of four conjugate $(-1)$-curves.
In this case $B\simeq \PP^1$ and the contraction $q: X\to \PP^2$ is given by 
$|-K_{X}-C|$, where $C$ is the class of a fiber.

\item \label{classification-Sarkisov-links-conic-bundle-deg=6}
If $K_X^2=6$, then $K_{X'}^2=8$ and $q$ is a contraction of a pair of 
conjugate $(-1)$-curves.
The contraction $q: X\to X'\subset \PP^8$ is given by 
$|-2K_{X}+\pi^*K_B|$.
\end{enumerate}
\end{theorem}

\begin{corollary}
\label{corollary-surfaces-rationality}
Let $\pi: X\to B$ be a standard conic bundle, where $X$ is a geometrically rational surface, 
and $B$ is a curve. Let $\Delta\subset B$ be the discriminant locus.
\begin{enumerate} 
\item
$\deg\Delta\neq 1$.
\item 
If $\deg \Delta=3$, then $X$ is $\Bbbk$-rational.
\item 
If $X$ has a $\Bbbk$-point and either $\deg \Delta=2$ or $\Delta=\emptyset$, then $X$ is $\Bbbk$-rational.
\end{enumerate}
\end{corollary}

Putting the above results together one obtains the following criterion of rationality.
\begin{theorem}[see {\cite[Theorem~2.6]{Iskovskikh-Factorization-1996e}}]
\label{surfaces-rationality-criterion}
Let $X$ be a minimal rational surface over $\Bbbk$.
Then $X$ is $\Bbbk$-rational if and only if $K_X^2\ge 5$ and $X$ has a $\Bbbk$-point.
\end{theorem}

As a consequence of Corollary~\xref{corollary-surfaces-rationality} we have the following.

\begin{proposition}
\label{proposition-rationality-conic-bundles}
Let the ground field $\Bbbk$ be algebraically closed.
Let $\pi: X\to S$ be a conic bundle over a rational surface
with discriminant curve $\Delta\subset S$.
Assume that there exists a base point free pencil $\LLL$ of rational curves on $S$ such that 
$\LLL\cdot \Delta\le 3$. Then $X$ is rational. 
\end{proposition}

\begin{proof}
Consider the generic member $L_{\eta}$ of $\LLL$ over the residue 
field $\Bbbk(\eta)$ of the generic point $\eta\in \PP^1=\LLL$. 
Since $\Bbbk(\eta)$ is a $\mathrm{c}_1$-field, by 
Tsen's theorem, the curve $L_{\eta}$ is isomorphic to $\PP^1_{\Bbbk(\eta)}$.
Consider the generic surface $F_{\eta}:=\pi^{-1}(L_{\eta})$ over $\Bbbk(\eta)$.
Then the morphism, $\pi_\eta: F_{\eta} \to L_{\eta}$ defines a standard 
conic bundle structure on the surface $F_{\eta}$ over the non-closed 
field $\Bbbk(\eta)$. The discriminant divisor of the fibration $\pi_\eta$
is $\Delta_\eta:=\Delta\cdot L_{\eta}$ and $\deg \Delta_\eta=\LLL\cdot \Delta\le 3$.
Corollary~\xref{corollary-surfaces-rationality} implies that $F_{\eta}$ is rational over $\Bbbk(\eta)$.
This implies the rationality of $X$ over $\Bbbk$.
\end{proof}
From Proposition~\xref{proposition-Sarkisov-elementary-transformations} we immediately obtain the following.
\begin{scorollary}
\label{corolary-rationality-conic-bundles}
Let $\pi: X\to \PP^2$ be a conic bundle 
with discriminant curve $\Delta$ of degree $\deg \Delta\le 4$. Then $X$ is rational.
\end{scorollary}

\section{The Artin-Mumford invariant}
\label{section:AM}
A natural birational invariant of an algebraic variety is the
cohomological Brauer--Grothendieck group
\[
\Br(X):=H^2_{\et}(X,\mathbb G_{\mathrm m}),
\]
where $\mathbb G_{\mathrm m}$ is the sheaf of invertible regular functions
in the \'etale topology on $X$.
This is a periodic Abelian group
which can be expressed in the classical situation
of smooth projective varieties over~$\CC$
in terms of complex-analytic 
invariants of a variety and the torsion in $H^3(X,\ZZ)$.
In particular,
if $\dim X=3$ and $\pg(X)=0$, then 
\begin{equation*}
\Br (X)\simeq H^3(X,\ZZ)_{\tors}.
\end{equation*}
The birational invariance of the last
group can be established immediately
using the theorem on resolution of singularities by means of
blowups with smooth centers:
\begin{proposition}[\cite{Artin-Mumford-1972}]
The torsion subgroup $H^3(X,\ZZ)_{\tors}$ is a birational invariant of a 
complete non-singular complex variety $X$ of any dimension. In particular, $H^3(X,\ZZ)_{\tors}=0$ 
if $X$ is rational.
\end{proposition}
There is a version of this statement valid in characteristic $p>0$.
In this case one has to replace $H^3(X,\ZZ)$ with \'etale $l$-adic cohomology
group $H^3(X, \ZZ_l)$, where $l$ is different from $p$ \cite{Artin-Mumford-1972}. 

\begin{remark}
It is easy to see that $H^3(X,\ZZ)_{\tors}$ is also a stably birational invariant:
if $X_1\times \PP^{n_1}$ is birationally equivalent to $X_2\times \PP^{n_2}$, where $X_1$ and 
$X_2$ are smooth projective varieties, then 
\begin{equation*}
H^3(X_1,\ZZ)_{\tors}\simeq H^3(X_2,\ZZ)_{\tors}.
\end{equation*}
Indeed, by K\"unneth formula one has 
\begin{equation*}
H^3(X_i\times \PP^{n_i},\ZZ)_{\tors}\simeq H^3(X_i,\ZZ)_{\tors}.
\end{equation*}
\end{remark}

The birational invariance of
$H^3(X,\ZZ)_{\tors}$ was used
in \cite{Artin-Mumford-1972} to
construct counter\-examples to the L\"uroth problem
in any dimension~$\geq 3$. They first construct examples of
unirational but not rational three-dimensional varieties
over~$\CC$ among conic bundles over rational surfaces
with $H^3(X,\ZZ)_{\tors}\neq 0$. 

It turns out that the group $H^3(X,\ZZ)_{\tors}$ 
is effectively computable in the case of conic bundles:

\begin{theorem}[\cite{Artin-Mumford-1972}, {\cite[Th. 2]{Zagorskiui1977}}]
\label{theorem-Artin-Mumford-computation}
Let $\pi: X\to S$ be a standard conic bundle over a rational surface~$S$
with discriminant curve~$\Delta$. Then 
\begin{equation}
\label{equation-Artin-Mumford-computation}
H^3(X,\ZZ)_{\tors}\simeq \Br(X) \simeq (\ZZ/2\ZZ)^{c-1},
\end{equation}
where~$c$ is the number of connected components
of $\Delta$.
\end{theorem}

This assertion states, in particular, that the number of connected components of the discriminant curve is a birational invariant:

\begin{scorollary}
Let $\pi: X\to S$ and $\pi: X'\to S'$ be standard conic bundles
over rational surfaces and let $\Delta$ and $\Delta'$
be the corresponding discriminant curves.
If $X$ and $X'$ are birationally equivalent, then $\Delta$ and $\Delta'$
have the same number of connected components. 
\end{scorollary}

The above result shows that to construct examples of non-rational
three-dimensional varieties, it is sufficient
to specify local invariants \eqref{eq:local-inv} and a disconnected
curve~$\Delta$ on some rational surface $S$.

\begin{scorollary} 
Let $\pi: X\to S$ be a standard conic bundle over a rational surface~$S$
with discriminant curve~$\Delta$.
If $\Delta$ is disconnected, then $X$ is not rational. Moreover,
$X$ is not stably rational.
\end{scorollary}

The collection of local invariants \eqref{eq:local-inv} can be chosen in such a way that
the resulting conic bundle $X/S$ is unirational.
This construction together with
the results of \cite[Theorem~5.10]{Sarkisov-1982-e}
can also be used to
produce examples of unirational non-rational
conic bundles with trivial intermediate Jacobian and trivial $H^3(X,\ZZ)_{\tors}$
(see Theorem~\xref{Sarkisov-example-H3}).

\begin{example}
Let $C_1,\, C_2\subset\PP^2$ be smooth non-rational curves meeting each other transversely.
Let $\sigma: S\to \PP^2$ the blowup of the intersection points and let
$\Delta_i\subset S$ be the proper transform of $C_i$. Take \'etale double covers 
$\tilde \Delta_i\to \Delta_i$. According to Proposition~\xref{proposition-AM}
there exists a standard conic bundle $\pi: X\to S$ with discriminant curve
$\Delta=\Delta_1\cup \Delta_2$. By \eqref{equation-Artin-Mumford-computation}
the variety $X$ is not stably rational.
\end{example}

For smooth three-dimensional Fano varieties, the Brauer group is trivial. 
This follows from the classification (see \cite[Tables~\S\S~12.2-12.7]{IP99}). 
However, for Fano varieties with singularities (more precisely, for their 
desingularizations) it can be non-trivial, as, for example, in 
\cite{Artin-Mumford-1972} where the simplest examples of non-rational 
unirational three-dimensional varieties are given by double covers of $\PP^3$ 
ramified in singular quartics, that is, by three-dimensional Fano varieties with 
double singularities. 
This example is constructed as follows.
\begin{example}\label{example-Artin-Mumford}
Let $C=\{f(x_1,x_2,x_3)=0\}\subset \PP^2$ 
be a non-degenerate conic
and let $E_i=\{g_i(x_1,x_2,x_3) = 0\}\subset \PP^2$, $i=1,\, 2$
be smooth cubic curves such that 
\begin{eqnarray*}
E_i\cap C&=&\{ P_1^{(i)}, P_2^{(i)}, P_3^{(i)}\},\quad i=1,\, 2,
\\
E_1\cap E_2&=&\{Q_1,\dots,Q_9\}.
\end{eqnarray*}
where $P_1^{(1)}, \dots P_3^{(2)},Q_1,\dots,Q_9$ are mutually 
distinct points, $E_i$ meet $C$ tangentially at $P_1^{(i)}, P_2^{(i)}, P_3^{(i)}$, and $E_1$ meets $E_2$ 
transversally at $Q_1,\dots,Q_9$ (see Fig. \ref{ris}).
\begin{figure} 
 \begin{tikzpicture}[scale=0.8]
\draw[name path=Q, thick,black] (0,0) ellipse (159pt and 129pt);
\draw[name path=E1, thick,black] plot [smooth,tension=1.1] coordinates{(0,3) (-4.5,2) (5.6,0) (-4.5,-2) (0,-3)} node[below]{};
\draw[name path=E2, thick,black] plot [smooth,tension=1.1] coordinates{(3,0) (2,-3.9) (0,4.55) (-2,-3.9) (-3,0)} node[below]{};
\fill[black,name intersections={of=E2 and E1}] 
(intersection-1) circle (2pt) node[above,yshift=0pt,xshift=8pt]{$Q_1$} 
(intersection-2) circle (2pt) node[above,yshift=0pt,xshift=9pt]{$Q_3$}
(intersection-3) circle (2pt) node[above,yshift=0pt,xshift=9pt]{$Q_2$}
(intersection-4) circle (2pt) node[above,yshift=0pt,xshift=10pt]{$Q_4$}
(intersection-5) circle (2pt) node[above,yshift=0pt,xshift=9pt]{$Q_5$}
(intersection-6) circle (2pt) node[above,yshift=0pt,xshift=9pt]{$Q_6$}
(intersection-7) circle (2pt) node[above,yshift=0pt,xshift=10pt]{$Q_7$}
(intersection-8) circle (2pt) node[above,yshift=0pt,xshift=9pt]{$Q_9$}
(intersection-9) circle (2pt) node[above,yshift=0pt,xshift=9pt]{$Q_8$}
;
\path [name intersections={of=Q and E1}] ;
\fill[black] (intersection-1) circle (2pt) node[above,yshift=0pt,xshift=10pt]{$P_1^{(1)}$} ;
\fill[black] (intersection-2) circle (2pt) node[left,yshift=0pt,xshift=0pt]{$P_2^{(1)}$} ;
\fill[black] (intersection-4) circle (2pt) node[left,yshift=0pt,xshift=0pt]{$P_3^{(1)}$} ;
\path [name intersections={of=Q and E2}] ;
\fill[black] (intersection-1) circle (2pt) node[above,yshift=0pt,xshift=10pt]{$P_1^{(2)}$} ;
\fill[black] (intersection-4) circle (2pt) node[below,yshift=0pt,xshift=8pt]{$P_2^{(2)}$} ;
\fill[black] (intersection-5) circle (2pt) node[below,yshift=0pt,xshift=8pt]{$P_3^{(2)}$} ;
\end{tikzpicture}
\caption{}\label{ris}
\end{figure}
Since curves of degree $3$ cut out a complete linear system on $C$, we have
\begin{equation*}
\sum_{i, j}P_j^{(i)} =B|_ C
\end{equation*}
for some third cubic curve $B$, that is, $(E_1 + E_2)|_ C=2B|_ C$, as a cycle on $C$.
Let $h(x_1,x_2,x_3) = 0$ be an
equation of $B$. 
For suitable choice of 
$h$ (up to scalar multiplication) we can write 
\begin{equation*}
g_1 g_2 = h^2-4f s 
\end{equation*}
for some $s=s(x_1,x_2,x_3)$ of degree $4$. Let $S\subset \PP^3_{x_0,\dots,x_3}$
be the quartic surface given by a homogeneous polynomial: 
\begin{equation*}
f (x_1,x_2,x_3)x_0^2 + h(x_1,x_2,x_3)x_0 + s (x_1,x_2,x_3) = 0.
\end{equation*}
It is easy to see that $S$ has exactly $10$ nodes 
and no other singularities. 
Next, let $Y$ be the double cover of $\PP^3$ (the ``double solid'') branched
over $S$. Clearly, $Y$ can be given in the weighted projective space 
$\PP(1,1,1,1,2)$ by the weighted homogeneous polynomial: 
\begin{equation*}
x_4^2+ f x_0^2 + h x_0 + s=0,
\quad \deg x_0 = \dots = \deg x_3 = 1, \quad \deg x_4 = 2.
\end{equation*}
Let $\tilde Y\to Y$ be the blowup of the point $O:=(0{:}0{:}0{:}1{:}0)$ and let 
$\pi_Y: Y\to \PP^2$ be the morphism induced by the projection from $O$. Then 
$\pi_Y$ is a singular conic bundle whose discriminant locus $\Delta_Y\subset \PP^2$
is $E_1\cup E_2$. A standard form of $\pi_Y$ is a conic bundle over the surface $S$
which is the blowup of $\PP^2$ in $Q_1,\dots,Q_9$ and the corresponding discriminant curve is 
a disjoint union of proper transforms of $E_1$ and $E_2$.
\end{example}

This example was generalized in 
\cite{Iliev-Katzarkov-Przyjalkowski2014} and \cite{PrzyjalkowskiShramov2016} (see also \cite{Beauville:sds}) to some other types of singular Fano 
threefolds.

The birational invariance of the cohomological Brauer was used also to construct 
counter\-examples to the E.~Noether problem of the rationality of fields of 
invariants for finite linear groups operating on vector spaces over an 
algebraically closed field \cite{Saltman-1984}, \cite{Bogomolov1987}, 
\cite{Shafarevich1991}. The problem can be formulated in geometrical terms as 
follows. Let $G$ be a finite group acting by linear transformations on a 
projective space $\PP^n$. The problem is whether the quotient $X=\PP^n/G$ is 
rational. See \cite{colliotthelene-Sansuc-2005}, \cite{Prokhorov2010}, 
\cite{Kang-survey-2001} for surveys on this subject. 			 

\section{Intermediate Jacobians and Prym varieties}
\label{section:J}
Let $X$ be a smooth projective three-dimensional variety over $\CC$ such that 
$H^1(X,{\OOO}_X)=H^3(X,{\OOO}_X)=0$. For example, this holds for rationally 
connected varieties. Then we have the Hodge decomposition in the following 
form:
\begin{equation*}
H^3(X,\CC)=H^{2,1}(X)\oplus H^{1,2}(X).
\end{equation*}
Therefore the integration of $(2,1)$-forms over $3$-dimensional cycles 
determines an embedding $H_3(X,\ZZ)/(\mbox{torsion})\hookrightarrow 
H^{2,1}(X)^\vee$ as a full rank lattice into the complex vector space 
$H^{2,1}(X)^\vee$. The alternating integral-valued intersection form on 
$H^3(X,\ZZ)$ is unimodular by Poincar\'e duality, and the corresponding 
Hermitian form on $H^{2,1}(X)^\vee$ is positive definite. Therefore, by the 
Riemann--Frobenius criterion, the complex torus 
\begin{equation*}
\J(X):=H^{2,1}(X)^\vee/\image(H^3(X,\ZZ))
\end{equation*}
is a principally polarized Abelian 
variety with polarization divisor $\Theta$.

\begin{definition}
\label{8.1.1}
In the preceding notation, the pair
$(\J(X),\Theta)$ is called the
\emph{intermediate Jacobian} \index{Intermediate Jacobian} of the variety $X$.
\end{definition}
Applications of intermediate Jacobians in
birational geometry are based on the following two
observations.

\begin{proposition}
\label{proposition-jacobian}
\begin{enumerate}
\item
Every principally polarized Abelian variety
$(A,\Theta)$ can be decomposed in a unique way
into a direct product of principally polarized
simple Abelian varieties
\begin{equation}\label{decomposition-Abelian-variety}
(A,\Theta)=\mathop\bigoplus\limits_{i=1}^m(A,\Theta_i).
\end{equation}
\item \textup(see \cite{Clemens-Griffiths} and also \cite[Lecture~1]{Tyurin-5-lectures-1973}, \cite[Proposition~1.2.8]{Tyurin-middle-Jacobian-e}\textup)
Let $X'\rightarrow X$ be the blow-up with a
smooth irreducible center $C\subset X$ \textup(that is, $C$ is either a
curve or a point\textup). Then there is the following isomorphism of principally 
polarized Abelian varieties
$$
\J(X')\simeq\begin{cases}\J(X), & \text{if $C$ is a point,}\\
\J(X)\oplus\J(C), & \text{if $C$ is a curve,}
\end{cases}
$$
where $\J (C)$ is the Jacobian of the curve $C$ with
the principal polarization determined by the
Poincar\'e divisor $\Theta$.
\end{enumerate}
\end{proposition}

This proposition and Hironaka's theorem on resolution of indeterminacies of rational maps 
immediately imply the birational invariance of that part (the so-called 
\textit{Griffiths component}) of the decomposition 
\eqref{decomposition-Abelian-variety} which is the product of those factors which are 
not Jacobians of curves. In other words, if one denotes by $\JG(X)$ the product of the 
components in \eqref{decomposition-Abelian-variety} which are not Jacobians of 
curves, then $\JG(X)$ is a birational invariant. In particular, we have the 
following
\begin{corollary}\label{corollary-J}
If the variety $X$ is rational, then $\JG(X)=0$.
\end{corollary}
Thus, to prove the non-rationality of some three-dimensional varieties 
one should be able to distinguish 
principally polarized Abelian varieties from Jacobians of curves.

It turns out that 
intermediate Jacobians of three-dimensional varieties 
having a structure of standard conic bundle can be described as 
so called \textit{Prym varieties}.
These varieties were studied first
in papers of F.~Prym, W.~Wirtinger, F.~Schottky, and H.~Jung
in connection with the Schottky
problem of distinguishing the Jacobians of curves in
the moduli spaces of principally polarized Abelian varieties
(see \cite{Beauville1977a}, \cite{Clemens1975}).

D.~Mumford was the first who drew attention to the possibility of applications 
of Prym varieties to the birational geometry of three-dimensional varieties in 
the appendix to the paper of \cite{Clemens-Griffiths}. He also studied double covers 
$\tilde\Delta\to\Delta$ of smooth curves $\Delta$ from the point of view of 
distinguishing their Prym varieties $\Pr(\tilde\Delta,\Delta)$ from Jacobians of 
curves \cite{Mumford1974}. Mumford's results were extended to singular curves 
with normal crossings in the subsequent papers \cite{Beauville1977}, \cite{Beauville1977a}, \cite{Shokurov1981/82},
\cite{Shokurov1983}.

\begin{definition}
\label{8.1.4}
Let $(\tilde\Delta,\tau)$ be the pair consisting of a
complete reduced (possibly reducible) curve
$\tilde\Delta$ with at most ordinary double points and
an involution~$\tau$ on~$\tilde\Delta$
(that is, $\tau^2=\id$) acting non-trivially
on every irreducible component of the curve $\tilde\Delta$.
Denote by $\Delta$ the quotient $\tilde\Delta/\tau$, and by
$\tilde\pi: \tilde\Delta\to\Delta$ the corresponding
double cover.
The involution $\tau$ induces an involution $\tau^*$
on the Jacobian 
$\J (\tilde\Delta)$, and $\pi$ induces
the norm map
$\Nm: \J (\tilde\Delta)\to\J (\Delta)$, where $\tilde\pi^*\comp \Nm:= \id+\tau^*$ and $\Nm\comp \tilde\pi^*=2$.
The connected commutative algebraic
group
\begin{equation*}
\Pr(\tilde\Delta,\Delta):=\ker (\Nm)^0=(\id-\tau^*)\J(\tilde\Delta)
\end{equation*}
is called the (generalized) \emph{Prym variety} of the
pair $(\tilde\Delta,\tau)$, where 
the symbol ``$\;^{0}$'' denotes the identity component.
It is easy to see that 
\begin{equation*}
\dim\Prym(\tilde\Delta,\Delta)=\p (\Delta)-1.
\end{equation*}
\end{definition}

\begin{scase}
\label{Prym-polarization}
The notion of polarization can be easily extended
to generalized Prym varieties
(being algebraic groups, they are not necessarily Abelian varieties
when $\Delta$ is singular). Under certain conditions,
$\Pr(\tilde\Delta,\Delta)$ is an Abelian variety
the polarization of which is divisible by~$2$ and after
the division becomes principal. This is true, for example,
if the Beauville conditions \eqref{Beauville-condition} are satisfied (see \cite{Beauville1977a} or \cite[Theorem 3.5]{Shokurov1983}).
For applications one has to impose a stronger condition:
\begin{equation}
\label{equation-condition-S-star}
\begin{tabularx}{0.8\textwidth}{X}
for every decomposition
$\Delta=\Delta_1+ \Delta_2$ with $\Delta_i> 0$
\\
we have $\#(\Delta_1\cap\Delta_2)\geq 4$.
\end{tabularx}
\end{equation}
The condition \eqref{equation-condition-S-star} implies that $\Delta$ is a stable curve.
In particular, if a curve $\Delta$ satisfies \eqref{equation-condition-S-star}, then 
the canonical linear system $|\upomega_\Delta|$ is base point free 
and defines the canonical morphism $\Delta \to \PP^{\p(\Delta)-1}$ which is finite onto its
image
(see \cite{Catanese1999}).
\end{scase}

The following theorem was proved by Shokurov \cite{Shokurov1983}. For weaker versions see
\cite{Mumford1974}, \cite{Beauville1977a}, \cite{Debarre1989}.
\begin{theorem}
\label{theorem-Mumford-Beauville-Shokurov}
Let $(\tilde\Delta,\tau)$ be a pair consisting of a curve
$\tilde\Delta$ of arithmetic genus $2g-1$, 
and an involution $\tau$ on $\tilde\Delta$
satisfying the Beauville conditions \eqref{Beauville-condition}, as well as, 
the condition \eqref{equation-condition-S-star}.
Let $\Delta=\tilde\Delta/\tau$ be the quotient curve, and
let $\Pr(\tilde\Delta,\Delta)$ be the corresponding
Prym variety with
the polarization divisor $\Xi$ so that $\p(\Delta)=g$ and $\dim \Pr(\tilde\Delta,\Delta)=g-1$.
Then
$(\Pr(\tilde\Delta,\Delta),\Xi)$ is the Jacobian of a curve or a product
of Jacobians of curves only in the following cases:
\begin{enumerate}
\item
$\Delta$ is a hyperelliptic curve;
\item
$\Delta$ is a trigonal curve \textup(this case is considered  in \cite{Recillas1974}\textup);

\item
$\Delta$ is a quasi-trigonal curve;

\item
$\Delta$ is a plane  quintic curve and the
corresponding double cover is given by an even theta-characteristic.
\end{enumerate}
\end{theorem}

Here, as usual, a (not necessarily smooth) curve $\Delta$ is said to be \textit{hyperelliptic} (resp. \textit{trigonal})
if there exists a finite morphism $\Delta\to \PP^1$ of degree $2$ (resp. $3$). 
A curve $\Delta$ is said to be \textit{quasi-trigonal} if obtained from a hyperelliptic curve 
by identifying two smooth points.

\begin{sremark}
\label{remark:Delta-can}
The canonical map $\phi_{\Delta}: \Delta\to \PP^{g-1}$ of the curves listed in 
Theorem \ref{theorem-Mumford-Beauville-Shokurov} satisfy the following properties.
\begin{enumerate} 
\item
If the curve $\Delta$ is hyperelliptic, then the map $\phi_{\Delta}$ is
a finite morphism of degree $2$. 

\item 
If $\Delta$  is a trigonal non-hyperelliptic curve,  then  the map $\phi_{\Delta}$ is an embedding 
and its image  has a one-dimensional family of $3$-secant lines whose intersections with $\phi_{\Delta}(\Delta)$ generate
the linear series $\mathfrak g^1_3$. 

\item 
If $\Delta$  is a  quasi-trigonal non-hyperelliptic curve,  then  the map $\phi_{\Delta}$ is an embedding 
and its image lies on a two-dimensional cone with vertex at a singular point.
Here
the generators of the cone  are $3$-secant lines of the curve $\phi_{\Delta}(\Delta)$. 

\item 
If $\Delta$  is a  plane quintic,  then  $\phi_{\Delta}$ is an embedding  
and its image lies on the Veronese surface.
\end{enumerate}
\end{sremark}

\begin{sremark}
\label{remark:Delta:Prym}
Assume that $\tilde\Delta_1\cap \tilde\Delta_2=\{P_1,P_2\}$
and let $\tilde\Delta_i'$, $i=1,2$ be curves
obtained from $\tilde\Delta_i$ by identifying the points
$P_1$ and $P_2$ (see Fig \ref{risss}). 
\begin{figure}[h]
 \begin{tikzpicture}
\draw [very thick,black,->,>=stealth] plot [smooth,tension=0.9] coordinates{ (1.9,0) (2.25,0.05)(2.5,-0.05) (2.9,0)};
\draw[name path=C2, thick,black] plot [smooth,tension=1] coordinates{(-1.5,1.1) (0,0.55) (-0.1,-0.95) (-1.4,-1)}  node[above,yshift=4pt, xshift=0pt]{$\scriptstyle \tilde \Delta_1$};
\draw[name path=C1, thick,black] plot [smooth,tension=1] coordinates{(1.5,1.1) (-0,0.95)  (-0.1,-0.95) (1.4,-1)} node[above,yshift=4pt, xshift=0pt]{$\scriptstyle \tilde \Delta_2$};
\draw[name path=C11, thick,black] plot [smooth,tension=0.9] coordinates{(4,1.1)(4.5,0)  (5.4,-0.55)  (5.7,0) (5.2,0.95)(4.5,0) (4,-1)} node[below,yshift=-2pt, xshift=0pt]{$\scriptstyle \tilde \Delta_1'$};
\draw[name path=C22,thick,black] plot  [smooth,tension=1] coordinates{(9,1.1)(8.3,0.2) (7.2,-0.55)  (6.7,0)  (7.2,0.95) (8.3,0.2)(9,-1)} node[below,yshift=-2pt, xshift=0pt]{$\scriptstyle \tilde \Delta_2'$};
\fill[black,name intersections={of=C1 and C2}](intersection-1) circle(2pt) node[above,yshift=4pt, xshift=0pt]{$\scriptstyle P_1$};
\fill[black,name intersections={of=C1 and C2}](intersection-2) circle(2pt) node[below,yshift=-4pt, xshift=0pt]{$\scriptstyle P_2$};
\fill[black](4.5,0) circle(2pt) node[left,yshift=-4pt, xshift=0pt]{$\scriptstyle P_1=P_2$};
\fill[black](8.3,0.2) circle(2pt) node[right,yshift=-4pt, xshift=0pt]{$\scriptstyle P_1=P_2$};
\end{tikzpicture}
\caption{}
\label{risss}
\end{figure}
\newline\noindent
Then $\Pr(\tilde \Delta,\Delta)=\Pr(\tilde \Delta_1',\Delta_1')\times \Pr(\tilde \Delta_2',\Delta_2')$,
where $\Pr(\tilde \Delta_i',\Delta_i')$ are Prym varieties for $\tilde\Delta_i'$ with the induced
involution (see \cite[Corollary 3.16]{Shokurov1983}). This is the reason for
the condition~\eqref{equation-condition-S-star} to be included in the hypothesis of the theorem.
\end{sremark}

Prym varieties are used in the three-dimensional birational geometry
mainly in the following situation.
Let $\pi: X\rightarrow S$ be a standard conic bundle over
a smooth projective rational surface.
Let $\Delta\subset S$ be the discriminant curve of $\pi$, and let
$\tilde\Delta$ be the curve parameterizing
irreducible components
of degenerate conics over $\Delta$.
Then if $\Delta\neq \emptyset$, then $\Delta$ is a reduced curve
with normal crossings, and $\pi$ induces a double cover
$\tilde\pi: \tilde\Delta\to\Delta$
satisfying Beauville conditions
\eqref{Beauville-condition}.

For the following theorem, see \cite[Ch.~3]{Beauville1977}, \cite{Beltrametti-1985}, \cite{Beauville1989}.
\begin{theorem}
\label{theorem-Prym-Jacobian}
Let $\pi: X\rightarrow S$ be a standard
conic bundle over a rational surface $S$ with discriminant
curve $\Delta\subset S$ and let $\tilde\Delta\to \Delta$ be the 
corresponding double cover. Then the intermediate Jacobian
$\J(X)$ is isomorphic as a
principally polarized abelian variety to the Prym variety
$\Pr(\tilde\Delta,\Delta)$. 
\end{theorem} 
 
A.~Beauville considered the case $S=\PP^2$ and obtained as a consequence of 
Theorems~\xref{theorem-Mumford-Beauville-Shokurov} and~\xref{theorem-Prym-Jacobian} 
the following statement \cite[Th\'eor\`eme~4.9]{Beauville1977}: if $S=\PP^2$, and $\deg 
\Delta\geq 6$, then the variety $X$ is non-rational since in this case 
$\Pr(\tilde\Delta,\Delta)$ is neither the Jacobian of a curve nor a product of 
Jacobians of curves.

Intermediate Jacobian was used to prove non-rationality of many types 
of (even singular) Fano threefolds, see e.g. \cite[Ch.~5]{Beauville1977}, \cite{Clemens1983},
\cite{Conte1977a,Conte1977b}, \cite{Przhiyalkovskij-Cheltsov-Shramov-2015},
\cite{PiccoBotta-Verra-1983}, \cite{Ehndryushka1985}, \cite{Alzati-Bertolini-1992a}.

\section{Birational transformations}
\label{section:bir}
The following result, based on the construction of \cite{Tjurin1975}, was obtained by I. Panin in \cite{Panin1980}. 
It is related to the exceptional case~\xref{conjecture-Iskovskikh-2} of Conjecture
\xref{conjecture-Iskovskikh}.

\begin{proposition}[\cite{Panin1980}]
\label{proposition-Panin}
Let $\pi: X\to \PP^2$ be a standard 
conic bundle with discriminant curve $\Delta\subset \PP^2$ of degree $5$.
Suppose that the corresponding double cover $\tilde\pi: \tilde\Delta\to \Delta$ associated to $\pi$ is given by an even 
theta-characteristic. Then $X$ is rational. Moreover, $\pi: X\to \PP^2$ is fiberwise birational to
the conic bundle obtained by blowing up of $\PP^3$ along a smooth curve of degree $7$ and genus $5$,
see Example~\xref{example-Panin}.
\end{proposition}

\begin{proof}[Sketch of the proof.]
According to Lemma~\xref{proposition-AM}, such a conic bundle exists and 
any two of them are birationally equivalent over $\PP^2$. 
Thus it is sufficient to prove the rationality of a particular one.
Since $\upomega_\Delta=\OOO_{\Delta}(2)$, the curve $\Delta$ is neither hyperelliptic nor (quasi-)trigonal
because its canonical 
model $\phi_{K_{\Delta}}(\Delta)\subset \PP^5$ lies in the Veronese surface.
As $\tilde\pi$ satisfies the conditions \eqref{Beauville-condition} 
the Prym variety $\Pr(\tilde\Delta, \Delta)$ is a principally polarized 
Abelian variety (see~\xref{Prym-polarization}). Let $\Xi$ be the polarization divisor.
By \cite[\S 7]{Mumford1974}, \cite[Prop. 5.4]{Shokurov1983} \ $\dim \Sing(\Xi)\le 1$.
Therefore, $\Pr(\tilde\Delta, \Delta)$ is not a product of two or more principally polarized 
Abelian varieties.
Thus the 
Prym variety $\Pr(\tilde\Delta, \Delta)$ is (by 
\cite{Masiewicki1976} and \cite{Tjurin1975}) 
isomorphic as a principally polarized 
Abelian variety to the Jacobian $\J(\Gamma)$ of a smooth curve of genus $5$. 
One can show that the curve 
$\Gamma$ cannot be hyperelliptic nor 
trigonal (see \cite{Shokurov1983}, \cite{Masiewicki1976}, \cite{Tjurin1975}). 
Thus the canonical image $\Gamma\subset \PP^4$ is a complete intersection of three 
quadrics $Q_1\cap Q_2\cap Q_3$. In other words, $\Gamma$ is the base locus of a net of quadrics $|2H-\Gamma |$, where 
$H$ is a 
hyperplane in $\PP^4$. 

\begin{sdefinition}[\cite{Tjurin1975}, \cite{Tyurin1976}]
Let $\mathcal C$ be a net of quadrics in $\PP^{2m}$ such that the base locus 
$\Bs(\mathcal C)$ is a smooth variety. Then a general quadric in the net $\mathcal C$ is smooth
and degenerate quadrics in $\mathcal C$ are parametrized by a curve $\Delta\subset 
\mathcal C\simeq \PP^2$. A general quadric $Q_s$ for $s\in \Delta$ has two families 
of $m$-dimensional linear subspaces. This defines a double cover $\tilde \Delta\to \Delta$
ramified exactly over singular points of $\Delta$. This double cover is called the \textit{invariant of the net}
$\mathcal C$.
\end{sdefinition}

\begin{stheorem}[\cite{Tjurin1975}, {\cite[Th\'eor\`eme~6.3]{Beauville1977}}]
\label{net-of-quadrics}
Let $\mathcal C$ be a net of quadrics in $\PP^{4}$ such that the base locus 
$\Bs(\mathcal C)$ is a smooth curve $\Gamma$. Let $\tilde \Delta\to \Delta$ be the 
invariant of the net $\mathcal C$. Then the Prym variety 
$\Pr(\tilde \Delta,\Delta)$ is isomorphic to the Jacobian $\J(\Gamma)$ of $\Gamma$ as a principally
polarized Abelian variety.
\end{stheorem}

Now starting from this net of quadrics $|2H-\Gamma |$ we build a standard conic bundle 
$\pi: X\to\PP^2$ with the given double cover $\tilde\Delta\to \Delta$ and prove the rationality of 
$X$.

\begin{slemma}
Let $\Gamma\subset \PP^4$ be a smooth canonical non-trigonal curve of genus $5$
and let $\mathcal C$ be the net of quadrics passing through $\Gamma$.
Let $\tilde \Delta \to \Delta$ be the invariant of the net $\mathcal C$.
Fix a point $P\in \Gamma$. Let $\Gamma_0\subset \PP^3$ be the projection of 
$\Gamma$ from $P$ and let $X\to \PP^3$ be the blowup of $\Gamma_0$.
Then the discriminant curve of the standard conic bundle $\pi: X\to \PP^2$ 
constructed in Example~\xref{example-Panin} is isomorphic to $\Delta$ 
and the corresponding double cover coincides with $\tilde \Delta \to \Delta$.
\end{slemma}

\begin{proof}[Sketch of the proof.]
Let $T_{P,\PP^4}$ be the tangent space to $\PP^4$ at $P$.
We identify the target of the projection $\PP^4 \dashrightarrow\PP^3$ from $P$ 
with the projectivization $\PP(T_{P,\PP^4})$.
Consider the following family of conics in $\PP(T_{P,\PP^4})$.
For each points $s\in \mathcal C\simeq \PP^2$ we consider the conic $C_s$ which is the 
projectivized tangent cone to the quadric $Q_s\cap T_{P,\PP^4}$ at the point $P$. Since $\Gamma$ 
is a smooth complete intersection of the members $Q_s$, $s\in \mathcal C$, 
none of the quadrics $Q_s$ has singular points on $\Gamma$, in 
particular at $P$. Therefore, we have one of the following possibilities:
\begin{eqnarray*}
\corank (Q_s)=0 &\Longleftrightarrow& \text{$C_s$ is smooth,}
\\
\corank (Q_s)=1 &\Longleftrightarrow& C_s=\PP^1 \vee \PP^1,
\\
\corank (Q_s)=2 &\Longleftrightarrow& \text{$C_s$ is a double line.}
\end{eqnarray*}
Moreover, each $C_s$ meets $\Gamma_0$ at $8-2=6$ points.

Conversely, if $C\subset T_{P,\PP^4}$ is a (possibly degenerate) conic
meeting $\Gamma_0$ at $6$ points, then there is a quadric $Q\in \mathcal C$
such that $C$ is the 
projectivized tangent cone to $Q$ at $P$.
As in Example~\xref{example-Panin}, by blowing up $\Gamma_0$ we obtain a conic bundle 
$\pi: X\to \PP^2$ whose fibers are proper transforms of conics $C_s$.
In particular, the degeneracies of $\pi: X\to \PP^2$ 
correspond precisely to the given cover $\tilde\pi:\tilde\Delta\to \Delta$ from which the net of quadrics was 
built. 
\end{proof}
By Theorems~\xref{net-of-quadrics} and~\xref{theorem-Prym-Jacobian} we have
\begin{equation*}
\Pr(\tilde \Delta,\Delta)\simeq \J(\Gamma)\simeq \J(X).
\end{equation*}
Since $X$ is rational, Proposition~\xref{proposition-Panin} follows.
\end{proof}

\begin{proposition}[{\cite[Proposition 2.4]{Sarkisov-1980-1981-e}}]
\label{proposition-Sarkisov-elementary-transformations}
Let $\pi: X\to S$ be a standard conic bundle. Let $\alpha: S_{1}\to S$ be the
blowup with center at a point $s\in S$. Then there exists a standard conic bundle
$\pi_{1}: X_{1}\to S_{1}$ and a birational map $\psi: X_{1}\dashrightarrow X$ such that the diagram
\begin{equation}
\label{equation-diagram-elementary-transformations-0}
\vcenter{
\xymatrix{
X\ar[d]^{\pi} & X_{1}\ar[d]^{\pi_{1}}\ar@{-->}[l]_{\psi}
\\
S & S_{1}\ar[l]_{\alpha}
} }
\end{equation}
is commutative. More precisely, if $\pi^{-1}(s)$ is a smooth conic, then 
$\psi$ is regular and is the blowup of $\pi^{-1}(s)$. If $\pi^{-1}(s)$ is a degenerate conic, then 
\eqref{equation-diagram-elementary-transformations-0}
can be completed to the following commutative diagram \textup(cf.~\xref{definition-Sarkisov-links}\textup)
\begin{equation}
\label{equation-diagram-elementary-transformations}
\vcenter{
\xymatrix@R=9pt{
Z\ar[d]^{p}\ar@{-->}[r]^{\chi}& X_{1}\ar[d]^{\pi_{1}}
\\
X\ar[d]^{\pi}&S_{1}\ar[dl]^{\alpha}&
\\
S&
} }
\end{equation}
Here $p$ is the blowup of a reduced irreducible component of $\pi^{-1}(s)$
and $\chi$ is a flop.
\end{proposition}

\begin{proof}[Sketch of the proof.]
Let $\Delta\subset S$ be the discriminant curve and let $l:=\pi^{-1}(s)_{\red}$
be the fiber over $s$.
There are three possibilities: 
\begin{equation*}
s\in S\setminus\Delta,\quad s\in\Delta\setminus\Sing(\Delta),\quad
\text{and}\quad s\in\Sing(\Delta).
\end{equation*}
\begin{scase}\label{elementary-transformations-1} 
\textbf{Case $s\in S\setminus\Delta$.}
Then we can take $X_{1}=X\times_{S} S_{1}$. In other words, $\psi: X_{1}\to X$ is the blowup
with center $l$. Since the curve $l$ is smooth, $X_{1}$
is a smooth variety. Moreover, $\psi$ is a morphism in this case.
\end{scase}

\begin{scase}\label{elementary-transformations-2} 
\textbf{Case $s\in\Delta\setminus\Sing(\Delta)$.}
Then $\pi^{-1}(s)=l=l'\vee l''$. Let $p: Z\to X$ be the blowup of $l'$
and let $\bar{l}$ be the proper transform of $l''$. For the normal bundle of 
$\bar{l}$ in $Z$ we have 
\begin{equation*}
\NNN_{\bar{l}/Z}=\OOO_{\PP^1}(-1)\oplus \OOO_{\PP^1}(-1).
\end{equation*}
Therefore, there exists an Atiyah-Kulikov flop $Z \dashrightarrow X_{1}$ with center $\bar{l}$.
The anti-canonical divisors of $Z$ and $X_{1}$ are relatively nef over $S$.
According to general theory \cite{Mori-1982} there exists a Mori
extremal contraction $\pi_{1}: X_{1}\to S_{1}$ over $S$ that
can be completed to the diagram \eqref{equation-diagram-elementary-transformations},
where $\pi_{1}$ is a standard conic bundle and $\alpha$ is a blowup of $s$.
\end{scase}

\begin{scase}\label{elementary-transformations-3} 
\textbf{Case $s\in\Sing(\Delta)$.}
Then $\pi^{-1}(s)=2l$ is a double line.
Let $p: Z\to X$ be the blowup of $l$.
As above the anti-canonical divisor $-K_{Z}$ is nef over $S$ and 
there exists a flop $Z\dashrightarrow X_{1}$ at the negative section of 
the exceptional divisor $p^{-1}(l)\simeq \FF_3$. Thus we obtain
the diagram \eqref{equation-diagram-elementary-transformations}.\qedhere
\end{scase}
\end{proof}

The transformations described in~\xref{elementary-transformations-1}-
\xref{elementary-transformations-3} 
are called \textit{elementary transformations} of conic bundles.
They are simplest examples of Sarkisov links.

\begin{lemma}[\cite{Iskovskikh-1996-conic-re}]
\label{Lemma-3}
Let $S$ and $S'$ be surfaces with rational singularities and let $\alpha: S\to S'$ be 
a birational contraction. 
Let $\Delta\subset S$ be a connected 
reduced curve and let $\Delta':=\alpha(\Delta)$.
Then $\p(\Delta)\le \p(\Delta')$.
\end{lemma}

\begin{lemma}[{\cite[Lemma~1]{Iskovskikh-1991-conic-re}}]
\label{Lemma:nK+Delta:invariant}
Let $\pi: X\to S$ and $\pi': X'\to S'$ 
be fiberwise birational equivalent standard conic bundles over smooth rational surfaces
and let $\Delta\subset S$ and $\Delta'\subset S'$ be 
the corresponding discriminant curves. Then we have
\begin{eqnarray}
\label{eqnarray:mK+D}\qquad
\dim |mK_S+n\Delta| &=& \dim |mK_{S'}+n\Delta'|\quad 
\forall m\ge n>0,
\\
\label{eqnarray:D}
\p(\Delta)&=&\p(\Delta'),
\\\label{eqnarray:lemma-J-2-conic-bundles}
\bb_3(X)&=&\bb_3(X').
\end{eqnarray}
Moreover, $\Delta$ is smooth if and only if so $\Delta'$ is.
\end{lemma}
Note that \eqref{eqnarray:mK+D} and \eqref{eqnarray:D} fail in the category $\QQ$-conic bundles,
see Example~\xref{example-simplest-link}.
\begin{proof}
First we note that elementary birational transformations as in Proposition~\xref
{proposition-Sarkisov-elementary-transformations} do not change $\p(\Delta)$ as well as 
$|mK_S+n\Delta|$.
Indeed, let $\alpha: S'\to S$ be the blowing up of $s\in S$ and let $\pi': X'\to S'$ be 
a standard conic bundle 
obtained from $\pi$ by one of the elementary transformations~\xref{elementary-transformations-1}-
\xref{elementary-transformations-3}. Then 
\begin{equation*}
\Delta'=
\begin{cases}
\alpha^*\Delta &\text{if $s\in S\setminus\Delta$},
\\
\alpha^*\Delta -E &\text{if $s\in\Delta$},
\end{cases}
\end{equation*}
where $E=\alpha^{-1}(s)$ is an exceptional divisor. Since $K_{S'}=\alpha^*K_S+E$, we have
\begin{equation*}
mK_{S'}+n\Delta'=\alpha^*(mK_X+n\Delta)+
\begin{cases}
mE &\text{if $s\in S\setminus\Delta$},
\\
(m-n)E &\text{if $s\in\Delta$.}
\end{cases}
\end{equation*}
This proves \eqref {eqnarray:mK+D} in the case where $X/S \dashrightarrow X'/S'$
is an elementary transformation. \eqref {eqnarray:D} in this case follows from Lemma~\xref{Lemma-3}.

In general case, applying transformations as in 
Proposition~\xref{proposition-Sarkisov-elementary-transformations} we may replace $\pi$ and $\pi'$ with
fiberwise birational conic bundles so that $S=S'$.
In this case $\Delta=\Delta'$ because the discriminant curve depends only on 
the generic fiber (see Lemma~\xref{lemma-bir-disc}). The last equality follows from \eqref{lemma-7-conic-bundle-3}.
\end{proof}

\begin{proposition}[{\cite[Remark 7]{Iskovskikh-1987}}, {\cite[Lemma 4]{Iskovskikh-1991-conic-re}}]
\label{proposition-contraction-(-1)-curve}
Let $\pi : X\to S$ be a standard conic bundle
with discriminant curve $\Delta$. Let $E\subset S$ be a $(-1)$-curve on $S$ such that
one of the following three conditions is satisfied:
\begin{enumerate}
\item \label{cases-cb-1}
$E\cdot\Delta=0$, $E\not\subset\Delta$;
\item \label{cases-cb-2}
$E\cdot\Delta=1$, $E\not\subset\Delta$;
\item \label{cases-cb-3}
$E\cdot\Delta=1$, $E\subset\Delta$, and $E \cdot (\Delta-E)=2$.
\end{enumerate}
Let $\alpha: S\to S'$ be the contraction of $E$, then there exists a standard
conic bundle $\pi' : X'\to S'$ with discriminant curve $\Delta'=\alpha(\Delta)$
birationally equivalent to $\pi: X\to S$.
\end{proposition}

\begin{proof}
We show that $X'$ can be obtained from $X$ by means of explicit
birational transformations. 
Denote $F:=\pi^{-1} (E)$. Since $\uprho(X/S')=2$, the relative Mori cone 
$\NE(X/S')$ is generated by two extremal rays.

We need the following easy fact.

\begin{slemma}
Let $\pi: X\to S$ be a conic bundle and let $E\subset S$ be a 
complete smooth rational curve such that $F:=\pi^{-1}(E)$ 
is a smooth surface. Let $C\subset F$ be a curve such that
$\pi_C: C\to E$ is an isomorphism. Then
\begin{equation}\label{equation-deg--K}
K_X\cdot C= -2- (E)_S^2- (C)_F^2, \quad \deg \NNN_{C/X}=(E)_S^2+ (C)_F^2.
\end{equation}
\end{slemma}

{\bf Case~\xref{cases-cb-1}.} Then $F$ is a 
geometrically ruled surface over $E\simeq\PP^1$. 
Let $\Sigma\subset F$ be the exceptional section and let $n:=-(\Sigma)_F^2$. From \eqref{equation-deg--K}
one has $K_X\cdot \Sigma=n-1$ and the normal 
bundle $\NNN_{\Sigma / X}$ is isomorphic to $\OOO_{\PP^1} (-n)\oplus\OOO_{\PP^1} (-1)$.
If $n=0$ on $F$, then $F\simeq\PP^1\times\PP^1$ 
and $\Sigma\cdot F=-1$ on $X$, since $E^2=-1$ on $S$. Therefore, $F\subset X$ 
can be contracted along the pencil $| \Sigma |$ in the category of smooth varieties
and the corresponding contraction is extremal and, in particular, projective.
We obtain a standard conic bundle $\pi': X'\to S'$. If $n\ge 1$, then blowing $\Sigma$ up and 
contracting the proper transform of $F$, we obtain a standard conic bundle $\pi_1: X_1\to S$ over $S$ with 
geometric ruled surface $F_1=\pi^{\prime -1}(E)$ and the exceptional section $\Sigma_1\subset F_1$ whose 
self-intersection index is equal to $- (n-1)$. 
The birational transformation $X\dashrightarrow X_1$ is a Sarkisov link of type~\typem{II}.
After $n$ similar transformations, 
we arrive at the situation considered above. See {\cite[Lemma 4]{Iskovskikh-1991-conic-re}}
or {\cite[Lemma 5.4]{Prokhorov-2005a}} for details.

{\bf Case~\xref{cases-cb-2}.} Then $F$ is a smooth ruled 
surface with one blowup point. Let $C\subset F$ be a section with minimal 
self-intersection number $C^2=-n$ on $F$. It is clear that $n\ge 1$ and 
that $C$ intersects exactly one component of a unique degenerate fiber of the 
ruling $F\to E$. We denote this component by $\Gamma$. Again by \eqref{equation-deg--K} 
one has
\begin{equation*} 
K_X\cdot C=n-1\quad \text{ and }\quad 
\NNN_{C / X}\simeq \OOO_{\PP^1} (-n)\oplus\OOO_{\PP^1} (-1).
\end{equation*}

Suppose first, that $n=1$. Since $K_X\cdot C=0$, the divisor $-K_X$ is nef over $S'$.
Then we can make a flop $X \dashrightarrow Z$ and then the proper transform of 
$F$ becomes contractible to a curve, i.e there exists a Mori extremal contraction 
$Z \to X'$ over $S'$, and then $\pi': X'\to S'$ is a standard conic bundle.
The transformation $X \dashrightarrow X'$ is a type~\typem{III} Sarkisov link inverse to the transformation described in 
\xref{elementary-transformations-2}.

If $n\ge 2$, then we consider the blowup $p: Z\to X$ of $C$ and 
complete it to a Sarkisov link of type~\typem{II} (we use the notation of ~\xref{definition-Sarkisov-links}).
Here $\chi: Z \dashrightarrow Z_1$ is a flop in the proper 
transform of $\Gamma$ and $q: Z_1\to X_1$ is the contraction of 
the proper transform of $F$ to a curve over $S$.
The exceptional divisor $R:=p^{-1}(C)$ over $C$ is isomorphic to 
$\FF_{n-1}$ and its negative section is disjoint from the flipping curve.
Thus we obtain a standard conic bundle $\pi_1: X_1\to S$. Moreover,
the inverse image 
$F_1 =\pi_1^{-1} (E)$ is obtained by blowing up of $\FF_{n-1}$
at a point that does not lie on the negative section. 
In particular, $F_1$ has 
a section $C_1$ with $C^{2}_1=- (n -1)$. 
Therefore, after $n-1$ transformations of the type we have considered, we 
arrive at the situation $n=1$. 

{\bf Case~\xref{cases-cb-3}.}
Then
the singular locus of the surface $F$ is a curve $\Sigma\subset F$ such that
the restriction $\pi_\Sigma:\Sigma\to E$ is an isomorphism. Let $\nu: \tilde F\to F$ be the normalization
and let $\tilde\Sigma:=\nu^{-1}(\Sigma)$. From the explicit equations in~\xref{discriminant}
one can see that the pair $(X,F)$ is log canonical, the surface $F$ has generically normal crossings along $\Sigma$ and 
$\nu : \tilde F \to F$ induces a morphism $\nu_\Sigma :\tilde\Sigma \to \Sigma$ of degree 2 branched 
at two points $\pi^{-1}(\overline{\Delta\setminus E})\cap \Sigma$. 
The surface $\tilde F$
is a minimal rational ruled surface isomorphic to $\FF_e$ for some $e\ge 0$ with ruling 
$\tilde F\to \tilde E$, where 
$\tilde E$ is a curve that parametrizes the
irreducible components of the fibers of $F\to E$ and 
$\tilde E\to E$ is the double cover branched at two points $E\cap \overline{\Delta\setminus E}$.

By the projection formula we have
\begin{equation*}
F\cdot \Sigma=\pi^*E\cdot \Sigma=E \cdot \pi_*\Sigma =E^2=-1.
\end{equation*}
Therefore,
\begin{equation}
\label{eq:relation1}
(K_X+ F) \cdot \nu_* \tilde\Sigma=(K_X + F) \cdot 2\Sigma=2K_X \cdot \Sigma- 2.
\end{equation}
On the other hand, since $(X,F)$ is log canonical, by the adjunction formula \cite[3.1]{Shokurov-1992-e-ba}
\begin{equation*}
\nu^* (K_X+F)|_F=K_{\tilde F}+\tilde\Sigma.
\end{equation*}
Hence,
\begin{equation}
\label{eq:relation2}
(K_X+F)\cdot \nu_* \tilde \Sigma=
(K_{\tilde F}+\tilde\Sigma)\cdot \tilde\Sigma= 2\p (\tilde\Sigma)- 2=-2.
\end{equation}
Combining the above two relations \eqref{eq:relation1} and \eqref{eq:relation2}, we obtain 
\begin{equation}
\label{equation-K.Sigma}
K_X \cdot \Sigma= 0. 
\end{equation} 
Let $m:= (\tilde \Sigma)_{\tilde F}^2$. Let $\Lambda$ be the ruling of $\FF_e=\tilde F\to \tilde E$.
Since $\pi: X \to S$ is a standard conic bundle, we have $\nu ^* K_{X} \cdot \Lambda= -1$.
Hence $\nu ^* K_X = -\tilde \Sigma + a \Lambda$ for some $a \in \ZZ$. 
The relation \eqref{equation-K.Sigma} implies $a=m$, i.e. $\nu ^* K_X = -\tilde \Sigma + m \Lambda$.
Therefore,
\begin{equation*}
(\nu ^*K_X)^2 =( - \tilde \Sigma+m\Lambda)^2=m- 2m=-m.
\end{equation*}
On the other hand, using \eqref{equation-canonical-bundle-formula} we can compute:
\begin{multline*}
-m=(\nu ^*K_X)^2 = \pi_*\nu_*\nu ^*(K_X^2)=\pi_*( \pi^* E\cdot K_X^2)=\\
=E\cdot \pi_*(K^2_X )=E \cdot (-4K_S-\Delta)=3.
\end{multline*}
Hence $\tilde F\simeq \FF_3$
and $\tilde \Sigma$ is the negative section of $\tilde F\to \tilde E$.
Since any curve on $\tilde F$ is linearly equivalent to a convex linear combination 
of $\tilde \Sigma$ and $\Lambda$, we conclude that $-K_X$ is nef over $S'$
and $\Sigma$ is the only curve contracted by $|-K_X|$ over $S'$.
Let $X\dashrightarrow Z$ be the flop in $\Sigma$. Then $Z$ has a Mori extremal contraction 
$Z\to X'$ over $S'$ and the induced morphism $\pi' : X'\to S'$ is a standard conic bundle.
The transformation described above is a type~\typem{III} Sarkisov link inverse to the transformation described in 
\xref{elementary-transformations-3}. 

The proof of Proposition~\xref{proposition-contraction-(-1)-curve} is complete.
\end{proof}

\begin{proposition}[\cite{Iskovskikh-1991-conic-re}]
\label{proposition-transformations:2K+S=0}
Let $\pi: X\to S$ be a standard conic bundle over a rational surface $S$
and let $\Delta\subset S$ be the discriminant curve.
Assume that $|2K_S+\Delta|=\emptyset$. Then $2K_S+\Delta$ is not nef and
there exists 
a standard conic bundle $\pi^{\sharp}: X^{\sharp}\to S^{\sharp}$ with discriminant curve $\Delta^{\sharp}$ and the following 
commutative diagram
\begin{equation}
\label{diagram-transformations:2K+S=0}
\vcenter{
\xymatrix{
X\ar[d]^{\pi}\ar@{-->}[r]& X^{\sharp}\ar[d]^{\pi^{\sharp}}
\\
S\ar[r]^{\alpha}&S^{\sharp}
} }
\end{equation}
where $S^{\sharp}\simeq \FF_n$ or $\PP^2$, the morphism $\alpha$ is birational and 
$X\dashrightarrow X^{\sharp}$ is a birational map. Moreover, we have 
$\Delta^{\sharp}\cdot \Lambda^{\sharp}\le 3$ for a fiber $\Lambda^{\sharp}$ of $\FF_n$ 
\textup(resp. $\deg \Delta^{\sharp}\le 5$\textup) if $S^{\sharp}=\FF_n$ 
\textup(resp. if $S^{\sharp}=\PP^2$\textup).
In particular, the discriminant curve $\Delta$ is connected.
\end{proposition}

\begin{proof}
By the Riemann-Roch theorem 
\begin{equation*}
\hh^0(S, K_S+\Delta)+ \hh^0(S, -\Delta)\ge \p(\Delta)\ge 1. 
\end{equation*}
Therefore, $|K_S+\Delta|\neq \emptyset$. 
Apply Lemma~\xref{lemma-rational-surface=contractions} below with $m=2$.
Then $2K_S+\Delta$ is not nef.
If $C$ is a $(-1)$-curve on $S$ with $(2K_S+\Delta)\cdot C<0$,
then $\Delta\cdot C<2$. Then we can apply Proposition~\xref{proposition-contraction-(-1)-curve} and obtain a new conic bundle
$\pi': X'\to S'$ with $|2K_{S'}+\Delta'|=\emptyset$ (see \eqref{eqnarray:mK+D}). 
We can continue the process
until we get a model with $S^{\sharp}\simeq \FF_n$ or $\PP^2$.

The connectedness of $\Delta^\sharp$ is easy to check on $S^{\sharp}=\FF_n$ 
and $\PP^2$. On the other hand, the number of 
connected components of the discriminant curve is preserved under our birational transformations. 
\end{proof}

\begin{slemma}[\cite{Iskovskikh-1991-conic-re}]
\label{lemma-rational-surface=contractions}
Let $S$ be a smooth rational projective surface and $\Delta$ a reduced effective 
divisor on $S$ such that for any smooth rational curve $C\subset\Delta$ 
one has $C \cdot (\Delta-C)\ge 2$. Let $m$ and $n$ be positive integers such that
$m>n$,
\begin{equation*}
|mK_S+n\Delta|=\emptyset,\qquad |(m-1)K_S+n\Delta|\neq \emptyset.
\end{equation*}
Then the divisor $mK_S+n\Delta$ is not nef and 
there exists an extremal ray $R\subset \NE (S)$ such that $K_S\cdot R<0$ and $(mK_S+n\Delta)\cdot R<0$. Moreover,
one of the following holds: 
\begin{enumerate}
\item \label{lemma-rational-surface=contractions-1}
there exists a base point free pencil of rational curves $\LLL$ on $S$ such that 
\begin{equation*}
2m- 2 \le n\Delta\cdot \LLL <2m;
\end{equation*}
\item \label{lemma-rational-surface=contractions-2}
there exists a birational morphism $\varphi : S \to \PP^2$ such that 
\begin{equation*}
3m-3\le n\deg \varphi (\Delta) <3m.
\end{equation*}
\end{enumerate}
\end{slemma}

\begin{proof}
Put $\Delta_{k,l}:=kK_S+l\Delta$. We claim that the divisor $\Delta_{m,n}$ is not nef.
Indeed, otherwise by the Riemann-Roch formula
\begin{equation*}
\hh^0(S, \OOO_{S}(\Delta_{m,n})) +
\hh^0(S, \OOO_{S}(-\Delta_{m-1,n}))\ge \frac 12 \Delta_{m,n}\cdot\Delta_{m-1,n}+1>0.
\end{equation*}
By our assumption, 
\begin{equation*}
H^0(S, \OOO_{S}(\Delta_{m,n}))=0. 
\end{equation*}
Hence,
$H^0(S, \OOO_{S}(-\Delta_{m-1,n}))\neq 0$ and by our assumption $\Delta_{m-1,n}=0$. In this case, 
$\Delta_{m,n}=K_S$ is not nef because $S$ is a rational surface, a contradiction.

Thus, the divisor $\Delta_{m,n}$ is not nef. We regard $\Delta_{m,n}$ 
as a linear function on the Mori cone $\NE (S)$. 
Put 
\begin{equation*}
\NE_+(S)=\{z \in \NE (S) \mid z\cdot K_S \ge 0\}.
\end{equation*}
We claim that $\Delta_{m,n}$ is non-negative on 
$\NE_+(S)$. Indeed, assume that $C\cdot K_S:=d> 0$ and $C\cdot \Delta_{m,n} <0$ 
for some irreducible curve $C$. Then 
\begin{equation*}
nC\cdot\Delta=C\cdot\Delta_{m,n}-mC\cdot K_S <-dm. 
\end{equation*}
Hence $C $ is a component of $\Delta$ and $C^2 \le C\cdot \Delta <-dm/n$. By the genus formula 
\begin{equation*}
2\p(C)-2=C^2+C\cdot K_S <d-d m/n<0.
\end{equation*}
Hence $\p (C)=0$. Then by our assumptions 
\begin{equation*}
-2-d=-2-C\cdot K_S=C^2 \le C\cdot \Delta-2 <-dm/n-2.
\end{equation*}
The contradiction shows that $\Delta_{m,n}$ is non-negative on $\NE_+(S)$. 

Thus, $\Delta_{m,n}$ is negative on some 
$K_S$-negative extremal ray $R \subset \NE (S)$. 
According to \cite{Mori-1982} there are only three possibilities: 

\begin{enumerate}
\renewcommand\labelenumi{\rm \arabic{enumi})}
\renewcommand\theenumi{\rm \arabic{enumi})}
\item \label{possibilities-surface-contractions-1}
$R=\RR_{\ge 0}[E]$, where $E$ is the $(-1)$ -curve, 
\item \label{possibilities-surface-contractions-2}
$S\simeq \FF_e$ and $R=\RR_{\ge 0}[L]$, where $L$ is the ruling, 
\item \label{possibilities-surface-contractions-3}
$S\simeq \PP^2$ and $R=\RR_{\ge 0}[L]$, where $L\subset \PP^2$ is a line. 
\end{enumerate}
In case~\xref{possibilities-surface-contractions-1}, let $\sigma: S \to S'$ be the contraction of $E$.
Put $\Delta':=\sigma_* \Delta$ and 
$\Delta_{k,l}':=kK_{S'}+l\Delta'=\sigma_* \Delta_{k,l}$.
Clearly, the divisor $\Delta'$ is reduced,
$\Delta'\ge 0$, and $|\Delta_{m-1,n}'|\neq \emptyset$.
Moreover, $C'\cdot (\Delta' -C') \ge 2$ for any irreducible rational component 
$C'\subset \Delta'$. Further,
\begin{equation*}
\Delta_{m,n}=\sigma^*\Delta'_{m,n}-(\Delta_{m,n}\cdot E) E.
\end{equation*}
Since $\Delta_{m,n}\cdot E<0$, one can see that
the divisor $\Delta'_{m,n}$ is not effective, as does the divisor 
$\Delta_{m,n}$. Therefore, all the conditions of our lemma are satisfied for $S'$ and 
$\Delta'$. 
Thus it is sufficient to show that either~\xref{lemma-rational-surface=contractions-1}
or~\xref{lemma-rational-surface=contractions-2} holds on $S'$.
Continuing the process of contraction of extremal 
$(-1)$-curves, we arrive at the situation of~\xref{possibilities-surface-contractions-2} or
\xref{possibilities-surface-contractions-3}. 

In case~\xref{possibilities-surface-contractions-2}
we have 
\begin{equation*}
-2\le \Delta_{m-1,n}\cdot L+K_S\cdot L=\Delta_{m,n}\cdot L=n\Delta \cdot L-2m <0, 
\end{equation*}
Put $\LLL:=|L|$. Then $2m-2\le n\Delta\cdot \LLL <2m$. 

In case~\xref{possibilities-surface-contractions-3}
we have 
\begin{equation*}
-3\le \Delta_{m-1,n} \cdot L+K_S\cdot L=n\Delta\cdot L-3m <0. 
\end{equation*}
Hence, $3n-3 \le n\deg \varphi (\Delta) <3m$. The lemma is 
proved.
\end{proof}

Using similar arguments one can prove the following.
\begin{slemma}
\label{lemma-rational-surface=D1}
Let $S$ be a smooth rational projective surface and $\Delta$ a reduced effective 
divisor on $S$ such that 
for any smooth rational curve $C\subset\Delta$ 
one has $C \cdot (\Delta-C)\ge 2$.
Let $\Delta_1$ be a connected component of $\Delta$.
Then 
\begin{equation}
\label{equation-rational-surface=D1-1}
\dim |K_S+\Delta|\ge \p(\Delta_1)-1\ge 0.
\end{equation}
Furthermore, if $(K_S+\Delta)\cdot E<0$ for some irreducible curve $E$, then 
$E$ is a $(-1)$-curve disjoint from $\Delta$.
\end{slemma}

\begin{scorollary}
Let $S$ be a smooth rational projective surface and let $\Delta$ be a reduced effective 
divisor on $S$ such that 
for any smooth rational curve $C\subset\Delta$ 
one has $C \cdot (\Delta-C)\ge 2$.
Then there exists a sequence of contractions of $(-1)$-curves 
$f: S\to S'$ contained in $S\setminus \Supp(\Delta)$ such that $K_{S'}+f_*\Delta$
is nef.
\end{scorollary}

\subsection{}\label{remark-effective-threshold}
Recall that a $\RR$-divisor is said to be \textit{pseudo-effective}
if its class is a limit of classes of effective $\RR$-divisors.
Now let $S$ be a projective variety and let $\Delta$ be an $\RR$-divisor on $S$.
Define the \textit{effective threshold} (see \cite{Reid1986},
\cite{Reid-1991-sarkisov}) as follows 
\begin{equation*}
\etr(S,\Delta):=\sup \{ t\in \RR \mid tK_S+\Delta \quad \text{is pseudo-effective}\}.
\end{equation*}
Assume now that $S$ is a smooth projective rational surface and $\Delta$ is an effective 
reduced divisor on $S$ such that $K_S+\Delta$ is nef.
To compute $\lambda:=\etr(S,\Delta)$ in this situation we can run $(K+\tau \Delta)$-MMP 
for suitable sequence $\tau=\tau_i$ and reduce the 
problem to the case where $\lambda K_S+\Delta$ is nef and then $S\simeq\PP^2$ or $\FF_n$.
Then it is easy to see that $\lambda:=\etr(S,\Delta) \in \frac 12 \ZZ \cup \frac 13 \ZZ$, see \cite{Reid1986} 
for details. In particular, $\etr(S,\Delta)$ is a rational number.
This consideration shows also $\lambda K_S+\Delta$ has a Zariski decomposition 
\begin{equation*}
\lambda K_S+\Delta\equiv N+P
\end{equation*}
whose positive part $P$ is semi-ample and $P^2=0$.
The negative part $N$ is contracted by $|nP|$, $n\gg 0$.
Moreover,
if $\etr(S,\Delta)= m/n$, then $|mK_S+n\Delta|\neq \emptyset$.
Hence, the condition $|2K_S+\Delta|= \emptyset$ in Conjecture~\xref{conjecture-Shokurov}
is equivalent to $\etr(S,\Delta)<2$.

From Proposition~\xref{proposition-rationality-conic-bundles}
we obtain the following.
\begin{scorollary}\label{corollary-transformations:2K+S=0}
Let $\pi: X\to S$ be a standard conic bundle over a 
rational surface $S$ and let $\Delta\subset S$ be the discriminant 
curve. Suppose that $|2K_S+\Delta|=\emptyset$ and $X$ is not rational.
Then there exists a diagram of the form \eqref{diagram-transformations:2K+S=0}
with $S^{\sharp}\simeq \PP^2$ and $\deg \Delta^\sharp=5$.
Moreover, the corresponding double cover $\tilde \Delta^\sharp \to \Delta^\sharp$
is defined by an odd theta characteristic \textup(see 
Proposition~\xref{proposition-Panin}\textup).
\end{scorollary}

\begin{scorollary}[\cite{Iskovskikh-1991-conic-re}]
\label{corollary-equivalence-conjectures}
Conjectures~\xref{conjecture-Shokurov} and~\xref{conjecture-Iskovskikh} are equivalent.
\end{scorollary}
\begin{proof}
Assume that~\xref{conjecture-Iskovskikh-1} or~\xref{conjecture-Iskovskikh-2} of Conjecture
\xref{conjecture-Iskovskikh} holds.
Then in the notation of Conjecture~\xref{conjecture-Iskovskikh}
we have $|2K_{S'}+\Delta'|=\emptyset$ on $S'$ and 
$|2K_{S}+\Delta|=\emptyset$ on $S$ by Lemma~\xref{Lemma:nK+Delta:invariant}.
On the other hand, $X$ is rational (in the case~\xref{conjecture-Iskovskikh-1}
this follows from Proposition~\xref{proposition-rationality-conic-bundles}).
Hence, $\JG(X)=0$.

Conversely, assume that $|2K_{S}+\Delta|=\emptyset$.
Apply Proposition~\xref{proposition-transformations:2K+S=0}.
If $S^{\sharp}\simeq \FF_n$, then the ruling of $\FF_n$ give us the desired 
pencil on $S^{\sharp}$ and $S$ (because $\Delta^\sharp=\alpha(\Delta)$).
Hence we have 
the case~\xref{conjecture-Iskovskikh}\xref{conjecture-Iskovskikh-1} in this situation.
Assume that $S^{\sharp}\simeq \PP^2$ and $\deg \Delta^\sharp\le 5$.
If $\deg \Delta^\sharp\le 4$, then the pencil of lines on $S^{\sharp}\simeq \PP^2$ passing through 
a smooth point $P^\sharp\in \Delta^\sharp$ again gives us 
a pencil as in the case~\xref{conjecture-Iskovskikh}\xref{conjecture-Iskovskikh-1}
(after blowing up $P^\sharp$, see Proposition~\xref{proposition-Sarkisov-elementary-transformations}).
Let $\deg \Delta^\sharp=5$. Then $\p(\Delta^\sharp)=\p(\Delta)=6$.
Since in this case $\JG(X)=0$ (see~\xref{conjecture-Shokurov}\xref{conjecture-Shokurov:2KS+D}),
by Theorems~\xref{theorem-Prym-Jacobian} and
\xref{theorem-Mumford-Beauville-Shokurov} the corresponding double cover 
$\tilde \Delta^\sharp \to \Delta^\sharp$ is given by an even theta-characteristic.
Applying Panin's construction~\xref{proposition-Panin} we obtain 
the case~\xref{conjecture-Iskovskikh}\xref{conjecture-Iskovskikh-2}.
\end{proof}

The following is a very weak version of Conjecture~\xref{conjecture-Iskovskikh}.
\begin{proposition}[\cite{Iskovskikh-1991-conic-re}]
\label{proposition-Iskovskikh-weak-result}
Let $\pi: X\to S$ be a standard conic bundle over a rational surface $S$ with discriminant curve $\Delta 
\subset S$. Assume that there exists a birational map $\Phi: X \dashrightarrow X^\sharp$
to another Mori fiber space $\pi^\sharp: X^\sharp\to S^\sharp$ such that $\Phi$ is not fiberwise.
Then one of the following conditions 
is satisfied:

\begin{enumerate}
\item \label{proposition-Iskovskikh-weak-result-1}
there is a base point free pencil of curves of genus zero $\LLL$ on $S$ such that $\Delta\cdot \LLL\le 7$;
\item \label{proposition-Iskovskikh-weak-result-2}
there is a birational contraction $\varphi:S\to\PP^2$ such that $\varphi(\Delta)=\Delta'\subset\PP^2$ is a
curve, not necessarily with normal crossings, of degree $\deg\Delta'\le 11$.
\end{enumerate}
\end{proposition}

\begin{proof}
It follows from Theorem~\xref{Sarkisov-theorem} below that $|4K_S+\Delta|=\emptyset$. 
If $|2K_S+\Delta|=\emptyset$, then the assertion follows from
Proposition~\xref{proposition-transformations:2K+S=0}. Thus we can take 
$m\in \{3,\, 4\}$ so that
$|mK_S+\Delta|=\emptyset$, and $|(m-1) K_S+\Delta|\neq \emptyset$. 
Then by Lemma~\xref{lemma-rational-surface=contractions}
we obtain~\xref{proposition-Iskovskikh-weak-result-1} or
\xref{proposition-Iskovskikh-weak-result-2}.
\end{proof}

\begin{proposition}[\cite{Tregub1990}, {\cite[Proposition 4.2 b]{Beauville2000}}]
\label{proposition-tregub}
Let $\pi: X\to \PP^2$ be a standard conic bundle with smooth discriminant curve $\Delta$ of 
degree $5$ such that the cover $\tilde \Delta\to\Delta$ is defined by an odd theta-characteristic. 
Then $\pi: X\to \PP^2$ is fiberwise birational to a standard conic bundle $\pi': X'\to \PP^2$
obtained from the projection of a smooth cubic hypersurface $Y_3\subset \PP^4$ from a line
\textup(Example~\xref{example-cubic-conic-bundle}\textup).
\end{proposition}

Note that the cubic $Y_3\subset \PP^4$ can be naturally reconstructed from $\tilde \Delta\to\Delta$:
there is an isomorphism of principally polarized Abelian varieties $\Pr(\tilde \Delta,\Delta)\simeq \J(X)$ and 
the theta-divisor $\Theta$ of $\J(X)$ has a unique singular point whose tangent cone 
is isomorphic to the affine cone over $Y_3\subset \PP^4$ \cite{Tjurin1971}, \cite[Proposition~2.1.5]{Tyurin-middle-Jacobian-e}, 
\cite{Beauville1982a}.

\section{Conic bundles over minimal surfaces}
\label{section:min-surface}
In the case when the base of a conic bundle is a minimal rational surface the criterion~\xref{conjecture-Shokurov} 
was proved by Shokurov \cite{Shokurov1983}. Below we reproduce his arguments.

\begin{theorem}
\label{theorem-criterion-Shokurov:P2}
Let $\pi: X\to \PP^2$ be a standard conic bundle
and let
$\Delta\subset \PP^2$ be the discriminant curve. 
Then the following conditions are equivalent:
\begin{enumerate}
\item \label{theorem-criterion-Shokurov-rat:P2}
$X$ is rational.
\item \label{theorem-criterion-Shokurov-2KSC:P2}
Either $\deg \Delta\le 4$ or $\deg \Delta=5$ and 
the corresponding double cover $\tilde \Delta\to \Delta$ is defined by an even theta-characteristic.

\item \label{theorem-criterion-Shokurov-J:P2}
The intermediate Jacobian $\J(X)$ is isomorphic as a
principally polarized Abelian variety to a product of Jacobians of smooth curves, i.e. its
Griffiths component $\JG(X)$ is trivial.
\end{enumerate}
\end{theorem}

For the case $\deg \Delta \ge 6$ this theorem was proved by Beauville {\cite[Th\'eor\`eme 4.9]{Beauville1977}}.
In the case $\deg \Delta \le 4$ the variety $X$ is rational by Corollary~\xref{corolary-rationality-conic-bundles}.
The case $\deg \Delta =5$ follows from {\cite[\S 3]{Tyurin-5-lectures-1973}}, \cite{Tjurin1975},
\cite{Masiewicki1976}, \cite{Panin1980}.

\begin{theorem}[\cite{Shokurov1983}]
\label{theorem-criterion-Shokurov}
Let $\pi: X\to S$ be a standard conic bundle over a minimal rational ruled surface $S$, that is, 
$S\simeq \FF_n$ with $n\neq 1$,
and let
$\Delta\subset S$ be the discriminant curve. 
Denote by $\Lambda$ a fiber of the ruling $\FF_n\to \PP^1$.
Then the following conditions are equivalent:
\begin{enumerate}
\item \label{theorem-criterion-Shokurov-rat}
$X$ is rational.
\item \label{theorem-criterion-Shokurov-2KSC}
$|2K_S+\Delta|=\emptyset$.

\item \label{theorem-criterion-Shokurov-tr}
$\Delta\cdot\Lambda\le 3$ \textup(for suitable choice of the ruling $\FF_0\to \PP^1$ if $n=0$\textup).

\item \label{theorem-criterion-Shokurov-J}
$\JG(X)=0$.
\end{enumerate}
\end{theorem}

\begin{sremark}
In the case $S=\FF_1$, one can show that the equivalences 
\xref{theorem-criterion-Shokurov-rat} $\Leftrightarrow$ 
\xref{theorem-criterion-Shokurov-2KSC} $\Leftrightarrow$ 
\xref{theorem-criterion-Shokurov-tr} $\Leftrightarrow$ 
\xref{theorem-criterion-Shokurov-J} hold whenever $\Delta\cdot \Sigma>1$. If 
$\Delta\cdot \Sigma\le 1$, then we can apply Proposition
\xref{proposition-contraction-(-1)-curve} to reduce the problem to $\PP^2$. Thus the 
equivalence~\xref{theorem-criterion-Shokurov-rat} $\Longleftrightarrow$ 
\xref{theorem-criterion-Shokurov-J} holds for $S=\PP^2$ and $\FF_n$ with 
\emph{arbitrary} $n$.
\end{sremark}

\begin{corollary}
Conjectures~\xref{conjecture-Shokurov} and~\xref{conjecture-Iskovskikh}
are true for $S=\PP^2$ and $\FF_n$, $n\ge 0$.
\end{corollary}

\begin{proof}[Proof of Theorem~\xref{theorem-criterion-Shokurov}.]
The implication~\xref{theorem-criterion-Shokurov-rat} $\Longrightarrow$ 
\xref{theorem-criterion-Shokurov-J} is an immediate consequence of Corollary 
\xref{corollary-J} and~\xref{theorem-criterion-Shokurov-tr} $\Longrightarrow$ 
\xref{theorem-criterion-Shokurov-rat} is a consequence of 
Propositions~\xref{proposition-rationality-conic-bundles}.
The implication~\xref{theorem-criterion-Shokurov-tr} $\Longrightarrow$ 
\xref{theorem-criterion-Shokurov-2KSC} is obvious. 

Let us prove 
\xref{theorem-criterion-Shokurov-2KSC} $\Longrightarrow$ 
\xref{theorem-criterion-Shokurov-tr}. 
Denote by $\Sigma$ the $(-n)$-section of $S=\FF_n$.
Write $\Delta\sim 
a\Sigma+b\Lambda$. We assume that $|2K_S+\Delta|=\emptyset$ and $a\ge 4$. If 
$\Sigma\not\subset \Delta$, then clearly $b\ge an$ (see e.g.~\cite[Ch.~V, Th.~2.17]{Hartshorn-1977-ag}). 
If $\Sigma\subset \Delta$, then 
\begin{equation*}
(\Delta-\Sigma)\cdot\Sigma=b-an+n\ge 2
\end{equation*}
by~\eqref{equation-condition-S-star-star}. Thus, in both cases, 
\begin{equation*}
b\ge \min\{an,\, n(a -1)+2\}.
\end{equation*} 
Further, 
\begin{equation*}
2K_S+\Delta\sim (a-4)\Sigma+(b-2n-4)\Lambda, 
\end{equation*}
where $b-2n-4< (a-4)n$ because $|2K_S+\Delta|=\emptyset$. This gives us $n=0$ 
and $b<4$. Replacing the ruling of $\FF_0=\PP^1\times \PP^1$ with another one we 
get~\xref{theorem-criterion-Shokurov-tr}.

It remains to prove the implication~\xref{theorem-criterion-Shokurov-J} $\Longrightarrow$ 
\xref{theorem-criterion-Shokurov-tr}. Assume to the 
contrary that $\Delta\cdot \Lambda \ge 4$.

\begin{lemma}\label{lemma-canonical-map}
In the above assumptions one of the following holds:
\begin{enumerate}
\item \label{case-Fn-f1}
$|K_S+\Delta|$ is very ample.
\item \label{case-Fn-f2}
$|K_S+\Delta|$ is base point free but not very ample,
the corresponding morphism is birational and contracts $\Sigma$.
Moreover, $(\Delta-\Sigma)\cdot \Sigma=2$. 
\end{enumerate}
\end{lemma}

\begin{proof}
As above, we write 
\begin{equation*}
\Delta\sim a\Sigma+b\Lambda,\quad K_S+\Delta\sim (a-2)\Sigma+(b-2-n)\Lambda. 
\end{equation*}
By our assumption $a\ge 4$. If $n=0$, then by symmetry we may assume that $b\ge 4$.

The linear system $|K_S+\Delta|$ is base point free if and only if 
\begin{equation}\label{equation-Fn-f}
b\ge an+2-n.
\end{equation}
This relation is obvious if $n=0$. Thus we assume that $n\ge 2$.
If $b\ge an$, then \eqref{equation-Fn-f}
holds. 
Consider the case $b<an$. Then $\Sigma$ is a component of $\Delta$.
Put $\Delta':=\Delta- \Sigma$.
Then
\begin{equation*}
2\le \Delta'\cdot\Sigma=b-an+n
\end{equation*}
which is equivalent to \eqref{equation-Fn-f}. 
Moreover, if $|K_S+\Delta|$ is not very ample, then in \eqref{equation-Fn-f}
the equality holds and we get~\xref{case-Fn-f2}.
\end{proof}

The following assertion can be proved similar to 
Lemma~\xref{lemma-canonical-map}.

\begin{lemma}
\label{lemma:Delta-decomp}
If the curve $\Delta$ does not satisfy the condition 
\eqref{equation-condition-S-star}, that is, there exists a decomposition 
$\Delta=\Delta_1+\Delta_2$ with $\Delta_1\cdot \Delta_2=2$, then up to 
permutation of $\Delta_i$ we may assume that $\Delta_2=\Sigma$ and 
$(K_S+\Delta)\cdot \Delta_2=0$. In particular, $n>0$. Thus $\Delta$ satisfies 
\eqref{equation-condition-S-star} if and only if the linear system $|K_S+\Delta|$ is very ample.
\end{lemma}

Now we are in position to finish the proof of the implication~\xref{theorem-criterion-Shokurov-J} $\Longrightarrow 
$\xref{theorem-criterion-Shokurov-tr} in Theorem~\xref{theorem-criterion-Shokurov}.
Consider the morphism $\phi: S \to \PP^N$ 
given by the linear system $|\upomega_S(\Delta)|$ (see Lemma~\xref{lemma-canonical-map}).
From the exact sequence
\begin{equation*}
0\longrightarrow\upomega_S\longrightarrow\upomega_S (\Delta)\longrightarrow\upomega_\Delta\longrightarrow 0
\end{equation*}
and the vanishing $H^0(S,\upomega_S)=H^1 (S,\upomega_S)=0$ it follows that the linear system $|\upomega_S(\Delta)|$ on $S$
restricts isomorphically to the canonical linear system $|\upomega_\Delta|$ on $\Delta$. 
Therefore, $\phi$ induces the canonical morphism on $\Delta$
\begin{equation*}
\phi_{\Delta}:\Delta\to\PP^{\p(\Delta)-1}. 
\end{equation*}
According to Lemma \ref{lemma-canonical-map} the morphism $\phi$ is birational.
Hence $\Delta$ is not hyperelliptic.

Assume that $K_S+\Delta$ is very ample.
Then $\phi_{\Delta}$ is an embedding and $\Delta$ satisfies the condition \eqref{equation-condition-S-star}.
By Theorem~\xref{theorem-Mumford-Beauville-Shokurov} the curve $\Delta$
is either trigonal, quasitrigonal or isomorphic to a plane quintic. 
The surface $\phi(S)$ is an intersection of quadrics. 
Therefore it contains all the $3$-secants of the curve $\phi(\Delta)$.
If $\Delta$ is trigonal or quasitrigonal, then this implies that $\phi(S)$ is swept out by $3$-secants (see Remark \ref{remark:Delta-can}),
and these $3$-secants must be the images of the fibers $\Lambda$. 
Hence, in this case, $\Delta\cdot \Lambda\le 3$ which contradicts our assumption.
If $\Delta$ is a plane quintic, then the intersection of quadrics passing through $\phi(\Delta)$
is the Veronese surface. But this surface cannot be an image of $\FF_n$ for $n\neq 1$.
Again we get a contradiction.

Finally assume that the divisor   $K_S+\Delta$ is not very ample.
By Lemma \ref{lemma:Delta-decomp} we have $\Delta=\Sigma+\Delta_1$, $\Sigma\cdot \Delta_1=2$, and  $(K_S+\Delta)\cdot\Sigma=0$.
Moreover the surface $\phi(S)$ is isomorphic to a cone over a rational normal curve. 
According to 
Remark \ref{remark:Delta:Prym}\ $\Pr(\tilde\Delta,\Delta)\simeq \Pr(\tilde\Delta',\Delta')$,
where $\Delta'=\phi(\Delta_1)$. 
It is easy to see that the embedding  $\Delta'=\phi(\Delta_1)\subset \PP^{\p(\Delta)-1}$
is canonical. Now we can apply  Theorem~\xref{theorem-Mumford-Beauville-Shokurov} and 
Remark \ref{remark:Delta-can} and obtain a contradiction as above.
\end{proof}

In view of Theorems~\xref{theorem-criterion-Shokurov:P2} and~\xref{theorem-criterion-Shokurov}
one can expect that the rationality of a three-dimensional variety 
is equivalent to the triviality of Griffiths component $\JG(X)=0$.
However this statement is wrong in general as the following Sarkisov's result shows.

\begin{theorem}[{\cite[Theorem~5.10]{Sarkisov-1982-e}}]
\label{Sarkisov-example-H3}
There exists a smooth algebraic rationally connected non-rational threefold $X$ whose 
three-dimensional integral cohomology group $H^3(X, \ZZ)$ is trivial. 
\end{theorem}

\begin{proof}
Let $\Delta_0\subset \PP^2$ be an irreducible curve with at most ordinary double 
singular points whose normalization $\Delta$ has genus $1$. Assume that $\deg 
\Delta_0>12$. Let $\sigma: (S,\Delta)\to (\PP^2,\Delta_0)$ be the minimal 
embedded resolution of singularities. Finally, let $\tilde\pi: \tilde\Delta\to 
\Delta$ be a nontrivial \'etale double cover. Then from Lemma 
\xref{proposition-AM} it follows that there exists a standard conic bundle 
$\pi:X\to S$ with discriminant divisor $\Delta$ such that the corresponding double cover 
is isomorphic to
$\tilde\pi: \tilde\Delta\to \Delta$. It is easy to 
see that the curve $\Delta$ satisfies the condition $|4K_S+\Delta| \neq 
\emptyset$ and therefore $X$ is a non-rational variety by Theorem 
\xref{Sarkisov-theorem} below. On the other hand, by 
\eqref{lemma-7-conic-bundle-3} we have $H^3(X,\CC)=0$ and since the 
discriminant curve of $\pi:X\to S$ connected, it follows that $H^3(X,\ZZ)$ is 
torsion-free by \eqref{equation-Artin-Mumford-computation}. This proves the 
theorem.
\end{proof}

\section{$\QQ$-conic bundles}
\label{section:Q-conic-bundles}
To study the local structure of $\QQ$-conic bundles one needs 
the corresponding concept that works in the analytic category.
\begin{definition}
Let $(X, C )$ be an analytic germ of a threefold with terminal singularities 
along a reduced complete curve. We say that $(X, C)$ is a \textit{$\QQ$-conic 
bundle germ} if there is a contraction 
\begin{equation*}
\pi:(X, C)\to (S, o)
\end{equation*}
to a normal 
surface germ $(S,o)$ such that $C=\pi^{-1}(o)_{\red}$ and the divisor $-K_X$ is 
$\pi$-ample.
\end{definition}

Note that in this definition we assume neither $X$ is (analytically) 
$\QQ$-factorial nor $\uprho_{\mathrm{an}}(X/S)=1$. This is because these 
properties are not stable under passing from algebraic to analytic category, cf. 
\cite[\S 1]{Kawamata-1988-crep}. 

$\QQ$-conic bundle germs were studied in the series of papers 
\cite{Prokhorov-1996b}, 
\cite{Prokhorov-1997_e}, 
\cite{Prokhorov-1999b},
\cite{Prokhorov-1999d},
\cite{Mori-Prokhorov-2008},
\cite{Mori-Prokhorov-2008a}, 
\cite{Mori-Prokhorov-2008III},
\cite{Mori-Prokhorov-IA},
\cite{Mori-Prokhorov-IC-IIB}, 
\cite{Mori-Prokhorov-IIA-1}, 
\cite{Mori-Prokhorov-IIA-II}.
These works essentially use the ideas and techniques introduced by S. Mori 
\cite{Mori-1988} and 
developed in \cite{Kollar-Mori-1992}.

If $X$ has only Gorenstein (terminal) singularities along $C$, then, as in the 
standard conic bundle case, the base surface $S$ is smooth at $o$ and $\pi$ is a 
conic bundle, i.e. $X$ admits a local embedding to $\PP^2\times S$ so that 
fibers of $\pi$ are conics \cite{Cutkosky-1988}. In this case, the singularities 
of $X$ and the discriminant curve are partially described in \cite[Prop. 
5.2]{Prokhorov-planes}. From now on we assume that $X$ is not Gorenstein. 

Recall that the \textit{index} of a $\QQ$-Gorenstein singularity $(X,P)$ (or a $\QQ$-Gorenstein variety $X$)
is called the smallest positive integer $m$ such that $mK_X$ is a Cartier divisor.
For the classification of three-dimensional terminal singularities and corresponding notations we refer to 
\cite[\S 6]{Reid-YPG1987} and \cite[Theorem 5.43]{Kollar-Mori-1988}.

If the 
total space $X$ has only singularities of index two and the base is smooth, then 
the description $\QQ$-conic bundle germs is very explicit:

\begin{theorem}[\cite{Mori-Prokhorov-2008}, \cite{Prokhorov-1997_e}]
\label{theorem-index=2-conic-bundles}
Let $\pi: (X,C)\to (S,o)$ be a $\QQ$-conic bundle germ. Assume that $X$ is not Gorenstein, $(S,o)$ is 
smooth and the singularities of $X$ have indices $\le 2$. Fix an isomorphism $(Z,o)\simeq (\CC^2,0)$. 
Then there is an embedding
\begin{equation}
\label{eq-diag-last-2}
\vcenter{
\xymatrix{X \ar@{^{(}->}[r] \ar[rd]_{\pi}& \PP(1,1,1,2)\times \CC^2
\ar[d]^{p}
\\
&\CC^2}
}
\end{equation}
such that $X$ is given by two equations
\begin{equation}
\label{eq-eq-index2}
\begin{array}{l}
q_1(y_1,y_2,y_3)=\psi_1(y_1,\dots,y_4;u,v),
\\[7pt]
q_2(y_1,y_2,y_3)=\psi_2(y_1,\dots,y_4;u,v),
\end{array}
\end{equation}
where $\psi_i$ and $q_i$ are quasihomogeneous polynomials
in $y_1,\dots,y_4$, quadratic 
with respect to the weight $\wt (y_1,\dots,y_4)=(1,1,1,2)$ and
$\psi_i(y_1,\dots,y_4;0,0)=0$. The only non-Gorenstein point of $X$
is $(0,0,0,1; 0,0)$.
\end{theorem}

In the case where the base of a $\QQ$-conic bundle is singular, we also have a complete classification:

\begin{theorem}[\cite{Mori-Prokhorov-2008}, \cite{Mori-Prokhorov-2008a}]
\label{theorem-conic-bundles-singular-base}
Let $\pi: (X,C)\to (S,o)$ be a $\QQ$-conic bundle germ. Assume that $(S,o)$ is singular. Then one of the
following holds:

\begin{emptytheorem}
\label{item=main--th-pr-toric}
$(X,C)$ is biholomorphic to the quotient of $\PP^1_x\times\CC^2_{u,v}$
by the $\muu_m$-action
\begin{equation*}
(x;u,v) \longmapsto(\varepsilon x; \varepsilon^a u,
\varepsilon^{-a} v),
\end{equation*}
where $\varepsilon$ is a primitive $m$-th root of unity and $\gcd
(m,a)=\gcd(m,b)=1$. 
The singular locus of $X$ consists of two cyclic quotient
singularities of types $\frac 1m(1,a,-a)$ and $\frac 1m(-1,a,-a)$.
The base surface $\CC^2/\muu_m$ has a singularity of type~\type{A_{m-1}}.
\end{emptytheorem}

\begin{emptytheorem}
\label{item=main--th-pr-ex3}
$(X,C)$ is biholomorphic to the quotient of the smooth $\QQ$-conic
bundle
\begin{equation*}
X'=\{ y_1^2+uy_2^2+vy_3^2=0\}\subset \PP^2_{y_1,y_2,y_3} \times\CC^2_{u,v}
\longrightarrow \CC^2_{u,v}.
\end{equation*}
by the $\muu_m$-action
\begin{equation*}
(y_1,y_2,y_3,u,v)\longmapsto (\varepsilon^{a} y_1,\varepsilon^{-1}y_2,y_3,\varepsilon
u,\varepsilon^{-1} v).
\end{equation*}
Here $m=2a+1$ is odd and $\varepsilon$ is a primitive $m$-th root of
unity. The singular locus of $X$ consists of two cyclic quotient
singularities of types $\frac 1m(a,-1,1)$ and $\frac 1m(a+1,1,-1)$.
The base surface $\CC^2/\muu_m$ has a singularity of type~\type{A_{m-1}}.
\end{emptytheorem}

\begin{emptytheorem}
\label{item-main-th-impr-barm=2-s=4} 
$(X,C)$ is the quotient of an index two $\QQ$-conic
bundle of the form \eqref {eq-eq-index2} with 
$q_1=y_1^2-y_2^2$, $q_2= y_1y_2-y_3^2$ 
by the $\muu_{4}$-action
\begin{equation*}
(y_1, y_2, y_3, y_4; u, v) \longmapsto
(-\ii y_1, \ii y_2, -y_3, 
\ii y_4; \ii u, -\ii v).
\end{equation*}
The base surface
$(S,o)$ 
is Du Val of type~\type{A_3}, $X$ has a cyclic quotient singularity $P$
of type $\frac18(5,1,3)$ and has no other singular points.
\end{emptytheorem}

\begin{emptytheorem}
\label {item-main-th-impr-barm=1}
$(X,C)$ is a quotient of a Gorenstein conic
bundle given by the following equation in $\PP^2_{y_1,y_2,y_3}\times
\CC^2_{u,v}$
\begin{equation*}
y_1^2+y_2^2+\psi(u,v)y_3^2=0, \qquad \psi(u,v)\in\CC\{u^2,\, v^2,\, uv\},
\end{equation*}
by the $\muu_{2}$-action
\begin{equation*}
(y_1, y_2, y_3; u, v) \longmapsto (-y_1, y_2, y_3; -u,-v).
\end{equation*}
Here $\psi(u,v)$ has no multiple factors. In this case, $X$ has
a unique singular point and it is of type~\type{cA/2} or \type{cAx/2} and the base surface $(S,o)$
is Du Val of type~\type{A_1}.
\end{emptytheorem}

\begin{emptytheorem}
\label{item-main-th-impr-barm=2-s=2-cycl}
$(X,C)$ is 
the quotient of an index two $\QQ$-conic
bundle of the form \eqref {eq-eq-index2} with 
$q_1=y_1^2-y_2^2$, $q_2=y_3^2$
by the $\muu_{2}$-action
\begin{equation*}
(y_1, y_2, y_3, y_4; u, v) \longmapsto
(y_1, -y_2, y_3, -y_4; -u,-v).
\end{equation*}
The base $(S,o)$ is Du Val of type~\type{A_1} and $X$ has a unique non-Gorenstein 
point which is either a cyclic quotient singularity
of type $\frac{1}{4}(1,1,-1)$ or a singularity of type~\type{cAx/4}.
\end{emptytheorem}

\begin{emptytheorem}\label{item=main--th-cyclic-quo}
$(X,C)$ is 
the quotient of an index two $\QQ$-conic
bundle of the form \eqref {eq-eq-index2} 
with $q_1=y_1^2-y_3^2$, $q_2=y_2^2-y_3^2$ by the $\muu_2$-action
\begin{equation*}
(y_1, y_2, y_3, y_4; u, v) \longmapsto 
(-y_1, -y_2, y_3, -y_4; -u, -v). 
\end{equation*}
The base 
$(S,o)$ is Du Val of type~\type{A_1} and $X$ has a unique non-Gorenstein point 
which is either a cyclic quotient singularity of type $\frac14(1,1,-1)$ or
a singularity of type~\type{cAx/4}.
\end{emptytheorem}

The central curve $C$ is irreducible except for the case~\xref{item=main--th-cyclic-quo}
where $C$ has two irreducible component meeting at the non-Gorenstein point. 
\end{theorem}

The main step in the proof of Theorem~\xref{theorem-conic-bundles-singular-base}
is the following version of Reid's ``general elephant conjecture''.
\begin{theorem}[\cite{Mori-Prokhorov-2008}, \cite{Mori-Prokhorov-2008III}]
\label{theorem-ge}
Let $\pi: (X,C)\to (S,o)$ be a $\QQ$-conic bundle germ.
Assume that $C$ is irreducible. Then a general member 
of the linear system $|-K_X|$ has only Du Val singularities.
\end{theorem}
The proof uses the techniques developed by Mori in the study 
of flipping extremal curve germs \cite{Mori-1988}, \cite{Kollar-Mori-1992}.

As an immediate consequence of Theorem~\xref{theorem-conic-bundles-singular-base} above we have the following

\begin{theorem}[{\cite{Mori-Prokhorov-2008}}]
\label{singularities-base-DV} 
If $\pi: X\to S$ is a $\QQ$-conic bundle, then $S$ can have at most
Du Val singularities of type~\type{A}.
\end{theorem}

If $(S\ni o)$ is of type~\type{A_{m-1}}, then 
we say that the germ $(X,C)$ has \textit{topological index} $m$.
In this case $(X,C)$ is a quotient of a $\QQ$-conic bundle germ 
over a smooth base by $\mumu_m$ (see \cite[(2.4)]{Mori-Prokhorov-2008}, 
\cite[Construction 1.9]{Prokhorov-1997_e}).

\subsection{}
Note that the condition that $X$ has terminal singularities is crucial 
in Theorem~\xref{singularities-base-DV}.
In general (for contractions in the log terminal category), 
the singularities of the base are worse than canonical.
It is known that the base $S$ of an arbitrary Mori fiber space $\pi: X\to S$ is a normal $\QQ$-factorial variety with
at worst log terminal singularities \cite[Lemma 5-1-5]{KMM}, \cite{Fujino-1999app}.

\begin{subexample}
\label{ex-toric}
Consider the following action of $\muu_m$ on $\PP^1_x\times\CC^2_{u,v}$:
\begin{equation*}
(x;u,v) \longmapsto(\varepsilon x; \varepsilon u,\varepsilon^{m-2} v),
\end{equation*}
where $m$ is odd and $\varepsilon$ is a primitive $m$-th root of unity. Let $X:=\PP^1\times\CC^2/\muu_m$,
$S:=\CC^2/\muu_m$ and let $\pi: X\to S$ be the natural
projection. Since $\muu_m$ acts freely in codimension one, $-K_X$ is
$f$-ample. Two fixed points on $\PP^1\times \CC^2$ gives two 
quotient singularities of $X$ which are terminal of type $\frac1m(-1,1,-2)$ and
canonical Gorenstein of type $\frac1m(1,1,-2)$. In this case, $f$ is an extremal contraction 
in the category threefolds with canonical singularities. The base surface $S$ has a 
quotient singularity of type $\frac1m (1,-2)$ which is not canonical but $\delta$-log canonical with
$\delta=(m+1)/2m$. 
\end{subexample}
However, 
V. Shokurov conjectured that for any extremal $K$-negative contraction from threefold with only
canonical singularities to a surface, the singularities of the base are 
$1/2$-log canonical.
More generally, there is the following conjecture posed by J. McKernan.

\begin{sconjecture}
For fixed integer $n>0$ and a real number 
$\varepsilon >0$, there exists a constant $\delta=\delta(n,\varepsilon)>0$ such that the following holds: 
If $\pi :X\to S$ is an extremal contraction, where $X$ is a $\QQ$-factorial $n$-dimensional variety with only 
$\varepsilon$-log terminal singularities, then the singularities of $S$ are $\delta$-log terminal. 
\end{sconjecture}

This conjecture 
in the toric case was proved by Alexeev and Borisov \cite{Alexeev-Borisov-toric}.
Partial results in general case were obtained in \cite{Birkar-sMFS}
and related to log adjunction (cf. \cite{Shokurov2013}, \cite{Shokurov2017}).

As a consequence of Theorem~\xref{theorem-conic-bundles-singular-base} we also 
have the following two facts.

\begin{corollary}
\label{corollary-Q-conic-bundles-}
Let $\pi: (X,C)\to (S,o)$ be a $\QQ$-conic bundle germ and 
let $\Delta\subset S$ be the discriminant curve.
Then $\Delta$ is a Cartier divisor at $o$ except for the cases 
\xref{item-main-th-impr-barm=2-s=4},~\xref{item-main-th-impr-barm=2-s=2-cycl},
\xref{item=main--th-cyclic-quo}.
\end{corollary}

\begin{corollary}
\label{corollary-Q-conic-bundles-plt}
Let $\pi: (X,C)\to (S,o)$ be a $\QQ$-conic bundle germ and 
let $\Delta\subset S$ be the discriminant curve. 
Then the following holds.
\begin{enumerate}
\item 
$\Delta\not\ni o$ if and only if
$(X,C)$ is of type~\xref{item=main--th-pr-toric}.
\item 
If $\Delta\ni o$ and the pair $(S,\Delta)$ is plt at $o$, then 
$X$ is smooth near $C$ \textup(and $S$ is smooth at $o$\textup).
\end{enumerate}
\end{corollary}

\subsection{}
Results in the case of smooth base surface $S$ are not complete.
However this case studied well under additional assumption that central curve $C$
is irreducible
\cite{Mori-Prokhorov-2008III},
\cite{Mori-Prokhorov-IA},
\cite{Mori-Prokhorov-IC-IIB}, 
\cite{Mori-Prokhorov-IIA-1}, 
\cite{Mori-Prokhorov-IIA-II}. The main strategy is as follows.
Using a ``good'' member $D\in |-K_X|$ (Theorem~\xref{theorem-ge})
it is possible to analyze a general hyperplane section $H\subset X$
passing through $C$. Then the threefold $X$ can be viewed as the total space 
of a one-parameter deformation of $H$. If $H$ is normal, then its singularities are rational.

\begin{subexample}[\cite{Mori-Prokhorov-IC-IIB}]
Consider a normal surface germ $(H,C)$ along a curve $C\simeq \PP^1$ 
whose minimal resolution $H_{\min}$ have the following dual graph
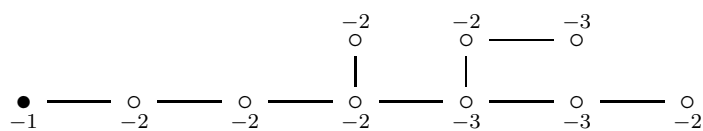
\begin{equation*} 
\xymatrix@R=2pt{
&&&\overset{-2}\circ\ar@{-}[d]&\overset{-2}\circ \ar@{-}[d] \ar@{-}[r]&\overset{-3}{\circ}
\\
\underset{-1}\bullet\ar@{-}[r]& \underset{-2}\circ\ar@{-}[r]&\underset{-2}\circ\ar@{-}[r]&\underset{-2}\circ\ar@{-}[r]
&\underset{-3}{\circ}\ar@{-}[r]&\underset{-3}{\circ}\ar@{-}[r]&\underset{-2}\circ 
}
\end{equation*}
Here the white vertices correspond to exceptional rational 
curves and the black vertex corresponds to the curve $C$. The number attached to a vertex 
is the self-intersection of the corresponding curve. It is easy to see that 
the configuration of white vertices can be contracted to a rational singularity 
$H\ni P$ and the whole configuration is a fiber of a rational curve fibration 
$H_{\min}\to \Gamma$. Detailed computations \cite[10.8]{Kollar-Mori-1992} show that, locally, $H\ni P$ can be realized as 
a hyperplane section of a threefold terminal singularity
\begin{equation*} 
(X,P)\simeq\CC^3_{y_1,y_2,y_4}/\muu_5(2,3,1)
\end{equation*}
such that $H\supset C$, where 
\begin{equation*}
C=\{y_1^{3}-y_2^2=y_4=0\}/\muu_5.
\end{equation*}
The deformations of $H$ are unobstructed, so a general one-parameter deformation of $H$
is a threefold $X\supset H$ with terminal singularity at $P$. We get a $\QQ$-conic bundle \
$\pi: (X,C)\to (S,o)$ with smooth $S$.
\end{subexample}
The situation is more complicated in the case where $H$ is not normal.
Let us give the simplest example.

\begin{subexample}[\cite{Mori-Prokhorov-IA}]
Let $\lambda_1, \lambda_2\in \CC$ be some general constants, and let
$X$ be the threefold given in $\PP(1,a,m-a,m)\times \CC_t$ by 
\begin{equation*}
x_1^{2m-2a}x_2^2+x_1^{2a}x_3^2+ x_2x_3x_4+(\lambda_1 x_1^m-x_4)(\lambda_2x_1^m-x_4)t=0. 
\end{equation*}
Then a small analytic neighborhood of the 
curve 
\begin{equation*}
C:=\{x_2=x_3=t=0\}
\end{equation*}
is a $\QQ$-conic bundle germ. 
The singular locus of $X$ near $C$ consists of a cyclic quotient singularity of type 
$\frac 1m (1,a,m-a)$ and two (Gorenstein) ordinary double points. 
A general hyperplane section passing through $C$ is not normal.
\end{subexample}

In the sequel we need the following simple lemmas.

\begin{lemma}
\label{lemma-Delta-multiplicity=2}
Let $\pi: (X,C)\to (S,o)$ be a $\QQ$-conic bundle germ over smooth base
and let $\Delta$ be the discriminant curve. 
Assume that $X$ is singular and $\mult_o(\Delta)=2$. 
Take a standard model $\pi^\bullet : X^\bullet\to S^\bullet$ 
fitting in 
diagram \eqref{equation-(1)} so that the relative Picard number 
$\uprho(S^\bullet/ S)$ is minimal. Let $\Delta^\bullet\subset S^\bullet$ be the 
discriminant curve. 
Then $\Delta^\bullet$ and the the exceptional divisor $\alpha^{-1}(o)$ have no common 
components. 
\end{lemma}
\begin{proof}
According to Corollary \ref{cor:disc} we have $\Delta^\bullet\le \alpha^* \Delta$ and $\alpha(\Delta^\bullet)= \Delta$. 
The curve germ $o\in\Delta$ locally can be 
given by the equation $y^2-x^{n}=0$ and the inverse image $\alpha^{-1}(o)$ is a 
tree of rational curves. Let $\Delta'\subset S^\bullet$ be the proper transform of 
$\Delta$. Clearly, $\Delta'\subset \Delta^\bullet$.
If $\Delta'$ is singular at a point $P\in \alpha^{-1}(o)$, then 
$\Delta'\cap \alpha^{-1}(o)=\{P\}$ and none of the components of $\alpha^{-1}(o)$
are contained in $\Delta^\bullet$ by Corollaries~\xref{normal-crossing} and~\xref{discriminant-divisor-pa=0}.
From now on we assume that $\Delta'$ is smooth.

Suppose that $n$ is even. Then $\Delta$ has two smooth analytic branches at $o$ 
and their proper transforms $\Delta^\bullet_1$ and $\Delta^\bullet_2$ on $S^\bullet$
are disjoint.
Furthermore, $\Delta^\bullet_1$ and $\Delta^\bullet_2$ are connected by a chain of smooth rational curves $E_1,\dots, 
E_n\subset \alpha^{-1}(o)$ and all components of $\alpha^{-1}(o)$ other than 
$E_1,\dots, E_n$ are not contained in $\Delta^\bullet$ by 
Corollary~\xref{discriminant-divisor-pa=0} and \eqref{equation-condition-S-star-star}. Thus we may assume that $\Delta^\bullet$
contains some $E_i$. 
Again by \eqref{equation-condition-S-star-star}
\begin{equation*}
\Delta^\bullet=\Delta^\bullet_1+E_1+\cdots+ E_n+\Delta^\bullet_2, \quad 
E_i\subset \Delta^\bullet \quad \forall i=1,\dots,n.
\end{equation*}
Clearly, $\alpha^{-1}(o)$ contains at least one $(-1)$-curve, say $E$, and 
$\Delta^\bullet\cdot E\le 1$. Then by Proposition~\xref{proposition-contraction-(-1)-curve} we can blow down $E$ 
and get a standard model $X'/S'$ with smaller value $\uprho(S'/ S)$. This contradicts our assumptions.

Suppose that $n$ is odd. Then $\Delta'$ is irreducible (and smooth)
and $\Delta'\cap \alpha^{-1}(o)$ is a single point, say $P$.
By \eqref{equation-condition-S-star-star} the curves $\Delta^\bullet$ and $\alpha^{-1}(o)$ have no common 
components.
\end{proof}

\begin{lemma}\label{lemma-Cartier}
Let $f: (X,C)\to (Z,o)$ be a $\QQ$-conic bundle germ.
Assume that there exists an effective Weil divisor $H$ on $X$ such that
the pair $(X,H)$ is canonical and $K_X+H$ is numerically trivial.
Then the divisor $K_X+H$ is Cartier. 
\end{lemma}
\begin{proof}
Assume the converse, i.e. $K_X+H$ is not Cartier at some point
$P\in C\subset X$ of index $m$. 
Since the point $P\in X$ is terminal, $kK_X+H$ is a Cartier divisor in a neighborhood of $P$ 
for some $k$ (see \cite[Corollary 5.2]{Kawamata-1988-crep}).
According to our assumption $k\not \equiv 1\mod m$.
In particular, $m>1$. 
We may assume that $1<k\le m$.

Assume that $H\ni P$. It follows from the main result of
 \cite{Kawamata-1992-e-app}
that there exists an exceptional divisor $E$ over $P\in X$ with discrepancy 
$a(E, X, 0)=1/m$. Since $kK_X+H$ is a Cartier divisor at $P$, 
\[
-ka(E, X, 0) +\mult_E(H) =-k/m+\mult_E(H)\in \ZZ.
\]
Hence, $\mult_E(H)\ge k/m>1/m$ and
\[
a(E, X, H)=a(E, X, 0)  -\mult_E(H) \le 1/m -k/m<0.
\]
This contradicts canonicity of the pair $(X,H)$.
Thus, $H$ does not contain  $P$.

Clearly, the point $P$ must be non-Gorenstein \cite[Lemma 5.1]{Kawamata-1988-crep}. 
Since the pair $(X,H)$ is canonical, it is easy to see from the main result of \cite{Kawamata-1992-e-app} that 
$P\notin H$. 
Since $\Pic(X)\simeq H^2(C,\ZZ)$ and $K_X+H\equiv 0$, the 
divisor $K_X+H$ is a torsion element in the Weil divisor class group
and it defines a non-trivial \'etale in codimension one cover of $(X,C)$
(see \cite[Corollaries 2.3.1 and 2.7.1]{Mori-Prokhorov-2008}).
In particular, $(Z,o)$ is a singular point. 
For any component $C_i\subset C$
we have 
\begin{equation*}
H\cdot C_i=-K_X\cdot C_i<1
\end{equation*}
(see \cite[Lemma 2.8]{Mori-Prokhorov-2008}).
Therefore $H$ is not a Cartier divisor at some point $P_1\in C\subset X$, $P_1\neq P$. 
Thus $X$ has at least two non-Gorenstein points.
By Theorem~\xref{theorem-conic-bundles-singular-base} the germ $(X,C)$ is of type
\xref{item=main--th-pr-toric} or~\xref{item=main--th-pr-ex3}. Then $C$ is irreducible
and $H\cap C=\{P_1\}$. 
The torsion subgroup  $\Cl(X)_\tors$ of the Weil divisor class group  $\Cl(X)$ is cyclic
\cite[Corollary~2.7.1]{Mori-Prokhorov-2008}. 
One can see from the constructions that in both cases \xref{item=main--th-pr-toric} and~\xref{item=main--th-pr-ex3}
the restrictions $\Cl(X)_\tors\to \Cl(X,P)$ and $\Cl(X)_\tors\to \Cl(X,P_1)$ to the local class groups 
are isomorphisms. 
In particular, this means that $K_X+H$ is not a Cartier divisor at $P_1$.
According to the discussions in the beginning of the proof we have $H\not\ni P_1$, a contradiction.
\end{proof}

\section{Examples of Sarkisov links on $\QQ$-conic bundles}
\label{section:ex}
In this section we consider Sarkisov links on non-Gorenstein $\QQ$-conic bundles.
Such links are elementary steps in the decomposition of birational 
transformations of conic bundles (see \S~\xref{section:I}).
For applications it would be very useful to have their classification or
at least substantive theorems, analogous to Theorem~\xref{classification-Sarkisov-links}, 
describing their structure. Since very little is 
known in this direction, we shall restrict ourselves to a number of examples.

In this section we consider some
examples of Sarkisov links on non-Gorenstein $\QQ$-conic bundles.
We use the following natural construction. 

\begin{construction}
\label{construction-examples-SL}
Let $\pi: (X,C)\to (S,o)$ be a $\QQ$-conic bundle germ, let $C_1,\dots, C_r$ be irreducible components
of $C$, and let $P\in C\subset X$ be a point of index $m_1>1$. 
We consider only the cases described in
Theorems~\xref{theorem-index=2-conic-bundles} or~\xref{theorem-conic-bundles-singular-base}.
In all these cases all the curves $C_i$ pass through $P$. 
Then for any $i=1,\dots,r$  we have 
\begin{equation}
\label{equation-computation-K.C-1}
-K_X\cdot C_i=
\begin{cases}
1/2& \text{in the cases 
\xref{theorem-index=2-conic-bundles},}
\\
& \text{\xref{item-main-th-impr-barm=2-s=4},~\xref{item-main-th-impr-barm=2-s=2-cycl}, 
\xref{item=main--th-cyclic-quo},}
\\
1&\text{in the case~\xref{item-main-th-impr-barm=1},}
\\
2/m_1&\text{in the case~\xref{item=main--th-pr-toric},}
\\
1/m_1&\text{in the case~\xref{item=main--th-pr-ex3}.}
\end{cases}
\end{equation}
These equalities can be checked directly or deduced from~\cite[Lemma 2.8]{Mori-Prokhorov-2008}
and~\cite[(2.3), (4.9)]{Mori-1988}.
According to~\cite{Kawamata-1992-e-app}, 
\cite{Hayakawa-1999}, \cite{Hayakawa-2000} there exists an extraction 
$p: Z\to X$ in Mori category with exceptional divisor $E$ such that $p(E)=P$ and 
the discrepancy of $E$ equals $1/m_1$. Let 
$\tilde C_i\subset Z$ be the proper transform of $C_i$. Then
\begin{equation}
\label{equation-computation-K.C-2}
K_Z\cdot \tilde C_i=K_X\cdot C_i+\textstyle{\frac1{m_1}}E\cdot \tilde C_i.
\end{equation}
Assume that 
\begin{equation}\label{eq:assu}
K_Z\cdot\tilde C_i\le 0 
\end{equation}
(thus $Z$ is the central object of the 
corresponding link~\cite{Shokurov-Choi-2011}). 
This inequality holds and can be checked directly in many cases.
Then there exists a flop or a flip along $\tilde C_i$ and we can run the MMP on $Z$
over $S$ in the direction different from the contraction to $X$.
After some number of flops and flips 
\begin{equation}
\label{eq:flip}
\chi: Z \dashrightarrow Z_1
\end{equation}
we get a non-small Mori extremal contraction. There are two possibilities: either 
this contraction $q: Z_1\to X_1$ is divisorial or it is a $\QQ$-conic bundle $\pi_1: Z_1=X_1\to S$. 
In other words, we obtain a link of type~\typem{II} or of type~\typem{I},
respectively (see~\xref{definition-Sarkisov-links}). Moreover, for a type~\typem{II} link, 
the morphism $q$ contracts the proper transform of $E$. 
Thus $X$ and $X_1$ are isomorphic in codimension one.
Since divisors $-K_X$ and $-K_{X_1}$ are ample over $S$, the varieties $X$ and $X_1$ are,
in fact, isomorphic. If in the case of a type~\typem{I} link
the base $(S,o)$ is singular, then the morphism $\alpha: S_1\to S$
must be a crepant contraction according to the following fact. 
\end{construction}

\begin{proposition}[\cite{Morrison-1985}]
\label{proposition-Morrison}
Let $S'\to S''$ be a proper birational contraction between 
surfaces with Du Val singularities. Then 
$\alpha$ can be decomposed as a sequence $\alpha_i: S_i\to S_{i+1}$, where 
each $\alpha_i$ is either a crepant contraction or a 
weighted blowup with weights $(1,n_i)$ of a smooth point.
\end{proposition}

\subsection{}
Recall that \emph{Shokurov's difficulty} $\dif(V )$ of a variety $V$ with terminal 
singularities is defined as the number of exceptional divisors on $V$ with 
discrepancy $<1$ (see~\cite[Definition 2.15]{Shokurov-1985}). It is known that this number is 
well-defined and finite. Moreover, in the three-dimensional case, it is strictly
decreasing under flips \cite[Corollary 2.16]{Shokurov-1985}. If $V \ni P$ is a terminal cyclic 
quotient of index $m$, then $\dif(V \ni P)= m-1$. Thus, in our case
\eqref{eq:flip} we have
\begin{equation}
\label{eq:flip-diff}
\dif(Z_1)\ge \dif(Z)
\end{equation}
and the inequality is strict 
if $\chi$ contains at least one flip.

Below we consider explicit examples.
We use the notation of Construction~\xref{construction-examples-SL}.
\begin{example}\label{example-simplest-link}
Let $(X,C)\subset \PP(1,1,1,2)$ be the index two $\QQ$-conic bundle germ 
given by 
\begin{equation}
\label{eq:index2-s}
\left\{
\begin{array}{lll}
y_1^2-y_2^2&=&uy_4
\\[7pt]
y_1y_2-y_3^2&=&vy_4.
\end{array}
\right.
\end{equation}
(see \eqref{eq-eq-index2}) and let $\pi: (X,C)\to (S,o)=(\CC^2,0)$ 
be the corresponding contraction. Then $X$ has a unique singular point
$P=(0,0,0,1; 0,0)$
which is terminal quotient of type $\frac12 (1,1,1)$.
The central curve has four components $C_1,\dots, C_4$.
All of them pass through $P$ and they do not meet each other 
elsewhere. 
The discriminant curve is given by $(u^2+4v^2)u=0$. Let $p: Z\to X$ be the 
blowup of $P$, let $E$ be the exceptional divisor, and let 
$\tilde C_i\subset Z$ be the proper transform of $C_i$. Then $Z$ is smooth,
$E\simeq \PP^2$, and $\OOO_E(E)\simeq \OOO_{\PP^2}(-2)$
(cf.~\cite{Mori-1982}). Furthermore, the curves $\tilde C_1$, \dots $\tilde C_4$ 
are disjoint. By \eqref{equation-computation-K.C-1} and \eqref{equation-computation-K.C-2} 
we have $K_Z\cdot \tilde C_i=0$. There exists a flop $\chi: Z \dashrightarrow X_1$ with center $\tilde C$
which is the simplest Atiyah-Kulikov flop along each curve $\tilde C_i$. Since $-K_{X_1}$ is nef, there exists 
a Mori extremal contraction $\pi_1: X_1\to S_1$ over $S$. The restriction 
$\chi_E: E \dashrightarrow E_1\subset X_1$ is the inverse to the blowup of four points 
$E\cap \tilde C_i$. Hence, the proper transform 
$E_1\subset X_1$ of $E$ is a del Pezzo surface of degree $5$.
Thus $\pi_1: X_1\to S_1$ is a standard conic bundle (see e.g.~\cite{Mori-1982}).
Then we get a type~\typem{I} link, where $S_1$ is a smooth surface and 
$\alpha: S_1\to S$ is the blowup of $o$.
The discriminant curve $\Delta_1$ of $\pi_1$ is the proper transform of $\Delta$
(and so $\Delta_1$ has three smooth irreducible branches near $\alpha^{-1}(o)$).
\end{example}

\begin{example}\label{example-Sarkisov-link-toric}
Let $\pi: (X,C\simeq \PP^1)\to (S,o)$ be the toric $\QQ$-conic bundle germ 
as in~\xref{item=main--th-pr-toric}. 
In this case the discriminant curve is empty.
Consider the Kawamata weighted blowup
$p: Z\to X$ 
of the point $P\in X$ of type $\frac 1m (1,a,m-a)$ (see~\cite{Kawamata-1996}).
It is easy to see that $Z$ has on $E$ two terminal quotient singularities 
of types 
\begin{equation}
\label{toric-2-singularities}
\textstyle\frac 1a(1,-m,m)\quad\text{ and}\quad \frac 1{m-a}(1,a,-m), 
\end{equation}
and $E\cap \tilde C$
is a smooth point of $Z$. By \eqref{equation-computation-K.C-1} and \eqref{equation-computation-K.C-2}
we have $K_Z\cdot \tilde C=-1/m$.
Hence $\tilde C$ is a flipping curve. Run the MMP on $Z$ over $S$
as in \eqref{eq:flip}.
We have 
\begin{equation}
\label{Shokurovs-difficulty-compare}
\dif(Z)=2m-3, \qquad \dif(Z_1)\le 2m-4. 
\end{equation}
If we are in the situation of type~\typem{II} link, then $\pi_1: X_1\to S$ 
must be $\QQ$-conic bundle of the same type as $\pi$
(because its discriminant divisor is trivial).
On the other hand, $\dif(X_1)\le \dif(Z_1)+1\le 2m-3$, so $\dif(X_1)< \dif(X)$, a contradiction. 
Therefore, we have a type~\typem{I} link and 
$\alpha: S_1\to S$ is a crepant blowup of $o\in S$ (see Proposition~\xref{proposition-Morrison}).
Thus $S_1$ has Du Val singularities of types \type{A_{m_1-1}} and \type{A_{m_2-1}}
with $m_1+m_2=m$. We may assume that $m_1\ge m_2$ (it is possible that $m_2=1$). 
The discriminant divisor of $\pi_1$ is trivial, so $\pi_1$ is also of type~\xref{item=main--th-pr-toric}.
More precisely, $\pi_1$ has two (or one) germs of type~\xref{item=main--th-pr-toric}
which are quotients by $\mumu_{m_1}$ and $\mumu_{m_2}$. 
In this case, 
\begin{equation*}
\dif(X_1)=2m_1-2+2m_2-2=2m-4. 
\end{equation*}
Thus in 
\eqref{Shokurovs-difficulty-compare} the equality holds.
This implies that $\chi$ is a single flip and so $X_1$ has 
singularities of types \eqref{toric-2-singularities}.
This means that up to permutation we have $m_1=a$ and $m_2=m-a$.
\end{example}

\begin{example}
Let $\pi: (X,C\simeq \PP^1)\to (S,o)$ be the $\QQ$-conic bundle germ 
as in~\xref{item=main--th-pr-ex3}. 
In this case the discriminant curve $\Delta$ is given by $\{uv=0\}/\mumu_m$.
In particular, the pair $(S,\Delta)$ is log canonical. 
Consider the Kawamata weighted blowup
$p: Z\to X$ of the point $P\in X$ of type $\frac 1m (a+1,1,-1)=\frac1m (1,2,m-2)$.
Then $E\cap \tilde C$ is a singularity of type $\frac 12(1,1,1)$ on $Z$. 
By \eqref{equation-computation-K.C-1} and \eqref{equation-computation-K.C-2} we have $K_Z\cdot \tilde C=-1/2m$.
Running the MMP as in Construction~\xref{construction-examples-SL}
we obtain a type~\typem{I} Sarkisov link, where $\pi_1$ is a $\QQ$-conic bundle
and $\alpha: S_1\to S$ is a crepant contraction.
For the discriminant curve $\Delta_1\subset S_1$ we have 
$\Delta_1=\alpha^{*}(\Delta)_{\red}$ (see Corollary~\xref{corollary-Q-conic-bundles-plt})
and the pair $(S_1,\Delta_1)$ is lc. This implies that $\pi_1$ cannot contain
germs of type~\xref{item=main--th-pr-toric}.
The same computations with difficulty as in Example~\xref{example-Sarkisov-link-toric}
show that $\chi$ is a single flip and so $X_1$ contains a singularity of type 
$\frac{1}{m-1}(1,2,-2)$. Therefore, $S_1$ has two Du Val points of types \type{A_{m-3}}
and \type{A_1} and over \type{A_{m-3}} we have a germ of type~\xref{item=main--th-pr-ex3}
of index $m-2$. Then by the classification~\xref{theorem-conic-bundles-singular-base}
for the germ over \type{A_1} there is only one possibility: case~\xref{item-main-th-impr-barm=1}
with $\psi=uv$.
\end{example}

\begin{example}
\label{example-imp-4}
Consider a $\QQ$-conic bundle germ $\pi: (X,C)\to(S,o)$ of type~\xref{item-main-th-impr-barm=2-s=4}
which is the $\mumu_4$-quotient of $X'\subset \PP(1,1,1,2)$ defined by \eqref{eq:index2-s}.
The discriminant curve $\Delta$ is given by $\{(u^2+4v^2)u=0\}/\mumu_4$
on $S=\CC^2/\mumu_4(1,-1)$. One can see that the dual graph of the minimal resolution 
of $(S\supset \Delta\ni o)$ has the following form 
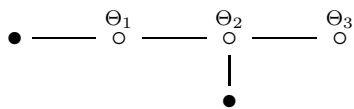
\begin{equation*}
\xymatrix@R=7pt{
\bullet\ar@{-}[r]&\overset{\Theta_1}\circ\ar@{-}[r]&\overset{\Theta_2}\circ\ar@{-}[r]\ar@{-}[d]&\overset{\Theta_3}\circ
\\
&&\bullet
} 
\end{equation*}
where the vertices $\circ$ correspond to (crepant) exceptional divisors over $o$ and the vertices 
$\bullet$ correspond to the components of $\Delta$. Consider the Kawamata weighted blowup
$p: Z\to X$ of the point $P\in X$ of type $\frac 18 (5,1,3)$.
Then $Z$ has two cyclic quotient points $P_1$ and $P_2$ of types $\frac 15(-3,1,3)$
and $\frac 13(2,1,1)$. The curve $\tilde C$ passes through $P_1$.
Using \eqref{equation-computation-K.C-1} and \eqref{equation-computation-K.C-2} we can compute
$K_Z\cdot \tilde C<0$. Thus there is a flip along $\tilde C$.
We have $\dif(Z)=6$ and $\dif(Z_1)\le 5$. As above we get a type~\typem{I} link, 
where $\alpha: S_1\to S$ is a crepant contraction. 
Let $\Delta_1\subset S_1$ be the discriminant curve. If $\alpha$ is the blowup of $\Theta_2$,
then $S_1$ has two points, say $o_1$ and $o_2$, of type~\type{A_1}.
By Corollary~\xref{corollary-Q-conic-bundles-plt} we have $\alpha^{-1}(o)\subset \Delta_1$.
But then the pair $(S_1,\Delta_1)$ is plt at one the $o_i$.
This is impossible again by Corollary~\xref{corollary-Q-conic-bundles-plt}. Thus $S_1$ has a point $o_1$, of type~\type{A_2}.
Then by the classification~\xref{theorem-conic-bundles-singular-base} the germ over $o_2$ is of type 
\xref{item=main--th-pr-ex3} and so the pair $(S_1,\Delta_1)$ is lc at $o_1$.
Therefore, $\alpha^{-1}(o)\not\subset \Delta_1$ and $\alpha$ is the extraction of $\Theta_3$.
\end{example}

\begin{example}[{\cite{Avilov2014}}]
\label{example-Sl-Gorenstein}
Consider a $\QQ$-conic bundle germ $\pi: (X,C)\to(S,o)$ such that $(S,o)$ is smooth
and the discriminant curve $\Delta$ has a node at $o$. Denote $U:=S\setminus \{o\}$ and 
$V:=X\setminus C=\pi^{-1}(U)$. Let $j: U \hookrightarrow S$ be the embedding.
The sheaf $\EEE_0:=\pi_*\OOO_V(-K_X)$ is locally free on $U$. It has a unique extension to
$S$ as a locally free sheaf. Indeed, the sheaf $(j_*\EEE_0)^{\vee\vee}$ is reflexive
and therefore it is locally free. Clearly, $V$ is embedded to $\PP_U(\EEE_0)\subset \PP_S(\EEE)$.
Let $X' \subset \PP_S(\EEE)$ be the closure of $V$. We may assume that $(S,o)$ is an 
open (analytic) subset of $\CC^2_{u,v}$ and $\Delta$ is given by $uv=0$.

Now assume that $X$ is singular.
Then easy computations 
show that $X'$ can be given locally by the equation 
\begin{equation*}
uvx_0^2+x_1^2+x_2^2=0
\end{equation*}
in $\PP_S(\EEE)=\PP^2_{x_0,x_1,x_2}\times \CC^2_{u,v}$.
In particular, the projection $\pi': X'\to S$ is flat and $X'$ has a unique singular point $P\in X'$ which is a node.
On the other hand the varieties $X'$ and $X$ are isomorphic in codimension one over $S$ and their anti-canonical divisors 
are relatively ample. Hence the varieties $X'$ and $X$ are isomorphic over $S$ and we can identify them.
In particular, $X$ can be embedded to $\PP_S(\EEE)$.

Let $p: Z\to X$ be the blowup of $P$. Then $Z$ is smooth and $-K_Z$ is nef and big over $S$.
There exists a flop $Z \dashrightarrow X_1$ with center the proper transform of $C$
and on $X_1$ there exists a Mori extremal contraction $\pi_1: X_1\to S_1$ over $S$.
Since $X_1$ is smooth, so $S_1$ is.
Thus we obtain a type~\typem{I} Sarkisov link, where $\alpha: S_1\to S$ is the blowup of $o$.
Moreover, the exceptional curve $\alpha^{-1}(o)$ is not contained in the discriminant divisor $\Delta_1\subset S_1$
of $\pi_1$ (otherwise we are in the situation of Proposition~\xref{proposition-contraction-(-1)-curve}\xref{cases-cb-3}
and $X$ is smooth). Therefore, $\Delta_1$ is the proper transform of $\Delta$.
\end{example}

\begin{scorollary}\label{corollary-lc}
Let $\pi: (X,C)\to(S,o)$ be a $\QQ$-conic bundle germ 
and let $\Delta$ be the discriminant curve. Assume that 
the pair $(S,\Delta)$ is lc but not plt at $o$.
Then 
either $\pi$ is a standard conic bundle, or $\pi$ is as in 
Example~\xref{example-Sl-Gorenstein}, or as in 
\xref{item=main--th-pr-ex3} or~\xref{item-main-th-impr-barm=1} with $\psi=uv$.
\end{scorollary}

\begin{proof}
By Theorem~\xref{singularities-base-DV} the germ $(S,o)$ is 
the quotient of a smooth germ $(S',o')$ by $\mumu_m$, $m\ge 1$. 
Let $X'$ be the normalization of the fiber product $X\times_S S'$
(see~\cite[(2.4)]{Mori-Prokhorov-2008}, 
\cite[Construction 1.9]{Prokhorov-1997_e}).
Then the projection $\pi': X'\to S'$ is a $\QQ$-conic bundle germ 
over smooth base. By Example~\xref{example-Sl-Gorenstein} the variety $X'$ is Gorenstein.
Then the assertion follows from \cite[Theorem 2.4]{Prokhorov-1997_e}. 
\end{proof}

\begin{scorollary}
Let $\pi: (X,C)\to(S,o)$ be a $\QQ$-conic bundle germ 
such that its discriminant curve $\Delta$ is non-empty and smooth at $o$.
Then both $X$ and $S$ are smooth and $\pi$ is a standard conic bundle.
\end{scorollary}
\begin{proof}
By Theorem~\xref {theorem-conic-bundles-singular-base} the base $(S,o)$
is either smooth or Du Val of type~\type{A_1} or \type{A_3}. 
In these cases the pair $(S,\Delta)$ must be lc because the curve $\Delta$ is smooth. Then the assertion follows from 
Corollaries~\xref{corollary-lc} and~\xref{corollary-Q-conic-bundles-plt}.
\end{proof}

A. Avilov~\cite{Avilov2014} used Sarkisov links to show the existence of standard models 
for three-dimensional $\QQ$-conic bundles over a non-closed field of characteristic $0$
(cf. Theorem~\xref{theorem-standard-models}).
His construction is canonical and works in the category of $G$-varieties. In particular,
he used the following easy observation.

\begin{lemma}[{\cite[Lemma 3]{Avilov2014}}]
\label{lemma-avilov}
Let $\pi: (X,C)\to (S,o)$ be a $\QQ$-conic bundle germ 
with singular $(S,o)$. Then there exists a Sarkisov link 
of type~\typem{I}, where $\alpha: S_1\to S$ is a crepant blowup of $o$. 
\end{lemma}

\begin{proof}[Outline of the proof]
Take a linear system $\LLL$ of hyperplane sections of $S$ passing trough $o$ and put $\MMM:=\pi^*\LLL$.
Take $c$ so that the pair $(X,c\MMM)$ is maximally canonical (in other words, $c=\ct(X, \MMM)$ is 
the canonical threshold of $(X,\MMM)$). Then, as in~\xref{Sarkisov-program-untwisting},
there exists a log-crepant extremal blowup 
$p: (Z,c\MMM_Z)\to (X,c\MMM)$. Here $\uprho (Z/S)=2$ and
\begin{equation*}
K_Z+c\MMM_Z\equiv p^*(K_X+c\MMM).
\end{equation*}
Moreover, $p$ is a divisorial contraction, and the variety $Z$ is $\QQ$-factorial and has only terminal singularities. 
Then we run the $(K_Z+c\MMM_Z)$-MMP over $S$. Since $\dim S=2$, we 
obtain a link of type~\typem{I}. By Proposition~\xref{proposition-Morrison} 
the contraction $\alpha: S_1\to S$ is a crepant blowup.
\end{proof} 

Avilov's proof  of the existence of standard models 
is based on (weak) Sarkisov program and essentially uses
Theorem~\xref{theorem-ge}.
In fact, since the number of crepant divisors on $S$ is finite, any sequence of 
transformations as in Lemma~\xref{lemma-avilov} terminates. Therefore, 
applying such transformations one can resolve the singularities of the base.

\section{Birational transformations of $\QQ$-conic bundles, I}
\label{section:I}
Below we basically follow Iskovskikh's paper \cite{Iskovskikh-1996-conic-re} with some improvements.
\begin{theorem}\label{theorem-main-construction}
Let $\pi: X\to S$ be a 
$\QQ$-conic bundle over a projective rational surface $S$.
Assume that there exist another Mori fiber space $\pi^{\sharp}: X^{\sharp}\to S^{\sharp}$
and a birational
map 
\begin{equation}
\label{equation-(4)}\vcenter{
\xymatrix{
X\ar@{-->}[r]^{\Phi}\ar[d]^{\pi}&X^{\sharp}\ar[d]^{\pi^{\sharp}}
\\
S & S^{\sharp}
}}
\end{equation}
which cannot be completed to a commutative square. 
Then there exists the following commutative diagram
\begin{equation}
\label{diagram-main-construction}
\vcenter{
\xymatrix@C=30pt@R=10pt{
&\hat X\ar@{-->}[dl]\ar@{-->}[dr]
\ar[d]^{\hat \pi}
\\
X\ar[d]_{\pi}&\hat S\ar[dr]\ar[dl]&\bar X\ar[d]^{\bar \pi}
\\
S&&\bar S
} }
\end{equation}
where $\hat X/\hat S$ and $\bar X/\bar S$ are $\QQ$-conic bundles, 
$\hat X/\hat S \dashrightarrow X/S$ is a sequence of Sarkisov links of 
types \typem{I} or \typem{II}, $\hat X/\hat S \dashrightarrow \bar X/\bar S$ is a sequence of Sarkisov links of 
types \typem{III}, and the contractions $\hat S\to S$ and $\hat S\to \bar S$ are described in Proposition~\xref{proposition-Morrison}. 
Furthermore, let $\bar \Delta\subset \bar S$ be the discriminant curve of $\bar \pi: \bar X\to \bar S$.
Then one of the following holds:
\begin{enumerate} 
\item \label{t-Sarkisov-program-apply-1}
$\uprho(\bar{S})=1$ and the divisor $-(4K_{\bar{S}}+\bar{\Delta})$ is ample;

\item \label{t-Sarkisov-program-apply-2}
$\uprho(\bar{S})=2$ and there is a base point free
pencil of rational curves $\bar{\LLL}$ 
such that 
\begin{equation}
\label{equation-transformations-rho=2}
2\le \bar{\LLL}\cdot\bar{\Delta} < -4\bar{\LLL}\cdot K_{\bar{S}}=8.
\end{equation}
\end{enumerate}
Moreover, if $\pi$ is a standard conic bundle, then we have 
\begin{equation*}
\p(\hat \Delta)=\p(\Delta)\le \p(\bar \Delta).
\end{equation*}
\end{theorem}

\begin{proof}
We apply Sarkisov program (see Theorem~\xref{theorem-Sarkisov-program}).
First, we untwist the maximal singularities.
This means that we apply a sequence of
type~\typem{I} and \typem{II} links as explained in~\xref{Sarkisov-program-untwisting}:
\begin{equation}
\label{equation-(5)}
\vcenter{
\xymatrix{
X\ar[d]^{\pi_{}}&X_{1}\ar@{-->}[l]_{\Phi_{1}}\ar[d]^{\pi_{1}}
&\cdots\ar@{-->}[l]_{\Phi_{2}}& X_{m-1}\ar[d]^{\pi_{m-1}}\ar@{-->}[l]_{\Phi_{m-1}}& 
*+[r]{\hat X=X_m}\ar[d]^{\hat \pi}\ar@{-->}[l]_{\Phi_m}
\\
S&S_{1}\ar[l]_{\alpha_{1}}&\cdots\ar[l]_{\alpha_{2}}& S_{m-1}\ar[l]_{\alpha_{m-1}}& *+[r]{\hat S=S_m}\ar[l]_{\alpha_m}
}}
\end{equation}
Here $\alpha_i$ are birational contractions or isomorphisms and each of the squares
$(\Phi_i,\alpha_i)$ is commutative.
Moreover, the proper transforms $\HHH_{X_i}$ of the linear system $\HHH_X$ have the
form $\HHH_{X_i}=-\mu K_{X_i}+\pi_i^*A_i$, where $A_i$ is a $\QQ$-divisor on $S_i$.
On the last step, the linear system $\HHH_{\hat{X}}$ has no maximal
singularities.

Next, we follow the step~\xref{Sarkisov-program--no-max-sing} of the Sarkisov program. 
The divisor $\hat{A}$ is not nef by 
the Noether-Fano inequality (Theorem~\xref{Noether-Fano}), where $\hat A$, as usual, satisfies the relation
\begin{equation}
\label{equation-H-hat}
\HHH_{\hat{X}}\equiv -\mu K_{\hat{X}}+\hat{\pi}^*\hat{A}.
\end{equation}
Therefore,
there is a type~\typem{III} or type~\typem{IV} link $\Phi_{m+1}: \hat X/\hat S \dashrightarrow X_{m+1}/S_{m+1}$.
Assume that this link 
is of type~\typem{III} with birational contraction $\alpha_{m+1}: \hat S\to S_{m+1}$.
Then $\pi_{m+1}: X_{m+1}\to \hat{S}_{m+1}$ satisfies the same conditions
as $\hat{\pi}: \hat{X}\to \hat{S}$
and the procedure may be repeated:
\begin{equation}\label{equation-(6)}
\vcenter{
\xymatrix{
*+[l]{X_m=\hat X}\ar@{-->}[r]^{\Phi_{m+1}}\ar[d]^{\hat \pi}&X_{m+1}\ar@{-->}[r]^{\Phi_{m+2}}\ar[d]^{\pi_{m+1}}
&\cdots\ar@{-->}[r]^{\Phi_{n}}& *+[r]{\bar X=X_n}\ar[d]^{\bar\pi}
\\
*+[l]{ S_m=\hat S}\ar[r]^{\alpha_{m+1}}&S_{m+1}\ar[r]^{\alpha_{m+2}}&\cdots\ar[r]^{\alpha_{n}}& *+[r]{\bar S=S_n}
}}
\end{equation}
such that all the $\alpha_{m+i}$, $i=1,\dots, n-m$, are birational contractions,
$\uprho(S_{m+i}/S_{m+i+1})=1$, each of the squares $(\Phi_{m+i},\alpha _{m+i})$ is commutative, the proper
transforms of $\HHH_{\hat{X}}$ on $X_{m+i}$ have the form 
$\HHH_{X_{m+i}}=-\mu K_{X_{m+i}}+\pi^*_{m+i}A_{m+i}$, where
$A_{m+i}$ is some non-nef $\QQ$-divisor on $S_{m+i}$, and $\HHH_{X_{m+i}}$ has no maximal
singularities.

After a finite number of steps we arrive at a link $\bar\Phi: \bar X/\bar S \dashrightarrow X^{\flat}/S^{\flat}$
of type~\typem{III} with a non-birational contraction $\delta=\bar\alpha : \bar S\to T=\bar S_1$ or of type~\typem{IV}. 
As above, for the proper transform $\HHH_{\bar{X}}$ of $\HHH$, we can write
\begin{equation}
\label{equation-HbarX}
\HHH_{\bar{X}}\equiv -\mu K_{\bar{X}}+\bar{\pi}^*\bar{A},
\end{equation}
where $\bar A$ is not nef by the Noether-Fano inequality (Theorem~\xref{Noether-Fano}).
By the projection formula and \eqref{equation-canonical-bundle-formula} we have
\begin{equation}
\label{equation-Lemma-2}
\Gamma:=\pi_*(\HHH_{\bar X}^2)\equiv 4\mu \bar A-\mu^2(4K_{\bar S}+\bar \Delta)
\end{equation}
and this cycle is an effective non-zero integral Weil divisor on $\bar S$.
By the Noether-Fano inequality (Theorem~\xref{Noether-Fano}) the divisor $-\bar A$
is not nef.
We have two possibilities:
\begin{scase} 
\label{item-del-Pezzo}
The link $\bar\Phi: \bar X/\bar S \dashrightarrow X^{\flat}/S^{\flat}$ is of type~\typem{IV} and 
$T$ is a point. Then 
$\uprho(\bar{S})=1$. In this case,
the divisors $-\bar{A}$, $-K_{\bar S}$, and $-(4K_{\bar{S}}+\bar{\Delta})$ are ample by \eqref{equation-Lemma-2}.
\end{scase}
\begin{scase}
\label{item-not-del-Pezzo}
The link $\bar\Phi: \bar X/\bar S \dashrightarrow X^{\flat}/S^{\flat}$ is of type~\typem{III} (resp. \typem{IV}),
where in the notation of~\xref{definition-Sarkisov-links},
the variety $S_1$ (resp. $T$) is a curve.
Then this curve must be smooth rational and $\uprho(\bar{S})=2$.
Thus we have an extremal contraction $\delta: \bar S\to \PP^1$. Let $\bar{\LLL}$ 
be the pencil of fibers. Then $\bar A\cdot\bar{\LLL} <0$ by the construction of the link $\bar\Phi$
(see~\xref{Sarkisov-program--no-max-sing}). Further, $(4K_{\bar S}+\bar \Delta)\cdot \bar{\LLL}<0$ by \eqref{equation-Lemma-2}. 
This proves \eqref{equation-transformations-rho=2}.
\end{scase}

To prove the last equality, consider 
a standard conic
bundle $\pi^{\bullet}: X^{\bullet}\to S^{\bullet}$ that is birationally equivalent to $\hat{\pi}: \hat{X}\to \hat{S}$,
as in Theorem~\xref{theorem-standard-models}.
Let $\Delta^{\bullet}\subset S^{\bullet}$ be the corresponding discriminant curve. Then by Lemma~\xref{Lemma-3},
\begin{equation*}
\p(\Delta^{\bullet})\le \p(\hat{\Delta})\le \p(\Delta)\le \p(\bar \Delta).
\end{equation*}
But $\p(\Delta)=\p(\Delta^{\bullet})$ by Lemma~\xref{Lemma:nK+Delta:invariant}. 
This completes the proof of the theorem.
\end{proof}

As an easy consequence of Theorem~\xref{theorem-main-construction}
we obtain the following.

\begin{theorem}[Sarkisov's theorem {\cite{Sarkisov-1980-1981-e}}]
\label{Sarkisov-theorem} 
Let $\pi: X\to S$ be a standard conic bundle over a rational surface $S$
and let $\Delta$ be the discriminant curve.
Assume that $|4K_S+\Delta|\neq \emptyset$. Then any birational map $X\dashrightarrow X^\sharp$
to another Mori fiber space $\pi^\sharp: X^\sharp\to S^\sharp$ is fiberwise \textup(in particular, 
$\pi: X\to S$ is birationally rigid\textup).
\end{theorem}

Note that the condition $|4K_S+\Delta|\neq \emptyset$ in terms of effective threshold 
is equivalent to $\etr(S,\Delta)\ge 4$ (see~\xref{remark-effective-threshold}).

\begin{proof}
Suppose that $4K_S+\Delta$ is effective. Then by Lemma~\xref{Lemma:nK+Delta:invariant} the 
divisor $4K_{S^{\bullet}}+\Delta^{\bullet}$ on the standard model 
$\pi^{\bullet}: X^{\bullet}\to S^{\bullet}$ is also effective (see 
Theorem~\xref{theorem-standard-models}). Its images under the birational contraction 
$\hat S\to \bar S$ is also an effective $\QQ$-divisor. This is a contradiction to 
Theorem~\xref{theorem-main-construction}\xref{t-Sarkisov-program-apply-1} 
and~\xref{t-Sarkisov-program-apply-2}.
\end{proof}

This theorem was proved by V. Sarkisov in~\cite{Sarkisov-1980-1981-e}.
It has a higher-dimensional generalization \cite[Theorem~4.1]{Sarkisov-1982-e}.
Another proof can be found in \cite{Pukhlikov2000-Essentials},
see also \cite[Corollary 7.3]{Shokurov-Choi-2011} for more general approach.

\begin{remark}
Note that the condition $|4K_S+\Delta|\neq \emptyset$
follows from the condition that the $1$-cycle $-(K_X)^2$ is effective.
\end{remark}

\begin{corollary}
In the notation and assumptions of Theorem~\xref{Sarkisov-theorem} 
the variety $X$ is not birationally equivalent to a $\QQ$-Fano threefold 
and $X$ is not birational to a variety with a structure of $\QQ$-del Pezzo fibration.
In particular, $X$ is not rational.
\end{corollary}

Given rationally connected variety $X$, it is an interesting question 
to describe all birational structures of fibrations on $X$
into varieties with trivial canonical divisor. 
This is an old classical problem raised by M. Halphen \cite{Halphen1882}, \cite{Dolgavcev1966}. 
For conic bundles, the following necessary condition was obtained 
by I. Cheltsov:

\begin{theorem}[\cite{Chelcprimetsov2004}]
\label{Chelcprimetsov-theorem} 
Let $\pi: X\to S$ be a standard conic bundle \textup(of any dimension\textup)
and let $\Delta$ be the discriminant divisor.
If $4K_S+\Delta$ is big, then $X$ is not birationally equivalent to 
a fibration whose general fiber is a smooth variety with 
numerically trivial canonical divisor. In particular, $X$ has no 
birational structure of an elliptic fibration.
\end{theorem}

This fact can be proved in the same style as
\xref{Sarkisov-theorem}. One has to modify the Noether-Fano inequalities 
for more general situation.

\section{Birational transformations of $\QQ$-conic bundles, II}
\label{section:II}
In this section we study the transformation $X/S \dashrightarrow \bar X/\bar S$
obtained in Theorem~\xref{theorem-main-construction} in details. We also discuss various applications to rationality problems.

From Theorem~\xref{theorem-main-construction} and the general structure of the Sarkisov
program one immediately obtains 
the following.

\begin{corollary}
\label{theorem-main-part2-construction}
In the notation of Theorem~\xref{theorem-main-construction} there exists a link 
$\bar\Phi: \bar X/\bar S \dashrightarrow X^{\flat}/S^{\flat}$ of one of the following types:
\begin{enumerate}
\item 
Type \typem{III} over a point. Then $X^{\flat}$ is a $\QQ$-Fano threefold with $\uprho(X^{\flat})=1$.
\begin{equation}
\label{equation-diagram-cd-IIIp}
\vcenter{
\xymatrix@R=10pt{
\bar{X}\ar[d]^{\bar{\pi}}\ar@{-->}[rr]^{\bar{\chi}}&&Z^{\flat}\ar[d]^{q^{\flat}}
\\
\bar{S}\ar[r]&\{\pt\}&X^{\flat}\ar[l]
}}
\end{equation}

\item 
Type \typem{IV} over a point. Then $\pi^{\flat}:X^{\flat}\to S^{\flat}$ is either a $\QQ$-conic bundle
with $\uprho(S^{\flat})=1$ or $S^{\flat}\simeq \PP^1$ and $\pi^{\flat}$ is a $\QQ$-del Pezzo fibration.
\begin{equation}
\label{equation-diagram-cd-IVp}
\vcenter{
\xymatrix@R=10pt{
\bar{X}\ar[d]^{\bar{\pi}}\ar@{-->}[rr]^{\bar{\chi}}&&X^{\flat}\ar[d]^{\pi^{\flat}}
\\
\bar{S}\ar[r]&\{\pt\}&S^{\flat}\ar[l]
}}
\end{equation}
\item 
Type \typem{III} over a curve. Then $\pi^{\flat}:X^{\flat}\to \PP^1$ is a $\QQ$-del Pezzo fibration.
\begin{equation}
\label{equation-diagram-cd-IIIc}
\vcenter{
\xymatrix@R=10pt{
\bar{X}\ar[d]^{\bar{\pi}}\ar@{-->}[rr]^{\bar{\chi}}&&Z^{\flat}\ar[d]^{q^{\flat}}
\\
\bar{S}\ar[dr]^{\bar\alpha}&&X^{\flat}\ar[dl]_{\pi^{\flat}}
\\
&\PP^1&
}}
\end{equation}

\item 
Type \typem{IV} over a curve. Then $\pi^{\flat}:X^{\flat}\to S^{\flat}$ is a $\QQ$-conic bundle.
\begin{equation}
\label{equation-diagram-cd-IVc}
\vcenter{
\xymatrix@R=10pt{
\bar{X}\ar[d]^{\bar{\pi}}\ar@{-->}[rr]^{\bar{\chi}}&&X^{\flat}\ar[d]^{\pi^{\flat}}
\\
\bar{S}\ar[dr]^{\bar\alpha}&&S^{\flat}\ar[dl]_{\alpha^{\flat}}
\\
&\PP^1&
}}
\end{equation}

\end{enumerate}
Moreover the link $\bar\Phi: \bar X/\bar S \dashrightarrow X^{\flat}/S^{\flat}$
decreases the coefficient $\mu$ \textup(see \eqref{definition-mu}\textup), that is, 
\begin{equation}
\label{mu-reduction}
\mu^{\flat} <\mu, 
\end{equation}
where 
$\mu^{\flat}$ is defined by 
\begin{equation*}
\HHH_{X^{\flat}}\equiv -\mu^{\flat}K_{X^{\flat}}+
\pi^{\flat *} A^{\flat}.
\end{equation*}
\end{corollary}

\subsection{}\label{DuVal:delPezzo}
Assume that we are in the situation of \eqref{equation-diagram-cd-IIIp} or \eqref{equation-diagram-cd-IVp}.
Then, in particular, $-K_{\bar{S}}$ is ample, i.e. $\bar S$ 
is a del Pezzo surface with at worst Du Val singularities.
All such surfaces have been
classified (see, for example,~\cite{Miyanishi-Zhang-1988},
\cite[Ch.~8]{Dolgachev-ClassicalAlgGeom}). 
In our case, we have additional restrictions: $\uprho(\bar S)=1$ and the singularities of 
$\bar S$ are of type~\type{A}. 
According to the classification,
there are 28 combinatorial types of Du Val del Pezzo surfaces of Picard number one
and among them there are exactly
16 types with at worst type~\type{A} singularities. We reproduce 
here the complete list, where
the notation \type{nA_m}
specifies $n$ is the number of \type{A_m} points:
\begin{center}
\begin{tabularx}{\textwidth}{l|X}
$K_S^2$ & \multicolumn{1}{c}{$\Sing(S)$}
\\\hline
1&\type{4A_2}, \type{2A_1+2A_3}, \type{2A_4}, \type{A_1+A_2+A_5}, \type{A_1+A_7}, \type{A_8} \\ 
2&\type{A_1+2A_3}, \type{A_2+A_5}, \type{A_7} \\
3& \type{3A_2}, \type{A_1+A_5}\\

4&\type{2A_1+A_3} \\

5& \type{A_4} \\

6& \type{A_1+A_2} ($S\simeq \PP(1,2,3)$)\\

8&\type{A_1} ($S\simeq \PP(1,1,2)$)
\\
9& $\emptyset$\ ($S\simeq \PP^2$)
\end{tabularx}
\end{center}
By Theorem~\xref{theorem-main-construction}\xref{t-Sarkisov-program-apply-1}
the divisor $-(4K_{\bar{S}}+\bar{\Delta})$ is ample.
Clearly, there is only a restricted number
of possibilities for $\bar{\Delta}$.
Note however that $\bar \Delta$ is not necessarily a Cartier divisor on $\bar S$
(see Corollary~\xref{corollary-Q-conic-bundles-}).

By the adjunction formula we have the following.

\begin{slemma}\label{lemma-estimates-pa}
Let $\bar S$ be a del Pezzo surface with at worst Du Val singularities and $\uprho(\bar S)=1$.
Let $d:=K_{\bar S}^2$ and let $\bar \Delta\subset \bar S$ be a reduced curve.
Let $J$ be the ample generator of $\Cl(\bar S)/\mathrm{tors}\simeq \ZZ$. 
Write $-K_{\bar S}\equiv \iota J$ and $\bar \Delta\equiv aJ$
for some positive integers $\iota$ and $a$.
Then one of the following holds:
\par
\begin{center}
\begin{tabularx}{\textwidth}{l@{\qquad}XXX}
$d$& $\bar S$ &$\iota$ &$\p(\bar \Delta)$
\\\hline
&&&
\\[-4pt]
$9$ & $\PP^2$ & $3$ & $(a-3)a/2+1$
\\
$8$ & $\PP(1,1,2)$ & $4$ & $\le (a-4)a/4+1$
\\
$\le 6$ & &$d$ &$\le (a-d)a/2d+1$
\end{tabularx}
\end{center}
\end{slemma}

\begin{scorollary}
\label{corollary-estimate-pa-del-Pezzo}
In the above notation, assume that $-(4K_{\bar{S}}+\bar{\Delta})$ is ample,
then $a<4\iota$ and $\p(\bar \Delta)\le 45$.
\end{scorollary}

\begin{sremark}
\label{remark-estimate-pa-del-Pezzo-deg=5}
In some cases the inequalities above can be significantly improved. For example,
if in the above notation the divisor $-(4K_{\bar{S}}+\bar{\Delta})$ is 
ample and $K_{\bar S}^2=5$, then by Corollary~\xref{corollary-Q-conic-bundles-}
the divisor $\bar \Delta$ on 
$\bar S$ is Cartier, $a\le 15$ and $\p(\bar \Delta)\le 16$.
\end{sremark}

\subsection{}
Now assume that we are in the situation of \eqref{equation-diagram-cd-IVc}
or \eqref{equation-diagram-cd-IIIc}. Then there exists a $K_{\bar S}$-negative extremal contraction
$\bar \alpha: \bar S\to \PP^1:=T$. It is easy to describe the degenerate fibers of 
this fibration. We leave the following statement as an exercise. 

\begin{slemma}
Let $(F,C)$ be the germ of a surface with Du Val singularities of type~\type{A}
along an irreducible reduced curve $C$ such that there exists a $K_F$-negative contraction
$\delta: (F,C)\to (T,o)$ to a curve.
Then $C\simeq \PP^1$ and for the dual graph of the minimal resolution of $(F,C)$ there are only two possibilities:
\begin{equation}
\label{graph-cb-2A1}
\xymatrix{
\circ\ar@{-}[r] & \bullet\ar@{-}[r]&\circ
} 
\end{equation} 
\begin{equation}
\label{graph-cb-A3}
\xymatrix@R=5pt{
\circ\ar@{-}[r] & \circ\ar@{-}[r]&\circ
\\
&\bullet\ar@{-}[u] &
} 
\end{equation} 
where the vertices $\circ$ correspond to exceptional $(-2)$-curves and $\bullet$ corresponds to $C$.
In particular, the singular locus of $F$ consist of two points of type~\type{A_1}
in the case \eqref{graph-cb-2A1}
and one point of type~\type{A_3} in the case \eqref{graph-cb-A3}.
\end{slemma}

\begin{slemma}\label{lemma-last-step-del-pezzo}
Let $\varphi: Y\to Z$ be a $\QQ$-conic bundle over a projective rational surface
and let $\MMM$ be a linear system on $Y$ without fixed components. 
Write $\MMM+\lambda K_Y\equiv \varphi^* B$, where $B$ is a $\QQ$-divisor on $Z$ and $\lambda\in \QQ$.
Assume that there exists a base point free pencil $\LLL$ on $Z$ such that $B\cdot \LLL<0$.
Let $C\subset Y$ be an irreducible curve such that $K_Y\cdot C\ge 0$.
Then one of the following holds:
\begin{enumerate}
\item 
$C\subset \Bs\MMM$,
\item 
$\varphi(C)\cdot \LLL>0$, or
\item 
$\varphi(C)$ is a component of a degenerate fiber of $\LLL$. 
\end{enumerate}
\end{slemma}

\begin{proof}
Clearly, $C$ is not contained in the fibers of $\varphi$. 
We have
\begin{equation*}
0\le \lambda K_{Y}\cdot C=-\MMM\cdot C+\varphi^* B\cdot C.
\end{equation*}
If $C\not \subset \Bs \MMM$, then
$\MMM\cdot C\ge 0$ and so $\varphi^*B\cdot C>0$. In this case, by the projection formula $B\cdot \varphi(C)>0$.
Hence, the curve $\varphi (C)$ cannot coincide with a full fiber of the pencil $\LLL$.
\end{proof}

\begin{scorollary}\label{corollary-del-Pezzo}
In the above notation, let $Y_\eta$ be the generic fiber of the composition 
\begin{equation*}
Y \overset{\varphi}\longrightarrow Z \overset{}\longrightarrow \PP^1.
\end{equation*}
Then $-K_{Y_\eta}$ is ample.
\end{scorollary}

Let $\bar \LLL$ be the base point free pencil of rational curves generated by the fibers of 
$\bar \alpha: \bar S\to \PP^1$.
Let $\bar X_{\eta}$ and $\bar S_\eta$ be 
generic fibers over $T$ of $\bar\alpha\comp\bar\pi:\bar X\to T$ and $\bar\alpha: \bar S\to T$,
respectively. 
Then $\bar X_{\eta}$ is a smooth surface over the non-closed field
$\Bbbk(T)$ and $\bar S_\eta\simeq \PP^1_{\Bbbk(T)}$ is a smooth rational curve.
The contraction $\bar \pi$ induces 
a conic bundle structure $\bar\pi_\eta: \bar X_\eta\to\bar S_\eta$ with
\begin{equation*}
\Pic(\bar X_\eta)=\ZZ\cdot \bar{\Lambda}_\eta\oplus\ZZ\cdot K_{\bar X_\eta}, 
\end{equation*}
where $\bar{\Lambda}_\eta$ 
is the class of the fiber of $\bar\pi_\eta$.
By Corollary~\xref{corollary-del-Pezzo} \ 
$\bar X_\eta$ is a del Pezzo surface with $\uprho(\bar X_\eta)=2$.
Since log flips $\bar \chi$ do not change the generic fiber,
we also have the following.
\begin{scorollary}
\label{corollary-s-Sarkisov-links}
The Sarkisov link $\bar\Phi: \bar X/\bar S \dashrightarrow X^{\flat}/S^{\flat}$ induces a 
link on $\bar X_\eta$ which is of type~\typem{IV} in the case \eqref{equation-diagram-cd-IVc}
and of type~\typem{III} in the case
\eqref{equation-diagram-cd-IIIc}.
\end{scorollary}

\begin{scase}
\label{classification-cases-discriminant-fibration}
Using the classification of two-dimensional Sarkisov links in Theorems~\xref{classification-Sarkisov-links}
and~\xref{classification-Sarkisov-links-conic-bundle}
we obtain the following possibilities:
\begin{enumerate}
\item 
\label{classification-cases-discriminant-fibration-1}
$K_{\bar X_\eta}^2=1$, $\bar{\LLL}\cdot \bar \Delta=7$, $\pi^\flat$ is a $\QQ$-conic bundle, 
the link is of type~\typem{IV} and is induced by the Bertini 
involution $\beta_\eta: \bar X_\eta\to \bar X_\eta$. 
Hence in \eqref{equation-diagram-cd-IVc} we have $\pi^\flat=\lambda\comp\bar \pi\comp \beta$ for some birational maps 
$\beta: \bar X\dashrightarrow \bar X$ and $\lambda: S^\flat\dashrightarrow S^\flat$.

\item 
\label{classification-cases-discriminant-fibration-2}
$K_{\bar X_\eta}^2=2$, $\bar{\LLL}\cdot \bar \Delta=6$, $\pi^\flat$ is a $\QQ$-conic bundle, 
the link is of type~\typem{IV} and is induced by the Geiser
involution $\gamma_\eta: \bar X_\eta\to \bar X_\eta$. 
As above, $\pi^\flat=\lambda\comp\bar \pi\comp \gamma$ for some birational maps 
$\gamma: \bar X\dashrightarrow \bar X$ and $\lambda: S^\flat\dashrightarrow S^\flat$.

\item 
\label{classification-cases-discriminant-fibration-3}
$K_{\bar X_\eta}^2=3$, $\bar{\LLL}\cdot \bar \Delta=5$, the link is of type~\typem{III} and is induced by
a contraction of a $(-1)$-curve on $\bar X_\eta$. $\pi^\flat: X^\flat\to \PP^1$ 
is a $\QQ$-del Pezzo fibration of degree $4$.

\item 
\label{classification-cases-discriminant-fibration-4}
$K_{\bar X_\eta}^2=4$, $\bar{\LLL}\cdot \bar \Delta=4$, $\pi^\flat$ is a $\QQ$-conic bundle, the link is of type~\typem{IV}.
In general, $\bar\Phi$ is not induced by a birational self-map of $\bar X$.

\item 
\label{classification-cases-discriminant-fibration-5}
$K_{\bar X_\eta}^2=5$, $\bar{\LLL}\cdot \bar \Delta=3$, the link is of type~\typem{III} and is induced by
a contraction of four conjugate $(-1)$-curves on $\bar X_\eta$. $\pi^\flat: X^\flat\to \PP^1$ 
is a generically $\PP^2$-bundle.

\item 
\label{classification-cases-discriminant-fibration-6}
$K_{\bar X_\eta}^2=6$, $\bar{\LLL}\cdot \bar \Delta=2$, the link is of type~\typem{III} and is induced by
a contraction of a pair of conjugate $(-1)$-curves on $\bar X_\eta$. $\pi^\flat: X^\flat\to \PP^1$ 
is a generically quadric bundle.
\end{enumerate}
\end{scase}

\begin{sremark}
In the cases~\xref{classification-cases-discriminant-fibration}
\xref{classification-cases-discriminant-fibration-1} and~\xref{classification-cases-discriminant-fibration-2}
we obtain a new conic bundle $\pi^{\flat}:X^{\flat}\to S^{\flat}$ which is fiberwise birational 
to the original one. On the other hand, by \eqref{mu-reduction} we have $\lambda^{\flat}<\lambda$. 
In order to prove the hard part of Conjecture~\xref{conjecture-Iskovskikh}, we may assume that 
the map $X \dashrightarrow \PP^3$ is chosen so that $\mu$ is minimal and then 
these cases~\xref{classification-cases-discriminant-fibration}
\xref{classification-cases-discriminant-fibration-1}--\xref{classification-cases-discriminant-fibration-2} do not occur.
\end{sremark}

\begin{sremark}\label{remark-bar-flat}
In the cases~\xref{classification-cases-discriminant-fibration}
\xref{classification-cases-discriminant-fibration-1},~\xref{classification-cases-discriminant-fibration-2},
and~\xref{classification-cases-discriminant-fibration-4}, by the construction, 
for the pencil $\LLL^\flat$ of fibers of $\alpha^\flat: S^\flat\to \PP^1$ 
and the discriminant curve $\Delta^\flat\subset S^\flat$ we have 
$\LLL^\flat\cdot \Delta^\flat= \bar\LLL\cdot \bar\Delta$
(see Corollary 
\xref{corollary-s-Sarkisov-links}).
\end{sremark}

Now we consider an easy case of Conjecture~\xref{conjecture-Iskovskikh}.
More precisely, we assume that in the notation of the proof of 
Proposition~\xref{theorem-main-construction} the constant $\mu$ equals $1$.
Note that in these notation $\mu\in \frac12 \ZZ$, $\mu>0$ and 
$\mu=1/2$ implies that $\Delta=\emptyset$ and $\pi: X\to S$ is
a $\PP^1$-bundle. 

\begin{proposition}[\cite{Iskovskikh-1987}]
\label{proposition-mu=1}
Let $\pi: X\to S$ be a standard conic bundle with discriminant curve
$\Delta\subset S$. Assume that there exists a birational map $\Phi: X\dashrightarrow \PP^3$
such that for the proper transform $\HHH$ on $X$ of the linear system 
of planes on $\PP^3$ one has $\HHH\equiv -K_X+\pi^*A$. Then, for $\pi: X\to S$, either
the condition~\xref{conjecture-Iskovskikh-1} or the condition~\xref{conjecture-Iskovskikh-2} 
of Conjecture~\xref{conjecture-Iskovskikh} 
is satisfied.
\end{proposition}

This fact was proved by Iskovskikh in \cite[Theorem 2]{Iskovskikh-1987}.
We propose below a slightly different proof.

\begin{proof}
Clearly, we may assume that $\Delta\neq\emptyset$. Apply Theorem~\xref{theorem-main-construction}.
Thus we have the diagram~\eqref{diagram-main-construction} and one of 
that in Corollary~\xref{theorem-main-part2-construction} with all the required properties.
Thus there is a $\QQ$-conic bundle $\bar \pi: \bar X\to \bar S$ that is fiberwise birational to $\pi$
and satisfies the properties in Theorem~\xref{theorem-main-construction}. 
In particular, for the proper transform  $\bar \HHH$ of the linear system 
$\HHH$ on $\bar X$ we have
\begin{equation}
\label{equation-HbarX-a}
\bar \HHH\equiv-K_{\bar X} +\bar \pi^*\bar A
\end{equation}
and the pair $(\bar X, \bar \HHH)$ is canonical
(see the proof of Theorem~\xref{theorem-main-construction}).

Assume that we are in the situation \xref{t-Sarkisov-program-apply-2} of Theorem~\xref{theorem-main-construction}, i.e. $\uprho(\bar S)=2$.
Thus one of the possibilities 
\xref{classification-cases-discriminant-fibration}\xref{classification-cases-discriminant-fibration-1}--\xref{classification-cases-discriminant-fibration-6} 
occurs.
In the cases 
\xref{classification-cases-discriminant-fibration}\xref{classification-cases-discriminant-fibration-1},\xref{classification-cases-discriminant-fibration-2}
and~\xref{classification-cases-discriminant-fibration-4}
the fibration $\pi^\flat$ must be a $\QQ$-conic bundle with non-trivial discriminant curve
by Remark~\xref{remark-bar-flat}.
Since $\mu^\flat<\mu=1$, these cases do not occur.
In the case
\xref{classification-cases-discriminant-fibration}\xref{classification-cases-discriminant-fibration-3}
the fibration $\pi^\flat$ must be a $\QQ$-del Pezzo fibration of degree $4$ with $\mu^\flat<\mu=1$.
Again this is impossible.
In the cases~\xref{classification-cases-discriminant-fibration}\xref{classification-cases-discriminant-fibration-5}
--\xref{classification-cases-discriminant-fibration-6} the proper transform 
of $|\bar{\LLL}|$ on a good model is a desired pencil as in 
Conjecture~\xref{conjecture-Iskovskikh}~\xref{conjecture-Iskovskikh-1}.

Assume now that we are in the case~\xref{t-Sarkisov-program-apply-1} of Theorem~\xref{theorem-main-construction}, i.e. $\uprho(\bar S)=1$.
By \eqref{equation-Lemma-2} we have
\begin{equation}
\label{equation-Lemma-2-mu=1}
4(-K_{\bar S}+\bar A)\equiv \bar \Delta +\Gamma,
\end{equation}
where $\Gamma=\bar \pi_*(\HHH_{\bar X})^2$ is an effective non-zero divisor on $\bar S$,
and $\bar A$ is Cartier by Lemma~\xref{lemma-Cartier}.
Since $\uprho(\bar S)=1$, by the Noether-Fano inequality (Theorem~\xref{Noether-Fano}) the divisor $-\bar A$
is ample.
If $\bar S\not\simeq \PP^2$, $\PP(1,1,2)$, then $\Pic(\bar S)=\ZZ\cdot K_{\bar S}$
and we get a contradiction.

Consider the case $\bar S\simeq\PP(1,1,2)$.
Then by \eqref{equation-Lemma-2-mu=1} the only possibility is $\bar A\equiv\frac 12 K_{\bar S}$ 
and $\bar \Delta< -2K_{\bar S}$. Consider the minimal resolution $\FF_2\to \bar S$.
According to Lemma~\xref{lemma-avilov} there exists a $\QQ$-conic bundle $\pi': X'\to \FF_2$
which completes the map $\FF_2\to \bar S$ to a Sarkisov link of type~\typem{I}. Let $\Delta'\subset \FF_2$ be the discriminant curve.
Write $\Delta'\sim a \Sigma+b\Lambda$, where $\Sigma$ and $\Lambda$ are negative section 
and a fiber of the ruling, respectively. Since $\bar \Delta< -2K_{\bar S}$, we have $b\le 7$.
We may assume that $a\ge 4$ (otherwise $|\Lambda|$ gives us a pencil desired in 
Conjecture~\xref{conjecture-Iskovskikh}~\xref{conjecture-Iskovskikh-1}).
But then $\Sigma$ is a rational component of $\Delta'$ and $(\Delta'-\Sigma)\cdot \Sigma \le 1$.
This is impossible for a discriminant curve (see \eqref{equation-condition-S-star-star}
and~\xref{discriminant-divisor-pa=0}).

Finally, consider the case $\bar S\simeq\PP^2$. 
We may assume that $\deg \bar \Delta\ge 5$.
Again by \eqref{equation-Lemma-2-mu=1} 
we have $\deg\bar \Delta+\deg \Gamma= 8$ and $-\bar A$ is the class on a line $l\subset \PP^2$.
Let $\bar H\in \HHH_{\bar X}$ and $\bar M \in |\bar \pi^* l|$ be general members of the 
corresponding linear systems. 
Thus, by \eqref{equation-HbarX-a},
\begin{equation*}
-K_{\bar X}\sim \bar H+\bar M.
\end{equation*}
By the adjunction formula the divisor $-K_{\bar M}=\bar H|_{\bar M}$ is ample,
i.e. $\bar M$ is a (smooth) del Pezzo surface. Since the restriction 
$\HHH_{\bar X}|_{\bar M}$ defines a birational map to a surface in $\PP^3$, we have
$K_{\bar M}^2\ge 3$. 
Using the projection formula and \eqref{equation-Lemma-2-mu=1} one obtains
\begin{equation*}
3\le K_{\bar M}^2=\bar H^2\cdot \bar M = \deg \Gamma,\qquad \deg \bar \Delta\le 5.
\end{equation*}
By our assumption $\deg \bar \Delta= 5$ and then by the 
Noether formula $K_{\bar M}^2=3$.

Suppose that $\bar \pi$ is not a standard conic bundle.
Thus $\bar X$ is singular at some point of a fiber $\bar\pi^{-1}(o)$, $o\in \bar S$.
Let $\pi^\bullet : X^\bullet\to S^\bullet$ be a standard model of 
$\bar \pi:\bar X\to \bar S$ as in Theorem~\xref{theorem-standard-models}
and let $\Delta^\bullet$ be the corresponding discriminant curve.
Thus we have a birational morphism $\alpha: S^\bullet\to \bar S$ so that
\begin{equation*}
\Delta^\bullet\le (\alpha^*\bar \Delta)_{\red}
\quad \text{and}\quad
\alpha(\Delta^\bullet)=\bar \Delta
\end{equation*}
(see Corollary \ref{cor:disc}).
Let $\bar \LLL$ be the pencil of lines on $\bar S=\PP^2$ passing through $o$
and let $\LLL^\bullet$ be its proper transform on $S^\bullet$. 
By the projection formula 
\[
\LLL^\bullet \cdot \alpha^* \bar \Delta=
\LLL \cdot \bar \Delta=5.
\]
By Corollary~\xref{corollary-Q-conic-bundles-plt} the point $o$ is singular on $\bar\Delta$.
Let $m_o$ the multiplicity of $\bar\Delta$ at $o$. Then a general member of $\LLL^\bullet$
meets some component of $\alpha^* \bar \Delta$ whose coefficient equals $m_o$.
Hence, 
\[
\LLL^\bullet \cdot \Delta^\bullet \le \LLL^\bullet \cdot \alpha^* \bar \Delta-(m_o-1)\le 6-m_o.
\]
If $m_o\ge 3$, then $\LLL^\bullet\cdot\Delta^\bullet\le 3$ and so 
the condition~\xref{conjecture-Iskovskikh-1} 
of Conjecture~\xref{conjecture-Iskovskikh} is satisfied. Thus we may assume that 
$m_o=2$. Then by Lemma 
\xref{lemma-Delta-multiplicity=2} the curves
$\alpha^{-1}(o)$ and $\Delta^\bullet$ have no common components.
Again 
\[
\LLL^\bullet 
\cdot\Delta^\bullet\le  \LLL^\bullet \cdot \alpha^* \bar \Delta -m_o\le3
\]
and 
the condition~\xref{conjecture-Iskovskikh-1} 
of Conjecture~\xref{conjecture-Iskovskikh} is satisfied.
Thus we may assume that $\bar \pi$ is a standard conic bundle.
If the corresponding double cover of $\bar \Delta$ is defined by an odd theta-characteristic,
then $X$ is not rational by Theorems~\xref{theorem-Mumford-Beauville-Shokurov} and~\xref{theorem-Prym-Jacobian}.
Aсcording to Proposition \ref{proposition-Panin} 
we have the case~\xref{conjecture-Iskovskikh-2} of \xref{conjecture-Iskovskikh}.
\end{proof}

Proposition~\xref{proposition-mu=1} shows that
Conjecture 
\xref{conjecture-Iskovskikh} is equivalent to the following classical conjecture of S. Kantor.
Recall that a \textit{congruence} of curves in $\PP^3$ is an irreducible two-dimensional family of curves that
cover $\PP^3$. The \textit{index} of a congruence is the number of curves passing through a 
general point.

\begin{conjecture}[S. Kantor, cf. \cite{Kantor1901}, \cite{Millevoi1960}] 
\label{Conjecture-KantorI}
For any congruence of
index $1$ of rational curves $\mathcal C$ in $\PP^3$ there exists a Cremona transformation 
$\tau: \PP^3 \dashrightarrow \PP^3$ that sends $\mathcal C$
to a two-dimensional family of conics or lines.
\end{conjecture}

In other words, Kantor's conjecture can be formulated as follows.

\begin{sconjecture}
\label{Conjecture-KantorII}
Let $\pi: X\to S$ be a standard conic bundle. Suppose that
there exists a birational map $\Phi: X \dashrightarrow \PP^3$.
Then there exists a Cremona transformation $\tau: \PP^3 \dashrightarrow \PP^3$
such that the composition $\tau\comp\Phi: X \dashrightarrow \PP^3$
sends a general fiber either to a conic or to a line.
\end{sconjecture}

\begin{proposition}
\label{proposition-equi-Kantor}
Conjectures~\xref{conjecture-Iskovskikh} and~\xref{Conjecture-KantorI} are equivalent.
\end{proposition}

\begin{proof}
Let $\pi: X\to S$ be a standard conic bundle and let $\Phi: X \dashrightarrow \PP^3$
be a birational map.

Assume that Conjecture~\xref{Conjecture-KantorI} holds true. By~\xref{Conjecture-KantorII}
we may replace 
$\Phi$ with another birational map $\Phi'=\tau\comp\Phi$ which sends a general fiber either to a conic or to a line.
This means that $\mu\le 1$ (in the notation of the proof of 
Proposition~\xref{theorem-main-construction}). Then by
Proposition~\xref{proposition-mu=1} either 
the condition~\xref{conjecture-Iskovskikh-1} or the condition~\xref{conjecture-Iskovskikh-2} 
of Conjecture~\xref{conjecture-Iskovskikh} 
is satisfied. 

Conversely, assume that Conjecture~\xref{conjecture-Iskovskikh} is true.
If we are in the situation of~\xref{conjecture-Iskovskikh}\xref{conjecture-Iskovskikh-2}, then by 
Proposition~\xref{proposition-Panin} and Example~\xref{example-Panin} modulo Cremona transformations 
we have $\mu=1$.
In the case~\xref{conjecture-Iskovskikh}\xref{conjecture-Iskovskikh-1} by Theorem
\xref{classification-Sarkisov-links-conic-bundle} we again can see that the images of 
the fibers of $X'/S'$ are conics.
\end{proof}

\begin{theorem}[\cite{Iskovskikh-1996-conic-re}]
\label{Iskovskikh-theorem}
Let $\pi: X\to S$ be a standard conic bundle with discriminant curve
$\Delta\subset S$. 
Assume that $\Delta$ is irreducible and smooth.
If $X$ is rational, 
then one of the following holds
\begin{enumerate}
\item \label{Iskovskikh-theorem-genus}
$\p(\Delta)\le 45$, 
\item\label{Iskovskikh-theorem-main-case}
the condition~\xref{conjecture-Iskovskikh-1} of Conjecture~\xref{conjecture-Iskovskikh} 
is satisfied,
\item\label{Iskovskikh-theorem-4}
$\Delta$ is hyperelliptic and $\pi: X\to S$ is fiberwise birationally equivalent to a standard conic bundle 
$\pi^{\bullet}: X^{\bullet}\to S^{\bullet}$ with discriminant curve $\Delta^{\bullet}\simeq \Delta$ such that 
there exists a base point free pencil $\LLL^{\bullet}$ of rational curves on $S^{\bullet}$
such that $\Delta^{\bullet}\cdot \LLL^{\bullet}=4$.
\end{enumerate}
\end{theorem}

Iskovskikh \cite{Iskovskikh-1996-conic-re} proposed a stronger assertion.
However, the work \cite{Iskovskikh-1996-conic-re} contains serious errors and gaps in the proofs 
which the author of the review could not eliminate.

\begin{proof}
If $\Delta=0$, 
then $\pi: X\to S$ is a locally trivial $\PP^1$-bundle and there is nothing to 
prove.
Assume now that $\Delta\neq\emptyset$ and $\Phi: X\dashrightarrow\PP^3$ 
is a birational map
\begin{equation}
\label{equation-(7)}
\vcenter{
\xymatrix{
X\ar@{-->}[r]^{\Phi}\ar[d]^{\pi}&\PP^3\ar[d]
\\
S&\Spec\Bbbk
}}
\end{equation}
Let $\HHH_X\equiv -\mu K_X+\pi^*A$ be the proper transform on $X$ of the linear system 
of planes on $\PP^3$. Since $\Delta\neq\emptyset$, we have $\mu\in\ZZ$, $\mu 
\ge 1$ and~\eqref{equation-(7)} cannot be completed to a commutative diagram that 
induces an isomorphism of fibers over generic points, 
Hence, the hypotheses of Theorem~\xref{theorem-main-construction} are 
satisfied. Hence, we have the diagram~\eqref{diagram-main-construction} and one of 
that in
Corollary
\xref{theorem-main-part2-construction} with all the required properties.

The case $\mu=1$ is a consequence of Proposition~\xref{proposition-mu=1}.

Assume now that the map $\Phi$ chosen in such a way that $\mu$ has the least possible
value. We can thus formulate the following inductive hypothesis:

\begin{scase}
\label{induction-hypothesis*}
Conjecture~\xref{conjecture-Iskovskikh} holds for all standard conic bundles over a 
rational surface for which there is a birational map onto $\PP^3$ with $\mu'<\mu$.
\end{scase}

We can make use of this hypothesis if we succeed in finding a birational
self-map of $X$ that transfers $\pi$ to another birational 
conic bundle structure on $X$ for which the map $\Phi$ has degree $\mu' <\mu$.

Assume that we are in the situation of Theorem~\xref{theorem-main-construction}\xref{t-Sarkisov-program-apply-2}. Let
$\pi^\bullet : X^\bullet\to S^\bullet$ be a standard conic bundle fiberwise birational to 
$\bar \pi: \bar X\to \bar S$ such that $S^\bullet$ dominates $\bar S$ (see Theorem~\xref{theorem-standard-models}).
Let $\Delta^\bullet$ be the corresponding discriminant curve and let 
$|L^{\bullet}|$ be the preimage of the pencil $\bar{\LLL}$ on $S^{\bullet}$. Then 
$L^{\bullet}\cdot\Delta^{\bullet}=\bar{\LLL}\cdot\bar{\Delta}$. 
Consider the possibilities of~\xref{classification-cases-discriminant-fibration} case by case.

In the cases~\xref{classification-cases-discriminant-fibration}\xref{classification-cases-discriminant-fibration-5}
and~\xref{classification-cases-discriminant-fibration}\xref{classification-cases-discriminant-fibration-6}
we have 
$L^{\bullet}\cdot\Delta^{\bullet}\le 3$, so~\xref{conjecture-Iskovskikh-1} of 
Conjecture~\xref{conjecture-Iskovskikh} is true, as stated in the theorem.

In the cases~\xref{classification-cases-discriminant-fibration}\xref{classification-cases-discriminant-fibration-1}
and~\xref{classification-cases-discriminant-fibration}\xref{classification-cases-discriminant-fibration-2} 
the $\QQ$-conic bundle $\pi^\flat : X^\flat\to S^\flat$ is fiberwise birational to $\bar \pi: \bar X\to \bar S$
and for the corresponding map $X^\flat\dashrightarrow \PP^3$ one has
$\mu^\flat<\mu$ (see \eqref{mu-reduction}). Hence by our hypothesis~\xref{induction-hypothesis*}
Conjecture~\xref{conjecture-Iskovskikh} is true for $\pi$.

It remains to consider cases 
\xref{classification-cases-discriminant-fibration}\xref{classification-cases-discriminant-fibration-3}--\xref{classification-cases-discriminant-fibration-4}.
The curve $\Delta^\bullet$ has a component whose normalization is isomorphic to 
$\Delta$. Since $\Delta^\bullet$ satisfies the condition \eqref{equation-condition-S-star-star}
and $\p(\Delta^\bullet)=\p(\Delta)$ by \eqref {eqnarray:D}, it must be isomorphic to $\Delta$.
In the case~\xref{classification-cases-discriminant-fibration}\xref{classification-cases-discriminant-fibration-3}
(resp. the case~\xref{classification-cases-discriminant-fibration}\xref{classification-cases-discriminant-fibration-4})
there is on $\bar{\Delta}$ and thus also on $\Delta^{\bullet}$ a linear series
of degree $5$ (resp. $4$): $q:\Delta^{\bullet}\to \PP^1$. 
Because of the rationality of $X^{\bullet}$ by Theorem~\xref{theorem-Mumford-Beauville-Shokurov}, we obtain a
morphism $p:\Delta^{\bullet}\to\PP^1$ of degree $3$ or $2$. We claim that the morphism
\begin{equation*}
p \times q:\Delta^{\bullet}\to\Delta'\subset\PP^1\times\PP^1,
\end{equation*}
is birational. Indeed, otherwise $q$ passes through $p$:
\begin{equation*}
q: \Delta^{\bullet} \overset p \longrightarrow \PP^1 \overset r \longrightarrow \PP^1
\end{equation*}
and $\deg q= (\deg p)\cdot ( \deg r)$. This is possible only if $\deg q=4$, $ \deg p =\deg r=2$.
In particular, the curve $\Delta^{\bullet}$ is hyperelliptic. 
Then we have the case~\xref{Iskovskikh-theorem-4}.
Thus the curve 
$\Delta'\subset\PP^1\times\PP^1$ is of bidegree $(n,m)$ with $n\le 3$ and $m\le 5$. 
Hence 
\[
\p(\Delta^{\bullet})\le \p(\Delta')=(n-1)(m-1)\le8, 
\]
i.e. we are 
in the case~\xref{Iskovskikh-theorem-genus}. 

Finally, if we are in the case~\xref{t-Sarkisov-program-apply-1} of Theorem~\xref{theorem-main-construction},
then by Corollary~\xref{corollary-estimate-pa-del-Pezzo} we have $\p(\Delta)=\p(\Delta^{\bullet})\le 45$, a contradiction.
The proof of Theorem~\xref{Iskovskikh-theorem} is complete.
\end{proof}

\subsection{}
Note that the inequality $\p(\Delta)\le 45$ in
Theorem~\xref{Iskovskikh-theorem}\xref{Iskovskikh-theorem-genus} can be strengthened
in many cases.
For example, if $K_{\bar S}^2\le 3$, then 
$\p(\Delta)\le 15$ by Corollary~\xref{corollary-estimate-pa-del-Pezzo}. 

If $K_{\bar S}^2=5$, then $\p(\Delta)\le 16$ by Remark~\xref{remark-estimate-pa-del-Pezzo-deg=5}.

In the case $K_{\bar S}^2=9$ we have $\bar S\simeq \PP^2$ and $\deg \bar \Delta\le 11$. 
The pencil of 
lines passing through a sufficiently general point $s\in\bar{\Delta}\subset\PP^2$ cuts out on $\bar{\Delta}$
a one-dimensional linear series of degree $\le 10$. It determines a map $\bar{q}:\bar{\Delta}\to\PP^1$
of degree $\le 10$. Since $s$ is sufficiently general, this linear series is not composed
of a hyperelliptic nor trigonal pencil. The map $\bar{q}:\bar{\Delta}\to\PP^1$ can be
lifted to a finite map $q^{\bullet}:\Delta^{\bullet}\to\PP^1$ whose degree is also $\le 10$. 
Hence, there is a birational morphism
\begin{equation*}
p^{\bullet}\times q^{\bullet}:\Delta^{\bullet}\to\tilde\Delta\subset\PP^1\times\PP^1,
\end{equation*}
where $\tilde\Delta$ is of type $(d_1,d_2)$ with $d_1\le 3$ and $d_2\le 10$. 
The birationality follows from
the fact that $q^{\bullet}$ does not factor through a hyperelliptic nor trigonal series on $\Delta^{\bullet}$. 
Hence, 
\[
\p(\Delta^\bullet)\le (d_1-1)(d_2-1)\le 18.
\]

\section{Some related results and open problems}
\label{section:problems}
\subsection{Birational rigidity}
It is interesting to have a good criterion for rigidity
in terms of discriminant curve and local invariants similar to \eqref{eq:local-inv}. A. Corti \cite{Corti2000} suggested that
a sufficient condition is $|3K_S+\Delta|\neq \emptyset$ (i.e. $\etr(S,\Delta)\ge 3$).
In particular, a standard conic bundle over $\PP^2$ should be birationally rigid 
if the discriminant has degree $\ge 9$. 
See \cite{BrownCortiZucconi-2004} for some supporting examples.

Similar to Theorem~\xref{Iskovskikh-theorem} one can prove the following 
\begin{sproposition}
\label{rigidity-theorem}
Let $\pi: X\to S$ be a standard conic bundle over a rational surface with discriminant curve
$\Delta\subset S$. 
Assume that $X$ is not birationally rigid. 
Then one of the following holds
\begin{enumerate}
\item 
$|3K_S+\Delta|= \emptyset$,

\item
$X$ is birational to a $\QQ$-conic bundle $\bar \pi: \bar X\to \bar S$ 
over a Du Val del Pezzo surface 
as in~\xref{DuVal:delPezzo}. If $\bar \Delta\subset \bar S$ is the discriminant curve,
then $-(4K_{\bar S}+\bar\Delta)$ is ample.
\end{enumerate}
\end{sproposition}

\subsection{Unirationality of conic bundles}
The unirationality is the most complicated and delicate 
property of algebraic varieties.
Of course one can show that a variety is unirational 
by explicit geometric construction. 
However, at the moment there is no techniques to prove 
\textit{non-unirationality} (in a non-trivial situation).

First we discuss the unirationality of surfaces over a non-closed field $\Bbbk$.
A necessary condition for $\Bbbk$-unirationality is the existence of a $\Bbbk$-point
\cite[Ex. 1.12]{Kollar2004}.
Any rational surface with a $\Bbbk$-point and with $K_X^2\ge 5$ is $\Bbbk$-rational 
(see Theorem~\xref{surfaces-rationality-criterion}).

\begin{sobservation}
\label{sobservation:unirationality}
Let $\pi: X\to S$ be a rational curve fibration.
Then $X$ is $\Bbbk$-unirational if and only if $\pi$ has a rational multi-section. 
\end{sobservation}

Indeed, if $\sigma: S \dashrightarrow X$ is a rational multi-section and $C:=\sigma(S)$, then
the base change $X\times_{S} C\to C$ is a rational curve fibration admitting a section.
Therefore, $X\times_{S} C$ is $\Bbbk$-rational and so $X$ is $\Bbbk$-unirational.
The converse is easy but there are some difficulties over finite fields (cf. \cite[Lemma 2.3]{Kollar2002a}).

\begin{corollary}
Let $\pi: X\to B$ be a standard conic bundle over a curve.
If $K_X^2\ge 3$ and $X$ has a $\Bbbk$-point, then 
the surface $X$ is either $\Bbbk$-rational or $\Bbbk$-unirational of degree $2$.
\end{corollary}

\begin{proof}[Sketch of the proof]
If $K_X^2= 3$, then by Theorem 
\xref{classification-Sarkisov-links-conic-bundle}\xref{classification-Sarkisov-links-conic-bundle-deg=3} 
on $X$ there exists a $(-1)$-curve defined 
over $\Bbbk$ which is a double section. 
Hence we can apply~\xref{sobservation:unirationality}.

Let $K_X^2=4$ and let $P\in X$ be a 
$\Bbbk$-point. Let $\sigma: \tilde X\to X$ be its blowup and let $E$ be the 
exceptional divisor. The point $P$ cannot lie on a $(-2)$-curve according to 
\xref{proposition-conic-bundle=del-Pezzo}\xref{case-conic-bundles-del-Pezzo-4}.In
this situation it is easy to show that the divisor $-K_{\tilde X}$ is nef and
$(-K_{\tilde X})^2=3$. Moreover, the linear system $|-K_{\tilde X}-E|$ is base point free
and defines a morphism $\tilde X\to\PP^1$ whose generic fiber is a rational curve. 
Here again $E$ is a double section.
\end{proof}

\begin{stheorem}[{\cite[Theorem~7]{Kollar-Mella-unirationality}}]
Let $\Bbbk$ be a field of characteristic $\neq 2$ and 
let $\pi : X \to B$ be a conic bundle over a curve with $K_X^2=1$. 
Then $X$ is $\Bbbk$-unirational. 
\end{stheorem}

The proof is based on the classification of surface conic bundles with
$K_X^2=1$ (see Proposition~\xref{proposition-conic-bundle=del-Pezzo}).
For example, if $-K_X$ is not nef then the curve $C$ 
from
\xref{proposition-conic-bundle=del-Pezzo}\xref{case-conic-bundles-del-Pezzo-1}
is a rational double section.
If $-K_X$ is nef, i.e. $X$ is a weak del Pezzo surface,
and the Bertini involution $\beta: X\to X$ does not preserve the conic bundle structure, 
then the image $\beta(F)$ of a fiber $F$ is a rational multisection.
Thus the only interesting case is the case where $-K_X$ is nef and 
the Bertini involution $\beta: X\to X$ acts fiberwise
(see \cite{Kollar-Mella-unirationality}).

\begin{scorollary}
\label{corollary:unirationality} 
Let $\Bbbk$ be a field of characteristic $\neq 2$ and 
let $\pi : X \to B$ be a conic bundle over a curve.
Assume that $K_X^2\ge 1$. Then $X$ is $\Bbbk$-unirational if and only if 
it has a $\Bbbk$-point. 
\end{scorollary}

On the other hand, Iskovskikh \cite{Iskovskih1967e} had shown that, 
in the case $\Bbbk=\RR$,
any rational surface over $\Bbbk$ with conic bundle structure 
is $\Bbbk$-unirational whenever it has a $\Bbbk$-point. 
In particular, this implies that there are $\Bbbk$-unirational
conic bundles whose discriminant locus is arbitrarily large
and so we cannot expect that a unirationality criterion can be formulated in the form
similar to Theorem~\xref{surfaces-rationality-criterion}. 
Unirationality of surface conic bundles over some ``large'' fields 
were discussed in \cite{Voronovich1986}, \cite{Yanchevskiui1985}, \cite{Yanchevskiui1992}.

Unirationality of del Pezzo surfaces of degree $\ge 3$ with a $\Bbbk$-point was proved in \cite{Manin-Cubic-forms-e-I}, \cite{Kollar2002a}.
It is expected that a del Pezzo surface of degree $2$ with a $\Bbbk$-point is 
always unirational. This was proved in many but not in all cases, see \cite{Manin-Cubic-forms-e-I},
\cite{Salgado-Testa-Varilly}, \cite{Festi2016}.
It is very little known about unirationality del Pezzo surfaces of degree $1$.

Manin observed that the invariant 
\begin{equation*}
H^1(\mathrm{Gal}(\bar\Bbbk/\Bbbk), \Pic(X\otimes \bar\Bbbk)) 
\end{equation*}
allows to bound the \textit{degree of unirationality} \cite[\S IV.7]{Manin-Cubic-forms-e-I}.

As an immediate consequence of~\xref{sobservation:unirationality}, one can see also that a cubic hypersurface in $\PP^4$ and an intersection of three
quadrics in $\PP^6$ are unirational (because they have structures of conic bundles with
rational multisections, see Examples~\xref{example-cubic-conic-bundle}
and~\xref{example-V8-conic-bundle}).
As a consequence of Corollary~\xref{corollary:unirationality} we obtain also a sufficient condition of unirationality of 
three-dimensional conic bundles over $\CC$:

\begin{scorollary}[cf. \cite{Mella2014}]
Let $\pi : X\to S$ be a standard conic bundle over a surface
and let $\Delta\subset S$ be the discriminant curve.
Assume that there is a 
base point free pencil of rational curves $\LLL$ on $S$ such that $\Delta \cdot \LLL\le 7$. 
Then $X$ is unirational. 
\end{scorollary}

\begin{scorollary}
Let $\pi : X\to \PP^2$ be a standard conic bundle 
and let $\Delta$ be the discriminant curve.
If $\deg \Delta\le 8$, then $X$ is unirational. 
\end{scorollary}

It is likely that a conic bundle with sufficiently large and \emph{general}
discriminant curve is not unirational.

\subsection{Stable rationality}

In a joint paper \cite{BCTSSD85} the authors 
constructed a counterexample to the birational Zariski cancellation problem.

\begin{stheorem}[\cite{BCTSSD85}]
Let $\Bbbk$ be a field of characteristic $\neq 2$, and 
$p(x)\in \Bbbk[x]$ an irreducible separable 
polynomial of degree $3$ with discriminant $a\in \Bbbk^*\setminus (\Bbbk^*)^2$. 
Then the surface $X$ given 
by the affine equation 
\begin{equation}
\label{equation:surface}
y^2-az^2=p(x)
\end{equation}
has the property that $X \times \PP^3$ is 
$\Bbbk$-rational but $X$ itself is not. 
\end{stheorem}
The stable rationality of the surface \eqref{equation:surface} was established by the technique of torsors.
The projection to $\mathbb A^1_x$ induces a conic bundle structure on a suitable projective model of $X$.
Then the non-rationality of $X$ follows from Theorem~\xref{surfaces-rationality-criterion}.
The result was improved by N. Shepherd-Barron
\cite{Shepherd-Barron2004}: one can replace $X \times \PP^3$
with $X \times \PP^2$.

Over $\CC$, an example a stable rational non-rational algebraic threefold $X_0$ 
can be given by the affine equation 
\begin{equation*}
y^2-a(t)z^2=p(t,x) 
\end{equation*}
for a suitable polynomial $p(t,x)$ and $a(t)=\operatorname{disc}_x p(t,x)$.
The variety $X_0$ is birationally equivalent to a standard conic bundle
$\pi: X\to S$ over a rational surface. The corresponding discriminant curve
$\Delta$ is a union of a trigonal curve $C$
of genus $g$ and smooth rational curves $L$, $F_1$, \dots, $F_{2g + 4}$ so that the 
singularities of $\Delta$ are the double points $\{P_i\} = C \cap F_i$, and $\{R_i\} = L \cap F_i$.
See \cite[\S 3]{BCTSSD85} for details.

On the other hand, it turns out that stably rational non-rational varieties 
are very rare.
In the paper \cite{HassettKreschTschinkel-2016} the authors applied 
specialization technique of Voisin \cite{Voisin2015}, developed by Totaro \cite{Totaro:stably}, Colliot-Th\'el\`ene and Pirutka \cite{ColliotThelene-Pirutka:cyclic-e},
\cite{ColliotThelene-Pirutka:quartic} 
to conic bundles:

\begin{stheorem}[\cite{HassettKreschTschinkel-2016}, see also \cite{Bohning-Bothmer}]
Let $S$ be a smooth projective rational surface. 
Let $\LLL$ be a linear system on $S$ whose general member is smooth and
irreducible. 
Let $\MMM$ be an irreducible component of the space of reduced nodal curves in $\LLL$ together with degree $2$
\'etale cover.
Assume that $\MMM$ contains a cover which is nontrivial over every irreducible component of a reducible curve with 
smooth irreducible components.
Then the standard conic bundle corresponding to a very general point of $\MMM$ is not stably rational.
\end{stheorem}
\begin{scorollary}
A very general conic bundle $\pi: X\to \PP^2$ with discriminant curve of degree 
$\ge 6$ is not stably rational.
\end{scorollary}
Higher-dimensional analog of this result see in \cite{Ahmadinezhad-Okada:conic}.
The stable rationality has an interesting analogue in the category of $G$-varieties \cite{Bogomolov-Prokhorov},
\cite{Prokhorov-stable-conjugacy-II}.

\subsection{Birationality to Calabi-Yau pairs}
We say that a pair $(Y, D)$ consisting of a projective variety $X$ and 
an effective $\RR$-divisor $D$
is a \textit{log Fano variety} if 
it has only Kawamata log terminal singularities and $-(K_X+D)$ is ample.
If some variety $X$ admits an effective $\RR$-divisor $D$ so that $(Y, D)$ is a log Fano,
we say that $X$ is of \textit{Fano type} \cite{Prokhorov-Shokurov-2009}. Usual Fano varieties and toric varieties 
are standard examples of varieties of Fano type.
It is known by that any $\QQ$-factorial variety of Fano type is a \emph{Mori dream 
space} \cite{BCHM}. It turns out that Fano type varieties 
satisfy many important properties: they are rationally connected, 
and this class is closed under 
the MMP, as well as, under arbitrary contractions \cite{Prokhorov-Shokurov-2009}.

We say that $(Y, D)$ is a \textit{log Calabi-Yau pair} if $K_X+D\qq 0$ and
$(Y, D)$ is log canonical, where $D$ is an effective $\RR$-divisor. 
We say that $Y$ is of \textit{Calabi-Yau type} if $(Y,D)$ is log Calabi-Yau for some $D$.
Finally, we say that $(Y, D)$ is a \textit{numerical log Calabi-Yau pair} if $K_X+D\equiv 0$, 
where $D$ is \textit{pseudo-effective} 
(no restrictions on the singularities of $Y$ are imposed). 
As above, one can define varieties of \textit{numerically Calabi-Yau type}.
By \cite{MiyaokaMori}, for any numerical log Calabi-Yau pair $(X,D)$ with $D\neq 0$,
the underlying variety $Y$ is uniruled.

It is known that log Fano varieties with bounded singularities form 
a bounded family \cite{Birkar2016}.
Thus the following question is natural.

\begin{squestion}
Under what conditions a conic bundle $\pi: X\to S$
is of Fano type (resp. Calabi-Yau type, numerically Calabi-Yau type)?
\end{squestion}

The following result of J. Koll\'ar shows that the condition for a variety even to be 
numerically Calabi-Yau type is very restrictive.

\begin{stheorem}[\cite{Kollar-Calabi-Yau-pairs}]
Let $\pi : X \to \PP^2$ be a standard conic bundle 
whose discriminant curve $\Delta \subset \PP^2$ has degree $\ge 19$. 
The following are equivalent: 
\begin{enumerate}
\item
$X$ is birational to a variety of numerically Calabi-Yau type. 
\item
There is a generically finite double section $D \subset X$ with normalization 
$\tau : \bar D \to D$ such that the branch curve of $\pi \comp \tau : \bar D \to \PP^2$ has degree $\le 6$.
\item
$X$ is birational to a standard conic bundle $\pi' : X' \to \PP^2$ 
\textup(with the same discriminant curve\textup) such that $|-K_{X'}|\neq 0$.
\end{enumerate}
\end{stheorem}
As a consequence one has.

\begin{stheorem}[\cite{Kollar-Calabi-Yau-pairs}]
Let $X_{d,2}\subset \PP^2\times \PP^2$ be a general hypersurface of bidegree 
$(d, 2)$. Then $X_{d,2}$ is not birational to a variety of numerically Calabi-Yau type for $d \ge 7$.
\end{stheorem}

There are analogs of these results for surfaces over non-closed fields.

Similar question for del Pezzo fibrations was studied by I. Krylov 
\cite{Krylov:Fano-type}. 

\subsection{Conic bundle structures on Fano threefolds}
The following question is natural:
\begin{squestion}
Which (smooth) Fano threefolds admit a birational conic bundle structure?
\end{squestion}
Recall \cite[\S~2.1]{IP99} that the \textit{Fano index} $\iota=\iota(X)$ of an $n$-dimensional Fano manifold $X$ is the maximal 
integer such that $-K_X=\iota A$, where $A$ is an (integral) divisor. 
The number $d=\dd(X):=A^{n}$ is called the \textit{degree} of $X$. 
In the case $\iota(X)=n-2$, the degree is even and so 
the number $g=\g(X):=\dd(X)/2+1$ is an integer. It is called the \textit{genus} of $X$.

Let us summarize the known facts about conic bundle structures on Fano threefolds
with $\uprho(X)=1$.
We use the classification and terminology of \cite[Table~\S~12.2]{IP99}.
\subsubsection*{Case $\iota>2$.}
If $\iota=4$, then $X\simeq \PP^3$; if $\iota=3$, then $X\simeq Q\subset \PP^4$
is a quadric. These varieties are rational and so they have a lot of birational conic bundle structures.

\subsubsection*{Case $\iota=2$, $d\ge 4$.} 
It is known that $d\le 5$. For $d=5$ and $4$ the Fano threefolds are rational.

\subsubsection*{Case $\iota=2$, $d=3$.} 
Then $X\simeq X_3\subset \PP^4$ is a cubic threefold.
It has a lot of conic bundle structures (see Example~\xref{example-cubic-conic-bundle}).

\subsubsection*{Case $\iota=2$, $d=2$.} 
Then $X=X_2$ can be represented as double cover $X\to \PP^3$
branched over a quartic (it is so-called \textit{quartic double solid}).
The existence of conic bundle structure on smooth varieties of this type in not known.
However, if $X$ has at least one ordinary double point, say $P$, then the ``projection'' 
from $P$ gives a conic bundle, see Examples~\xref{example-4-double-solid} and~\xref{example-Artin-Mumford}.

Rationality questions for singular quartic double solids were studied in
\cite{Clemens1983}, \cite{Voisin1988}, \cite{Varley1986}, \cite{Debarre1990}, \cite{Przhiyalkovskij-Cheltsov-Shramov-2015}, \cite{Voisin2015}.

\subsubsection*{Case $\iota=2$, $d=1$.} 
Then $X=X_1$ is so-called \textit{double Veronese cone}.
The birational geometry of this variety is very rich.
It was studied in the series of papers 
of M. Grinenko (see \cite{Grinenko2003}, \cite{Grinenko2004}). 
In particular, he proved that there is no conic bundle structures 
on $X_1$. 

\par \medskip
If $\iota=1$, then $g\in \{2,\, 3,\dots,10,\, 12\}$.

\subsubsection*{Case $\iota=1$, $g\in \{7,\, 9,\, 10,\, 12\}$.}
Then the variety $X$ is rational.

\subsubsection*{Case $\iota=1$, $g=8$.}
The double projection from a line gives a structure of a conic bundle structure 
\cite[Theorem~4.3.3]{IP99}. Conic bundle structures exist 
also because $X$ is birational to a cubic threefold \cite{Iskovskikh1980}, 
\cite{Tregub1985a}, \cite[Theorem~4.5.8]{IP99} (see 
Example~\xref{example-cubic-conic-bundle}).

\subsubsection*{Case $\iota=1$, $g=6$.} 
The existence of conic bundle structures is unknown
for smooth varieties of this type. However, singular Fano threefolds with $\iota=1$, $g=6$
may have conic bundle structures. For example, if such a variety $X$ is $\QQ$-factorial and has a unique node, then 
is can be birationally transformed to a standard conic bundle over $\PP^2$ 
with discriminant curve a smooth sextic $\Delta\subset \PP^2$ \cite{Debarre-Iliev-Manivel-2011},
\cite{Prokhorov-factorial-Fano-e}.
If $X$ is not $\QQ$-factorial and has a unique node, then it can be transformed to a
smooth quartic double solid (see above) by \cite[Proposition 5.2]{Debarre2012},
\cite[Example 1.11]{Przhiyalkovskij-Cheltsov-Shramov-2005en}, or \cite[\S 4.4.1]{BrownCortiZucconi-2004}.

\subsubsection*{Case $\iota=1$, $g=5$.} 
Then $X\simeq X_{2\cdot 2\cdot 2}\subset \PP^6$
is an intersection of three quadrics. Conic bundle structures exist by 
Example~\xref{example-V8-conic-bundle}.

\subsubsection*{Case $\iota=1$, $g=4$.} 
Then $X\simeq X_{2\cdot 3}\subset \PP^5$
is an intersection of a quadric and a cubic.
A general member of the family is birationally rigid and has no conic bundle structures
\cite[Ch.~3, \S~2]{Iskovskikh-Pukhlikov-1996}, \cite[Ch.~2, \S\S~5-6]{Pukhlikovbook2013}.
However, this is not known for \textit{any} smooth variety of this type. 

\subsubsection*{Case $\iota=1$, $g=3$.} 
Then $X$ is either a quartic in $\PP^4$ or a 
double cover a quadric branched over a surface of degree $8$ ($X$ is said to be a \textit{double quadric}). 
In both cases $X$ is birationally rigid and has no conic bundle structures (see \cite{Iskovskih-Manin-1971b},
\cite{Iskovskikh1980}).

\subsubsection*{Case $\iota=1$, $g=2$.}
Then $X$ is a so-called \textit{sextic double solid}.
It is birationally superrigid and so it has no conic bundle structures \cite{Iskovskikh1980}.
\par\medskip
More examples of singular Fano threefolds with birational conic bundle structure can be found in 
\cite[\S\S~7.2, 7.5-7.7]{Jahnke-Peternell-Radloff-II}.
In contrast with the above considered case $\uprho(X)=1$, many Fano threefolds with 
$\uprho(X)>1$ have (even biregular) conic bundle structures \cite{Mori-Mukai-1981-82}, see 
also \cite{Kuznetsov2015} for higher-dimensional examples. 

Note that Koll\'ar's method of reduction to positive characteristic \cite{Kollar1995}
allows to show that certain Fano hypersurfaces of high dimension do not admit a birational conic bundle
structures.

\subsection{Del Pezzo fibrations}
It is desirable to construct a good theory of three-dimensional del Pezzo fibrations.
Some generalization of Sarkisov's theorem on standard models was obtained by A. Corti 
\cite{Corti1996} (see also \cite{Kollar1997a}, \cite{Loginov2017}). For del Pezzo fibrations of degree $\ge 3$
he proved that there exist a Gorenstein terminal model such that its fibers are reduced and irreducible.
For del Pezzo fibrations of degree $2$ and $1$ the situation is more complicated: 
it turns out that a Gorenstein terminal model does not always exist, one should 
consider models with non-Gorenstein points of low indices. 

A del Pezzo fibration of degree $\ge 5$ over $\PP^1$ is always rational (cf. Theorem~\xref{surfaces-rationality-criterion}). 
Any del Pezzo fibration $\pi: X\to B$ of degree $4$ birationally has a conic bundle structure
(cf. Theorem~\xref{classification-Sarkisov-links-conic-bundle}~\xref{classification-Sarkisov-links-conic-bundle-deg=3}).
In the case of smooth $X$, the rationality criterion was obtained by Alexeev \cite{Alexeev-898133}
(see also \cite{Shramov2006}). The case of del Pezzo fibrations of degree $\le 3$ is much more complicated.
For results on
birational rigidity we refer to \cite{Iskovskikh1995}, \cite{Pukhlikov1998a}, \cite[Ch.~4, \S~1]{Pukhlikovbook2013}, \cite{Cheltsov2005c},
\cite{Sobolev2002}, \cite{Grinenko2000}, \cite{Shokurov-Choi-2011}, \cite{Krylov2016}.
A higher-dimensional generalization was discussed in a recent paper \cite{Pukhl17e}.
It would be interesting to reformulate these results in terms of 
degeneracy loci and local invariants. 
There are also results on \emph{local}
birational rigidity \cite{Park2001}, \cite{Park2003}, \cite{Pukhlikov2000c}.

\subsection{Extremal contractions of relative dimension one}
Consider a Mori extremal contraction $\pi: X\to S$ from a smooth
$n$-dimensional variety such that the generic fiber is of dimension $1$.
If $n\le 3$, then $S$ is smooth and $\pi$ is a standard conic bundle 
\cite{Mori-1982}. The same holds in arbitrary dimension under the additional assumption 
that all the fibers are one-dimensional \cite{Ando1985}.
However, if we relax the equidimensionality condition, then the situation becomes more complicated
even for $n=4$. 

\begin{proposition}\label{contractions-smooth}
Let $\pi: X\to S$ be a Mori extremal contraction with $0<\dim S<\dim X$.
Assume that $X$ is smooth.
Then the following holds. 
\begin{enumerate}
\item\label{contractions-smooth-1}
$S$ has at worst locally factorial canonical singularities.

\item \label{contractions-smooth-2}
For any point $P\in S$, its local algebraic fundamental group is trivial.
\item \label{contractions-smooth-1-2}
$\codim \Sing(S)\ge 3$. In particular, $S$ is smooth if $\dim S\le 2$.

\item \label{contractions-smooth-3}
\textup(cf. \cite[Proposition 1.3]{Andreatta-Wisniewski-view1997}\textup)
If $S$ has at worst quotient singularities, then it is smooth.
\item\label{contractions-smooth-4}
If $\pi$ is equidimensional, then $S$ is smooth.
\end{enumerate}
\end{proposition}

\begin{proof}
\xref{contractions-smooth-1}
Recall that in general the singularities of $S$ are rational \cite[Corollary 7.4]{Kollar-1986-I}. 
The factoriality of $S$ is proved in the same style as \cite[Lemma 5-1-5]{KMM}
or \cite[Corollary 3.18]{Kollar-Mori-1988}. For convenience of the reader we 
reproduce these arguments.
Let $B$ be a prime Weil divisor on $S$ and let $D$ be the divisorial part of 
$\pi^{-1}(B)$. 
Then $D$ is a Cartier divisor on $X$ and $\pi(D)=B$. Since $D$ does not meet a general fiber,
$D=\pi^*B_0$ for some Cartier divisor on $S$ \cite[Corollary 3.17]{Kollar-Mori-1988}.
Clearly, $B_0\subset B$ and so $B_0=B$.
Thus, the singularities of $S$ are rational and $K_S$ is a Cartier divisor. 
In this case the singularities must be canonical (see \cite[Corollary 5.24]{Kollar-Mori-1988}).

\xref{contractions-smooth-2}
Regard $S$ as a sufficiently small neighborhood of $o$.
Assume that there is a Galois cover $(S'\ni o')\to (S\ni o)$ which is \'etale over 
$S\setminus \Sing(S)$.
Thus $(S\ni o)=(S'\ni o')/G$ and $G$ is a finite group acting on $S'$ 
freely in codimension one.
Consider the normalization $X'$ of the fiber product of $X\times_SS'$. 
Then $f: X' \to X$ is a finite Galois morphism which is \'etale in codimension one.
Thus, by purity of the branch locus,
$f$ is \'etale and $X'$ is smooth.
The projection 
$\pi': X'\to S'$ is a contraction such that $-K_{X'}$ is relatively ample.
By \cite[Lemma 3.4]{ProkhorovShramov-RC} there exists an irreducible $G$-invariant rationally connected 
subvariety $V\subset {\pi'}^{-1}(o')$. On the other hand, an action of a cyclic subgroup $G_1\subset G$
on a rationally connected variety always has a fixed point.
Hence, the action of $G$ on $V$ and $X'$ cannot be free and so $G=\{1\}$.

Assertions~\xref{contractions-smooth-1-2} and~\xref{contractions-smooth-3}
easily follow from~\xref{contractions-smooth-1} and~\xref{contractions-smooth-2}.
Finally, for~\xref{contractions-smooth-4}, we just have to note that in this case
a general section $H$ of $X$ by 
subspace of codimension $\dim X-\dim S$ is smooth and the restriction 
$\pi_H: H\to S$ is finite, hence $S$ has at worst quotient singularities. 
\end{proof} 

Kachi \cite{Kachi1997} classified 
contractions from smooth fourfold to threefolds.
In this case the base variety is not necessarily smooth.

\subsection{Equivariant conic bundles}
It would be useful to develop the theory of conic bundles in 
the category of $G$-varieties \cite{Manin-1967}, that is, for varieties over non-closed fields, as well as, for varieties with
a group action. The first step was done by A. Avilov \cite{Avilov2014}.
He proved an analog of Theorem~\xref{theorem-standard-models}.
We hope that Theorems~\xref{Sarkisov-theorem},~\xref{Chelcprimetsov-theorem},
and Conjectures~\xref{conjecture-Iskovskikh} and~\xref{conjecture-Shokurov}
should be generalized for the equivariant case. 
This is important for applications to the study of Cremona groups and rationality problems,
see \cite{Prokhorov-Shramov-J-const}, \cite{Prokhorov-Shramov-3folds}, \cite{Prokhorov-Shramov-p-groups},
\cite{Popov-diffeomorphism}, \cite{Yasinsky2016}, \cite{Yasinsky-J-const}, \cite{Popov-Borel-Cremona},
\cite{PrzyjalkowskiShramov2016}, \cite{Trepalin2014}, \cite{Trepalin2016}.

 \newcommand{\etalchar}[1]{$^{#1}$}
\def\cprime{$'$}


\end{document}